\documentclass[10pt,twoside,a4paper,english,leqno]{article}

\input{a-packages.sty}
\input{a-style.sty}
\input{a-commandes.sty}

\usepackage[left,pagewise,modulo,mathlines]{lineno}
\linenumbers

\newcommand*\patchAmsMathEnvironmentForLineno[1]{%
  \expandafter\let\csname old#1\expandafter\endcsname\csname #1\endcsname
  \expandafter\let\csname oldend#1\expandafter\endcsname\csname end#1\endcsname
  \renewenvironment{#1}%
     {\linenomath\csname old#1\endcsname}%
     {\csname oldend#1\endcsname\endlinenomath}}%
\newcommand*\patchBothAmsMathEnvironmentsForLineno[1]{%
  \patchAmsMathEnvironmentForLineno{#1}%
  \patchAmsMathEnvironmentForLineno{#1*}}%
\AtBeginDocument{%
\patchBothAmsMathEnvironmentsForLineno{equation}%
\patchBothAmsMathEnvironmentsForLineno{align}%
\patchBothAmsMathEnvironmentsForLineno{flalign}%
\patchBothAmsMathEnvironmentsForLineno{alignat}%
\patchBothAmsMathEnvironmentsForLineno{gather}%
\patchBothAmsMathEnvironmentsForLineno{multline}%
}
\usepackage[textwidth=40pt,textsize=scriptsize,dvistyle]{todonotes} 
\title{Asymptotic models for the generation of internal waves by a moving ship, and the dead-water phenomenon}

\author{Vincent Duch\^ene\thanks{D\'epartement de Math\'ematiques et Applications, UMR 8553, \'Ecole normale sup\'erieure, 45 rue d'Ulm, F 75230
Paris cedex 05, France; e-mail: duchene@dma.ens.fr}}
\date{\today}

\begin{document}
\maketitle

\begin{abstract}
 This paper deals with the dead-water phenomenon, which occurs when a ship sails in a stratified fluid, and experiences an important drag due to waves below the surface. More generally, we study the generation of internal waves by a disturbance moving at constant speed on top of two layers of fluids of different densities. Starting from the full Euler equations, we present several nonlinear asymptotic models, in the long wave regime. These models are rigorously justified by consistency or convergence results. A careful theoretical and numerical analysis is then provided, in order to predict the behavior of the flow and in which situations the dead-water effect appears.
\end{abstract}


\section{Introduction}
The so-called ``dead-water'' phenomenon has been first reported by Fridtjof Nansen~\cite{Nansen00}, describing a severe (and then inexplicable) decrease of velocity encountered by a ship, sailing on calm seas. Bjerknes, and then Ekman~\cite{Ekman04} soon explained that this phenomenon occurs as the ship, on top of density-stratified fluids (due to variations in temperature or salinity), concedes a large amount of energy by generating internal waves. Our work is motivated by the recent paper of Vasseur, Mercier and Dauxois~\cite{VasseurMercierDauxois11}, who performed experiments that revealed new insights on the phenomenon. 

Relatively little consideration has been given to this problem, after the early works of Lamb~\cite{Lamb16} and Sretenskii~\cite{Sretenskii34}. Miloh, Tulin and Zilman in~\cite{MilohTulinZilman93} produce a model and numerical simulations for the case of a semi-submersible moving steadily on the free surface. The authors assume that the density difference between the two layers is small, an assumption that is removed by Nguyen and Yeung~\cite{NguyenYeung97}. Motygin and Kuznetsov~\cite{MotyginKuznetsov97} offer a rigorous mathematical treatment of the issue when the body is totally submerged in one of the two layers, and Ten and Kashiwagi~\cite{TenKashiwagi04} present a comparison between theory and experiments, in the case of a body intersecting the interface as well as the surface. Finally, we would like to cite Lu and Chen~\cite{LuChen09} for a more general treatment of the problem. All of these works use linearized boundary conditions, that rely on the assumption that the amplitudes of the generated waves are small. The linear theory is convenient as it allows to obtain the flow field by a simple superposition of Green functions, replacing the moving body by a sum of singularities. The integral representation of the flow, as well as the wave resistance experienced by the body, are therefore found explicitly. However, the smallness assumption on the wave amplitudes, that is necessary to the linear theory, is quite restrictive, and the experiences in~\cite{VasseurMercierDauxois11} clearly exhibit nonlinear features. Our aim is to produce simple nonlinear models, that are able to predict the apparition and the magnitude of the dead-water effect, and recover duly most of its key aspects.

\medskip

In this paper, we introduce several different asymptotic models, corresponding to two different regimes (size of the parameters of the system). Each time, we assume that the depth of the fluid layers are small when compared with the wave lengths (shallow water). In the one-layer case, the equations obtained with this single assumption are the shallow-water or Saint Venant equations (at first order), or the Green-Naghdi extension (one order further, therefore involving nonlinear dispersion terms). See Wu, Chen~\cite{WuChen03} and references therein for a numerical treatment of the waves generated by a moving ship in the one layer case, using the Green-Naghdi equations. The two regimes we study carry additional smallness assumptions, that allow to substantially simplify the models. The first regime considers the case of small surface deformations, and small differences between the densities of the two layers. Such additional assumptions are very natural in the oceanographic framework, and have been frequently used in the literature (see~\cite{Keulegan53,Long65,MilewskiTabakTurnerEtAl04} for example). In the second regime, we assume that the magnitude of the produced internal waves, when compared with the depth of the two layers, are small and of the same order of magnitude as the shallowness parameter. This regime, known as the Boussinesq regime, is particularly interesting as it allows models with competing dispersion and nonlinearity. Along with the coupled Boussinesq-type model, we introduce the KdV approximation, which consists in a decomposition of the flow into two waves, each one being a solution of an independent forced Korteweg-de Vries equation (fKdV). The fKdV equation has been extensively studied in the framework of the one-layer water wave problem (where a moving topography, or pressure, is the forcing term that generates waves); see~\cite{Wu87,MaleewongGrimshawAsavanant05,LeeYatesWu89}, for example.

The system we study consists in two layers of immiscible, homogeneous, ideal, incompressible fluids only under the influence of gravity. The bottom is assumed to be flat, and we use the rigid-lid approximation at the surface. However, the surface is {\em not flat}, but a given function, moving at constant speed, that reflects the presence of the ship. Moreover, as we are interested in unidirectional equations (the fKdV equations), we focus on the two-dimensional case, {\em i.e.} the horizontal dimension $d=1$. However, our method could easily be extended to the three-dimensional configuration. Starting from the governing equations of our problem, the so-called full Euler system, and armed with the analysis of~\cite{Duchene10,Duchene} in the free surface case, we are able to deduce asymptotic models for each of the regimes presented above.
Each of the models presented here are justified by a consistency result, or a convergence theorem (in a sense precisely given in Section~\ref{sec:regimes} page~\pageref{def:consistency}). We compute numerically the fully nonlinear system of the first regime, as well as the KdV approximation, which allows us to investigate the effect of different parameters of the system, such as the velocity of the boat, or the depth ratio of the two layers. The wave resistance encountered by the ship is also computed, so that we are able to predict in which situations the dead-water effect shows up. 

\medskip

Our models allow to recover most of the known aspects of the dead-water phenomenon (see~\cite{VasseurMercierDauxois11,MilohTulinZilman93,NguyenYeung97}), especially:
\begin{enumerate}
 \item transverse internal waves are generated at the rear of a body while moving at the surface;
 \item the body suffers from a positive drag when an internal elevation wave is located at its stern;
 \item this effect can be strong (as the generated wave reach large amplitudes) near critical Froude numbers, that is when the velocity of the body approaches maximum internal wave speed. The effect is always small away from the critical value;
 \item the maximum peak of the drag is reached at slightly subcritical values.
\end{enumerate}
We would like to emphasize that these considerations are not new. In particular, the fact that the drag experienced by the ship is strong only when its velocity approaches the maximum internal wave speed has been observed through experiments~\cite{Ekman04,VasseurMercierDauxois11}, and the maximum peak of the drag being reached at slightly subcritical values is predicted by linear models~\cite{MilohTulinZilman93,NguyenYeung97}. The main contribution of our work is to offer simple adapted models, that are rigorously justified for a wide range of parameters (in particular, in contrast with linear models, our models allow large amplitude waves, in agreement with experiments).

Our models allow to easily investigate the dead-water phenomenon for several values of the parameter. For example, we numerically compute the system for several values of the 
depth ratio between the two layers. It appears that, in disagreement with the intuition, the dead-water phenomenon is stronger when the upper layer is thicker than the lower one. This observation is new, as far as we know. This curiosity is due to the fact that the magnitude of the drag suffered by the body does not depend on the distance between the body and the generated internal wave, but rather on the amplitude and sharpness of the latter. A thicker upper layer allows the generation of internal elevation waves reaching very large amplitudes, when compared with the size of the body, which strengthen the dead-water phenomenon.

The behavior of the system, depending on the depth ratio of the two layers of fluid and the normalized velocity of the body, is summarized in Table~\ref{tab:Resume}, below.

\begin{sidewaystable}

\begin{center}
\setlength{\extrarowheight}{2pt}\renewcommand{\thefootnote}{\thempfootnote}
\begin{tabular}{|m{1.8cm}||m{4cm}|m{4cm}|m{4cm}|m{4cm}|}
\hline Velocity of the ship & \multicolumn{1}{c|}{subcritical case} & \multicolumn{2}{c|}{critical case}  & \multicolumn{1}{c|}{supercritical case} \\ \hline
 Depth ratio between the layers &  & \multicolumn{1}{c|}{thicker lower layer} &\multicolumn{1}{c|}{thicker upper layer} &  \\ \hline \hline
  Regime 1 &\rr The generated waves are small, and conclusions of Regime 2 apply. &\rr Generation of dull elevation-depression wave below the body.\linebreak \linebreak Moderate wave resistance.\linebreak \linebreak See Figure~\ref{fig:FNLd1}, page~\pageref{fig:FNLd1} &\rr Generation of a sharp elevation wave below the body.\linebreak \linebreak Strong wave resistance. \linebreak \linebreak See Figure~\ref{fig:FNLd2}, page~\pageref{fig:FNLd2}& \rr The generated waves are small, and conclusions of Regime 2 apply. \tn \hline
 Regime 2 &\rr Continuous generation of small up-stream propagating and down-stream propagating waves.\linebreak \linebreak Very weak wave resistance. \linebreak \linebreak See Figure~\ref{fig:subcritical}, page~\pageref{fig:subcritical} &\rr Periodic generation of up-stream propagating depression waves with very large time-period.\footnote{As discussed in Section~\ref{sec:critical}, the time period of the generation of the waves, and therefore of the oscillation of the wave resistance, may be larger than the time-range of validity of the model.} \linebreak \linebreak  Oscillating wave resistance, with positive mean.\linebreak \linebreak See Figure~\ref{fig:critical}, page~\pageref{fig:critical}  &\rr Periodic generation of up-stream propagating elevation waves with very large time-period.\footnotemark[\value{mpfootnote}] \linebreak \linebreak Oscillating wave resistance, with mean zero.\linebreak \linebreak See Figure~\ref{fig:critical}, page~\pageref{fig:critical}  &\rr Generation of small down-stream propagating waves. Convergence towards a steady state. \linebreak \linebreak No lasting wave resistance. \linebreak \linebreak See Figure~\ref{fig:supercritical}, page~\pageref{fig:supercritical}\tn \hline                                                                                                                                                                                                                        \end{tabular}
                                                                                                                                                                                                                                      \end{center}

\caption{Behavior of the flow in the two studied regimes, depending on the velocity of the body and the depth ratio of the fluids}
\label{tab:Resume}
\end{sidewaystable}

However, one peculiar phenomenon, that is described in details in~\cite{VasseurMercierDauxois11}, is not recovered by our models. Indeed, the dead-water effects exhibits a somewhat periodic behavior, where during each period, a wave is generated, slows down the ship, and then breaks. 
This discrepancy between our simulations and the experiments is due to the fact that our models are based on the assumption of a constant velocity for the traveling body, while their experiments are conducted with a constant force brought to the body~\cite{Vasseur08,VasseurMercierDauxois11}. The latter setting is of course more natural, but the constant velocity hypothesis is crucial in our analysis (and is, in fact, consistently assumed within all existing models in the literature~\cite{Ekman04,MilohTulinZilman93,YeungNguyen99}). We perform a numeric experiment which roughly matches this setting (the velocity is then adjusted at each time step as a function of the drag suffered by the ship), and recover the oscillating behavior. This point is discussed in more details on page~\pageref{Hyst}.

\bigskip

\para{Outline of the paper} In Section~\ref{sec:basicequations}, we introduce the governing equations of our problem: the so-called full Euler system. A specific care is given to an adapted rescaling, that leads to the dimensionless version of the full Euler system~\eqref{eqn:AdimEulerComplet}. The models presented here are asymptotic approximations of this system. In Section~\ref{sec:regimes}, we present precisely the different means of justification of our models, as well as the regimes at stake. The main tool that we use in order to construct our models, as well as a broad, strongly nonlinear model~\eqref{eqn:Sbegin} are presented in Section~\ref{sec:expansion}. 

The simpler models we study are deduced from~\eqref{eqn:Sbegin}, using the additional assumptions of the regimes considered. Strongly nonlinear models~\eqref{eqn:D1} and~\eqref{eqn:Dfinal} are introduced in Section~\ref{sec:stronglynonlinear}, and justified with consistency results. Numerical simulations are then displayed and discussed in Section~\ref{sec:numericsFNL}. The weakly nonlinear models, {\em i.e.} the Boussinesq-type system~\eqref{eqn:Bouss} (and its symmetrized version~\eqref{eqn:SymBouss}), and the KdV approximation~\eqref{eqn:KdV1}, are presented in Sections~\ref{sec:Bouss} and~\ref{sec:KdV}, respectively. The convergence of their solutions towards the solutions of the full Euler system are given in Propositions~\ref{prop:ConvBouss} and~\ref{PROP:CONVKDV}, respectively. An in-depth analysis of the forced Korteweg-de Vries equation, and its consequences to the dead-water effect, is then displayed in Section~\ref{sec:KdVana}. 

Finally, in Section~\ref{sec:conclusion}, we conclude with an overview of the results presented here, together with suggestions for future work.

Calculations that lead to~\eqref{eqn:Sbegin} are postponed until Appendix~\ref{sec:fullyscheme}, and Appendix~\ref{sec:proof} contains the proof of Proposition~\ref{PROP:CONVKDV}. Finally, Appendix~\ref{sec:deadwater} is devoted to the analysis of the wave resistance encountered by the ship, 
and the numerical schemes used in the simulations are presented and justified in Appendix~\ref{sec:numericalschemes}.

\bigskip

\para{Notations}
We denote by $C(\lambda_1,\lambda_2,\ldots)$ any positive constant, depending on the parameters ${\lambda_1,\lambda_2,\ldots }$, and whose dependence on $\lambda_j$ is always assumed to be nondecreasing.

Let $f(x_1,\dots,x_d)$ be a function defined on $\RR^d$. We denote by $\partial_{x_i} f$ the derivative with respect to the variable $x_i$. If $f$ depends only on $x_i$, then we use the notation 
 \[ \frac{\dd}{\dd x_i}f \ \equiv \ \partial_{x_i} f .\]
 
We denote by $L^2=L^2(\RR)$ the Hilbert space of Lebesgue measurable functions ${F:\RR\to\RR^n}$ with the standard norm ${\big\vert F\big\vert_2 \ =\ \left(\ \int_\RR |F(x)|^2\ \dd x\ \right)^{1/2}\ <\ \infty}$. Its inner product is denoted by ${\big(\ F_1\ ,\ F_2\ \big)\ =\ \int_\RR F_1 \cdot F_2}$.

Then, we denote by $H^s=H^s(\RR)$ the $L^2$-based Sobolev spaces. Their norm is written ${\big\vert\cdot\big\vert_{H^s}\ =\ \big\vert\Lambda^s\cdot\big\vert_{2}}$, with the pseudo-differential operator ${\Lambda\equiv(1-\partial_x^2)^{1/2}}$. It is also convenient to introduce the following norm on the Sobolev space $H^{s+1}$ (for $\varepsilon>0$ a small parameter):
 \[\big| U\big|_{H^{s+1}_\varepsilon}^2=\big| U\big|_{H^{s}}^2+\varepsilon \big| U\big|_{H^{s+1}}^2.\]

Let $0 < T \leq \infty$ and $f(t,x)$ be a function defined on $[0,T]\times\RR$. Then one has $f\in L^\infty([0,T];H^s)$ if $f$ is bounded in $H^s$, uniformly with respect to $t\in [0,T]$. Moreover, $f\in W^{1,\infty}([0,T];H^s)$ if $f, \ \partial_t f \in L^\infty([0,T];H^s)$. Their respective norm is written $\big\vert\cdot\big\vert_{L^\infty H^s}$ and $\big\vert\cdot\big\vert_{W^{1,\infty}H^s}$.

\section{Construction of asymptotic models}
Our framework, the formulation of the system using Dirichlet-Neumann operators, as well as the techniques used for constructing the asymptotic models, are not new. We follow here the strategy initiated in~\cite{BonaChenSaut02,BonaChenSaut04,BonaColinLannes05} in the one-layer (water wave) configuration, and extended to the bi-fluidic case in~\cite{BonaLannesSaut08,Duchene10,Duchene}. The major difference here is that the surface is not free, but rather a non-flat rigid lid, moving at constant velocity. 

We present in the following, as briefly as possible, the governing equations of our problem. The full details of the construction can be found in the references above.
Our models are then derived starting from the non-dimensionalized version of these equations (the full Euler system~\eqref{eqn:AdimEulerComplet}; see Section~\ref{sec:nondimensionalization}). The idea is to use smallness assumptions on some parameters, that characterize natural quantities of the system, to construct the asymptotic models. Together with the means to justify our models, the regimes under study are precisely described in Section~\ref{sec:regimes}. Finally, the main tool that we use in order to construct our models, that is the asymptotic expansion of the Dirichlet-Neumann operators in terms of small shallowness parameter $\mu$, is presented in Section~\ref{sec:expansion}. A broad, nonlinear model~\eqref{eqn:Sbegin} is then deduced, and the precise systems we study are derived from it, using the additional assumptions of Regime~\ref{regimeRL} and Regime~\ref{regimeSA}, respectively in Section~\ref{sec:stronglynonlinear} and Section~\ref{sec:weaklynonlinear}.

\subsection{The basic equations}\label{sec:basicequations}
\begin{figure}[htb]
\centering
\psfrag{z1}[Bc][Br]{\begin{footnotesize}$\zeta_1$                          \end{footnotesize}}
\psfrag{z2}[Bl][Bl]{\begin{footnotesize}$\zeta_2(t,X)$                          \end{footnotesize}}
\psfrag{g}{\begin{footnotesize}$\mathbf{g}$                          \end{footnotesize}}
\psfrag{h1}[Bc][Bl]{\begin{footnotesize}$d_1$                   \end{footnotesize}}
\psfrag{h2}[Bc][Bl]{\begin{footnotesize}$-d_2$                   \end{footnotesize}}
\psfrag{a1}[Bc][Br]{\begin{footnotesize}$a_1$                   \end{footnotesize}}
\psfrag{a2}[Bc][Br]{\begin{footnotesize}$a_2$                   \end{footnotesize}}
\psfrag{0}[Bc][Bl]{\begin{footnotesize}$0$                   \end{footnotesize}}
\psfrag{l}{\begin{footnotesize}$\lambda$                   \end{footnotesize}}
\psfrag{z}{\begin{footnotesize}$z$                   \end{footnotesize}}
\psfrag{X}{\begin{footnotesize}$X$                   \end{footnotesize}}
\psfrag{O1}{\begin{footnotesize}$\Omega^t_1$                   \end{footnotesize}}
\psfrag{O2}{\begin{footnotesize}$\Omega^t_2$                   \end{footnotesize}}
\psfrag{n1}{\begin{footnotesize}$\mathbf{n}_1$                          \end{footnotesize}}
\psfrag{n2}{\begin{footnotesize}$\mathbf n_2$                          \end{footnotesize}}
 \includegraphics[width=1\textwidth]{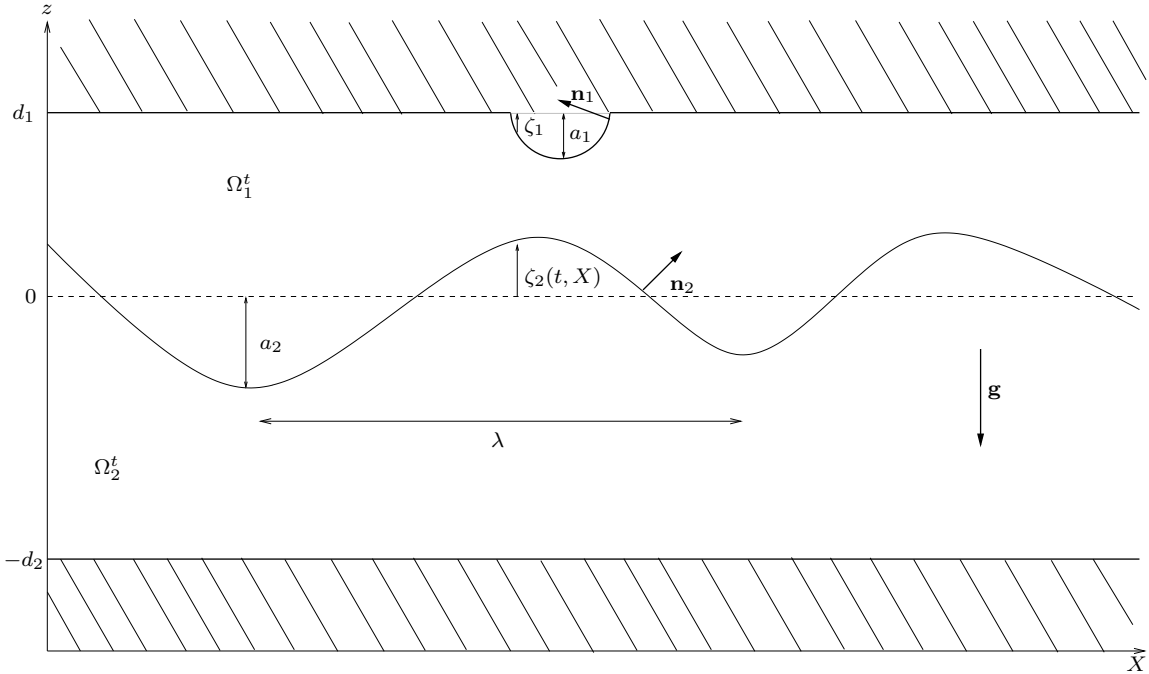}
 \caption{Sketch of the domain}
\label{fig:SketchOfDomainDW}
\end{figure}

The system we study consists in two layers of immiscible, homogeneous, ideal, incompressible fluids only under the influence of gravity (see Figure~\ref{fig:SketchOfDomainDW}). Since we are interested in unidirectional models (the KdV approximation), we focus on the two-dimensional case, {\em i.e.} the horizontal dimension $d=1$. However, most of this study could easily be extended to the case $d=2$, using the techniques presented here.

We denote by $\rho_1$ and $\rho_2$ the density of, respectively, the upper and  the lower fluid.
Since we assume that each fluid is incompressible, $\rho_1$ and $\rho_2$ are constant and the velocity potentials $\phi_i$ $(i=1,2)$,
respectively associated to the upper and lower fluid layers, satisfy the Laplace equation
\[\partial_x^2 \phi_i \ + \ \partial_z^2 \phi_i \ = \ 0.\] 
Moreover, it is presumed that the surface and the interface are given as the
graph of functions (respectively, $\zeta_1(t,x)$ and $\zeta_2(t,x)$) which express the deviation from their rest
position (respectively, $(x,d_1)$ and $(x,0)$) at the spatial
coordinate $x$ and at time $t$. The bottom is assumed to be flat, and the surface is flat away from the location of the ship moving at constant speed $c_s$, so that $\zeta_1$ matches the submerged part of the ship and is given by 
\[\zeta_1(t,x) \ = \ \zeta_1(x-c_st).\]
Therefore, at each time $t\ge 0$, the domains of the upper and lower fluid (denoted, respectively, $\Omega_1^t$ and $\Omega_2^t$), are given by
\begin{align*}
 \Omega_1^t \ &= \ \{\ (x,z)\in\RR^{d}\times\RR, \quad \zeta_2(t,x)\ \leq\ z\ \leq \ d_1+\zeta_1(x-c_s t)\ \}, \\
 \Omega_2^t \ &= \ \{\ (x,z)\in\RR^{d}\times\RR, \quad -d_2 \  \leq\ z\ \leq \ \zeta_2(t,x)\ \}.
\end{align*}
 The fluids being ideal, they satisfy the Euler equation (or Bernoulli equation when written in terms of the velocity potentials), and kinematic boundary conditions are given through the assumption that no fluid particle crosses the surface, the bottom or the interface. Finally, the set of equations is closed by the continuity
of the stress tensor at the interface, which takes into account the effects of surface tension.

Altogether, the governing equations of our problem are given by the following
  \begin{equation}  \label{eqn:EulerComplet}
\left\{\begin{array}{ll}
         \partial_x^2 \phi_i \ + \ \partial_z^2 \phi_i \ = \ 0 & \mbox{ in }\Omega^t_i, \ i=1,2,\\
         \partial_t \phi_i+\frac{1}{2} |\nabla_{x,z} \phi_i|^2=-\frac{P}{\rho_i}-gz & \mbox{ in }\Omega^t_i, \ i=1,2, \\
         \partial_t \zeta_1  \ = \ \sqrt{1+|\partial_x\zeta_1|^2}\partial_{n_1}\phi_1  & \mbox{ on } \Gamma_1\equiv\{(x,z),z=d_1+\zeta_1(t,x)\}, \\
         \partial_t \zeta_2  \ = \ \sqrt{1+|\partial_x\zeta_2|^2}\partial_{n_2}\phi_1 \ = \ \sqrt{1+|\partial_x\zeta_2|^2}\partial_{n_2}\phi_2  & \mbox{ on } \Gamma_2\equiv\{(x,z),z=\zeta_2(t,x)\},\\ 
         \partial_{z}\phi_2  \ = \ 0 &  \mbox{ on } \Gamma_b\equiv\{(x,z),z=-d_2\}, \\
         \llbracket P \rrbracket \ = \ \sigma\ \partial_x\left(\dfrac{\partial_x\zeta_2}{\sqrt{1+|\partial_x\zeta_2|^2}}\right)& \mbox{ on } \Gamma_2,
         \end{array}
\right.
\end{equation}
where $\partial_{n_i}$ is the upward normal derivative at $\Gamma_i$:
\[\partial_{n_i}\ \equiv\ n_i\cdot\nabla_{x,z}, \qquad \mbox{ with }\ n_i\ \equiv\ \frac{1}{\sqrt{1+|\partial_x \zeta_i|^2}}(-\partial_x\zeta_i,1)^T;\]
we denote $\sigma>0$ the surface tension coefficient and $\llbracket P \rrbracket$ the jump of the pressure at the interface:
\[\llbracket P(t,x) \rrbracket \ \equiv \ \lim\limits_{\varepsilon\to 0} \Big( \ P(t,x,\zeta_2(t,x)+\varepsilon) \ - \ P(t,x,\zeta_2(t,x)-\varepsilon) \ \Big).\]

This system can be reduced into evolution equations located at the surface and at the interface. Indeed, when we define the trace of the potentials at the surface and at the interface
\[\psi_1(t,x) \ \equiv \ \phi_1(t,x,d_1+\zeta_1(t,x)),\qquad  \mbox{ and } \qquad\psi_2(t,x)\ \equiv\ \phi_2(t,x,\zeta_2(t,x)),\]
then $\phi_1$ and $\phi_2$ are uniquely given by the Laplace equation, and the (Dirichlet or Neumann) boundary conditions (see Figure~\ref{fig:Laplace}):
\begin{equation}\label{eqn:Laplace}
\left\{\begin{array}{ll}
 (\ \partial_x^2  \ + \ \partial_z^2\ )\ \phi_2=0 & \mbox{in } \Omega_2, \\
\phi_2 =\psi_2 & \mbox{on } \Gamma_2, \\
 \partial_{z}\phi_2 =0 & \mbox{on } \Gamma_b, \\

\end{array}
\right. \quad \text{and} \quad 
\left\{\begin{array}{ll}
 (\ \partial_x^2  \ + \ \partial_z^2\ )\ \phi_1=0 & \mbox{in } \Omega_1, \\
\phi_1 =\psi_1 & \mbox{on } \Gamma_1, \\
 \partial_{n_2}\phi_1 =\partial_{n_2}\phi_2  & \mbox{on } \Gamma_2.
\end{array}\right.\end{equation}
\begin{figure}[hptb]
\psfrag{Cb}[Bl][Bl]{\begin{footnotesize}$\partial_{z}\phi_2 =0$                          \end{footnotesize}}
\psfrag{C1}[Bl][Bl]{\begin{footnotesize}$\phi_1 =\psi_1$                          \end{footnotesize}}
\psfrag{C2}[Bl][Bl]{\begin{footnotesize}$\partial_{n_2}\phi_1 =\partial_{n_2}\phi_2$                          \end{footnotesize}}
\psfrag{C3}[Bl][Bl]{\begin{footnotesize}$\phi_2 =\psi_2$                          \end{footnotesize}}
\psfrag{O1}[Bl][Bl]{\begin{footnotesize}$(\ \partial_x^2  \ + \ \partial_z^2\ )\ \phi_1=0$                          \end{footnotesize}}
\psfrag{O2}[Bl][Bl]{\begin{footnotesize}$(\ \partial_x^2  \ + \ \partial_z^2\ )\ \phi_2=0$                          \end{footnotesize}}
\subfigure{\includegraphics [width=0.49\textwidth]{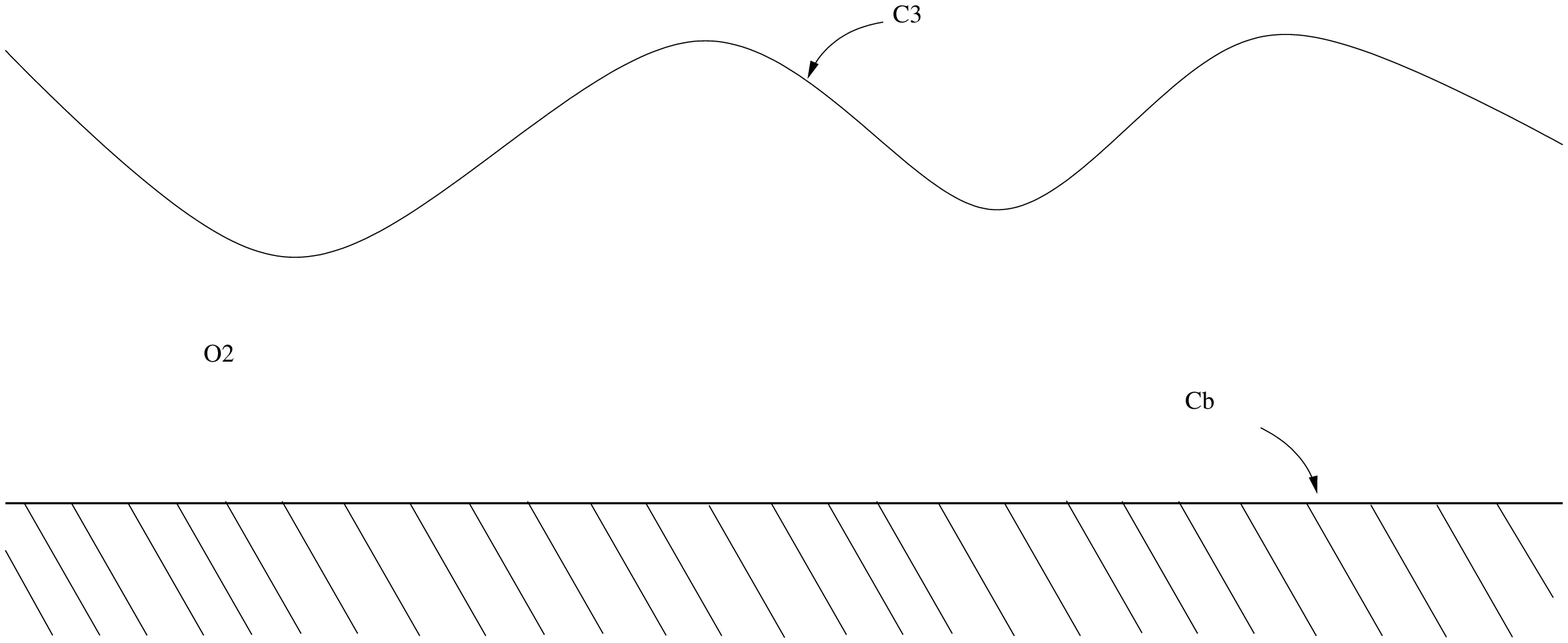}} 
\subfigure{\includegraphics [width=0.49\textwidth]{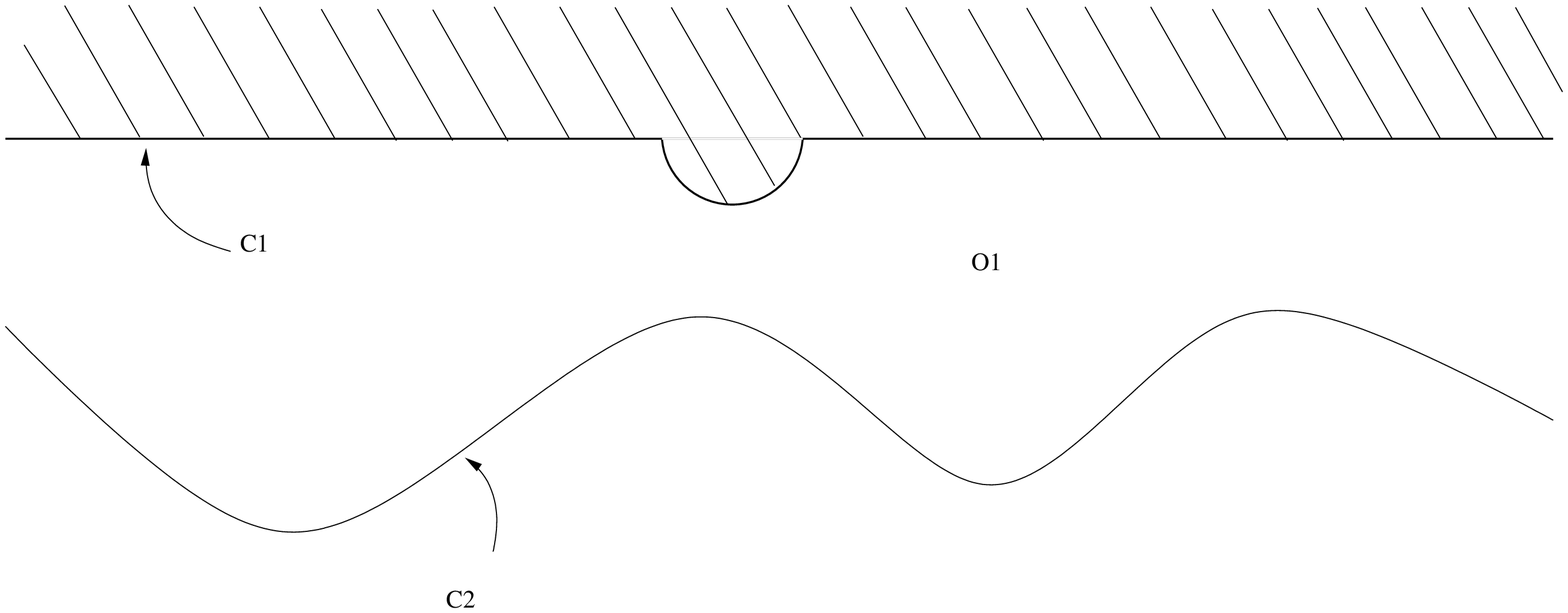}} 
 \caption{Laplace problems in the two domains.}
\label{fig:Laplace}
\end{figure}

Therefore, the following operators are well-defined:
\begin{align*}
  &G_1[\zeta_1,\zeta_2](\psi_1,\psi_2)\ \equiv\ \sqrt{1+|\partial_x\zeta_1|^2}(\partial_{n_1}\phi_1)\id{z=1+\zeta_1},\\
  &G_2[\zeta_2]\psi_2\ \equiv\ \sqrt{1+|\partial_x\zeta_2|^2}(\partial_{n_2}\phi_2)\id{z=\zeta_2},\\
  &H[\zeta_1,\zeta_2](\psi_1,\psi_2)\ \equiv\ \partial_x \big( {\phi_1}\id{z=\zeta_2}\big).
\end{align*}
Then, the chain rule and straightforward combinations of the equations of~\eqref{eqn:EulerComplet} lead to the following equivalent system:
\begin{equation}
\label{eqn:Zakharov}
\left\{ \begin{array}{lr}
 \multicolumn{2}{l}{\displaystyle\partial_t\zeta_1 \ - \ G_1[\zeta_1,\zeta_2](\psi_1,\psi_2) \ = \ 0,}  \\ \\
\multicolumn{2}{l}{\displaystyle\partial_t\zeta_2 \ - \ G_2[\zeta_2]\psi_2 \ =\ 0,}  \\ \\
\multicolumn{2}{l}{\displaystyle\partial_t \Big(\rho_2\partial_x\psi_2-\rho_1 H[\zeta_1,\zeta_2](\psi_1,\psi_2)\Big)\ +\ g(\rho_2-\rho_1)\partial_x\zeta_2} \\ 
&\qquad \displaystyle+\ \frac{1}{2} \partial_x\Big(\rho_2|\partial_x \psi_2|^2-\rho_1 |H[\zeta_1,\zeta_2](\psi_1,\psi_2)|^2\Big)\ 
=\ \partial_x\N \ + \ \sigma\ \partial_x^2\bigg(\dfrac{\partial_x\zeta_2}{\sqrt{1+|\partial_x\zeta_2|^2}}\bigg),
\end{array}
\right.
\end{equation} with\[
 \N \ = \ \frac{\rho_2\big(G_2[\zeta_2]\psi_2+(\partial_x\zeta_2)(\partial_x\psi_2)\big)^2\ -\ \rho_1\big(G_2[\zeta_2]\psi_2+(\partial_x\zeta_2) H[\zeta_1,\zeta_2](\psi_1,\psi_2)\big)^2}{2(1+|\partial_x\zeta_2|^2)}.\]

\subsection{Nondimensionalization of the system}\label{sec:nondimensionalization}
The next step consists in nondimensionalizing the system. The study of the linearized system (see~\cite{Lannes10}, for example), which can be solved explicitly, leads to a well-adapted rescaling. Moreover, it is convenient to write the equations in the frame of reference of the ship.

\medskip

Let $a_1$ and $a_2$ be the maximum amplitude of the deformation of, respectively, the surface and the interface. We denote by $\lambda$ a characteristic horizontal length, say the wavelength of the interface. Then the typical velocity of small propagating internal waves (or wave celerity) is given by
\[c_0 \ = \ \sqrt{g\frac{(\rho_2-\rho_1) d_1 d_2}{\rho_2 d_1+\rho_1 d_2}}.\]
Consequently, we introduce the dimensionless variables
\[\begin{array}{cccc}
 \t z \ \equiv\  \dfrac{z}{d_1}, \quad\quad & \t x\ \equiv \ \dfrac{x+c_s t}{\lambda}, \quad\quad & \t t\ \equiv\ \dfrac{c_0}{\lambda}t,
\end{array}\]
the dimensionless unknowns
\[\begin{array}{cc}
 \t{\zeta_i}(\t x)\ \equiv\ \dfrac{\zeta_i(x)}{a_i}, \quad\quad& \t{\psi_i}(\t x)\ \equiv\ \dfrac{d_1}{a_2\lambda c_0}\psi_i(x),
\end{array}\]
and the seven independent dimensionless parameters
\[\begin{array}{cccccccc}
 \gamma\ =\ \dfrac{\rho_1}{\rho_2}, \quad &\epsilon_1\ \equiv\ \dfrac{a_1}{d_1},\quad & \epsilon_2\ \equiv \ \dfrac{a_2}{d_1}, \quad & \mu\ \equiv\ \dfrac{d_1^2}{\lambda^2}, & \quad \delta\ \equiv \ \dfrac{d_1}{d_2}, \quad & \Fr\ \equiv\ \dfrac{c_s}{c_0}, \quad & \Bo\ \equiv\ \dfrac{c_0^2\lambda^2\rho_2}{d_1\sigma}.
\end{array}\]
With this rescaling, the full Euler system~\eqref{eqn:EulerComplet} becomes (we have withdrawn the tildes for the sake of readability)
\begin{equation}\label{eqn:AdimEulerComplet}
\left\{ \begin{array}{l}
\displaystyle-\epsilon_1 \Fr\partial_x{\zeta_1}\ -\ \frac{\epsilon_2}{\mu}G_1(\psi_1,\psi_2)\ =\ 0,  \\ \\
\displaystyle(\partial_{ t}-\Fr\partial_x){\zeta_2} \ -\ \frac{1}{\mu}G_2\psi_2\ =\ 0,  \\ \\
\displaystyle(\partial_{ t}-\Fr\partial_x)\Big(\partial_x{\psi_2}-\gamma H(\psi_1,\psi_2)\Big)\ + \ (\gamma+\delta)\partial_x{\zeta_2} \\ 
\multicolumn{1}{r}{\displaystyle \qquad + \ \frac{\epsilon_2}{2} \partial_x\Big(|\partial_x {\psi_2}|^2-\gamma |H(\psi_1,\psi_2)|^2\Big) \ = \ \mu\epsilon_2\partial_x\N \ + \ \dfrac{1}{\Bo} \partial_x^2\bigg(\dfrac{\partial_x\zeta_2}{\sqrt{1+\mu\epsilon_2^2|\partial_x\zeta_2|^2}}\bigg),}
\end{array} 
\right. 
\end{equation}with
\[  \N \ \equiv \ \dfrac{\big(\frac{1}{\mu}G_2\psi_2+\epsilon_2(\partial_x{\zeta_2})(\partial_x{\psi_2})\big)^2\ -\ \gamma\big(\frac{1}{\mu}G_2\psi_2+\epsilon_2(\partial_x{\zeta_2}) H(\psi_1,\psi_2)\big)^2}{2(1+\mu|\epsilon_2\partial_x{\zeta_2}|^2)},      
      \]
      and the dimensionless Dirichlet-to-Neumann operators defined by
      \begin{align*}
  &G_1(\psi_1,\psi_2)\ \equiv\ G_1^{\mu,\delta}[\epsilon_1\zeta_1,\epsilon_2\zeta_2](\psi_1,\psi_2)\ \equiv\ -\mu\epsilon_1\big(\frac{\dd}{\dd x} \zeta_1\big) (\partial_x\phi_1)\id{z=1+\epsilon_1\zeta_1}+(\partial_z\phi_1)\id{z=1+\epsilon_1\zeta_1},\\
  &G_2\psi_2\ \equiv \ G_2^{\mu,\delta}[\epsilon_2\zeta_2]\psi_2\equiv -\mu\epsilon_2(\partial_x\zeta_2) (\partial_x\phi_2)\id{z=\epsilon_2\zeta_2}+(\partial_z\phi_2)\id{z=\epsilon_2\zeta_2},\\
  &H(\psi_1,\psi_2)\ \equiv\ H^{\mu,\delta}[\epsilon_1\zeta_1,\epsilon_2\zeta_2](\psi_1,\psi_2) \ \equiv\  (\partial_x\phi_1)\id{z=\epsilon_2\zeta_2}+\epsilon_2(\partial_x \zeta_2)(\partial_z\phi_1)\id{z=\epsilon_2\zeta_2},
\end{align*}
where $\phi_1$ and $\phi_2$ are the solutions of the Laplace problems
\begin{align*}&\left\{\begin{array}{ll}
 \left(\ \mu\partial_x^2\ + \ \partial_z^2\ \right)\ \phi_2=0 & \mbox{ in } \Omega_2\equiv\{(x,z)\in \RR^{2}, -\frac{1}{\delta}<z<\epsilon_2\zeta_2(x)\}, \\
\phi_2 =\psi_2 & \mbox{ on } \Gamma_2\equiv \{z=\epsilon_2 \zeta_2\}, \\
 \partial_{z}\phi_2 =0 & \mbox{ on } \Gamma_b\equiv \{z=-\frac{1}{\delta}\}, \\
\end{array}
\right.\\ 
&\left\{\begin{array}{ll}
 \left(\ \mu\partial_x^2 \ +\  \partial_z^2\ \right)\ \phi_1=0 & \mbox{ in } \Omega_1\equiv \{(x,z)\in \RR^{2}, \epsilon_2{\zeta_2}(x)<z<1+\epsilon_1\zeta_1(x)\}, \\
\phi_1 =\psi_1 & \mbox{ on } \Gamma_1\equiv \{z=1+\epsilon_1 \zeta_1\}, \\
 \partial_{n_2}\phi_1 = G_2\psi_2 & \mbox{ on } \Gamma_2.
\end{array}
\right.\end{align*}
\begin{Remark}
 Let us keep in mind that in our case, the function $\zeta_1$ is not an unknown, but some fixed data of the problem:
 \[\zeta_1(t,x) \ = \ \zeta_1(x),\]
 where $\zeta_1$ is the submerged part of the ship (independent of time thanks to the change in the frame of reference). In that way, the first line of~\eqref{eqn:AdimEulerComplet} is a relation connecting $\psi_1$ with $\psi_2$, and~\eqref{eqn:AdimEulerComplet} can be reduced into a system of two equations with two unknowns. This reduction is computed explicitly in the following asymptotic models.
\end{Remark}
\begin{Remark}\label{rem:surfacetension}
 The Cauchy problem associated to the full Euler system at the interface of two fluids is known to be ill-posed in Sobolev spaces in the absence of surface tension, as Kelvin-Helmholtz instabilities appear. However, in~\cite{Lannes10}, Lannes proved that adding a small amount of surface tension guarantees the well-posedness of the Cauchy problem (with a flat rigid lid), with a time of existence that may remain quite large even with a small surface tension coefficient, and thus is consistent with the observations. The idea behind this result is that the Kelvin-Helmholtz instabilities appear for high frequencies, where the regularization effect of the surface tension is relevant, even if $\Bo$ the Bond number measuring the ratio of gravity forces over capillary forces is very large.\footnote{As an example, let us consider the values of the experiments of Vasseur, Mercier and Dauxois~\cite{VasseurMercierDauxois11}. One has $d_1=0.05\text{m}$, $d_2=0.12\text{m}$, $\rho_1=1\,000.5\text{kg.m}^{-3}$, $\rho_2=1\,022.7\text{kg.m}^{-3}$. Taking advantage of the analysis of Lannes~\cite{Lannes10}, we use $\sigma=0.095\text{N.m}^{-1}$ as a typical value for the interfacial tension coefficient. As a conclusion, one has $\Bo  =  \frac{c_0^2 \rho_2 d_1}{\sigma} \frac1\mu \approx  \frac{40}\mu$. The coefficient in front of the surface tension term is therefore much smaller than the coefficients in front of any other term.} On the contrary, the main profile of the wave that we want to capture is located at lower frequencies, and will therefore be unaffected by surface tension. Driven by this analysis, we decide to omit the surface tension term in the models of the Regime~\ref{regimeSA} presented below, as they do not give rise to Kelvin-Helmholtz instabilities. On the contrary, we keep the surface tension term in the fully nonlinear models of Regime~\ref{regimeRL}. As a matter of fact the regularization effect of the surface tension in this case plays a critical role in our numerical simulations, as it stabilizes the scheme, even if $\Bo^{-1}$ is one order smaller than the coefficients in front of all other terms.
\end{Remark}
The following terminology is used throughout the paper.
\begin{Definition}[Adapted solutions]
 We call {\em adapted} solution of~\eqref{eqn:AdimEulerComplet} any strong solution $(\zeta_1,\zeta_2,\psi_1,\psi_2)$, bounded in $W^{1,\infty}([0,T];H^{s+t_0})$ with $T>0$, $s>1$ and $t_0$ big enough, and such that $\zeta_1(t,x) \ = \ \zeta_1(x)$ (in the frame of reference of the ship) and the domains of the fluids remain strictly connected, {\em i.e.} there exists $h_{\min}>0$ such that
 \begin{equation}\label{eqn:h}
   1+\epsilon_1\zeta_1(x)-\epsilon_2\zeta_2(t,x)\geq h_{\min}>0 \quad \mbox{and} \quad \frac1\delta+\epsilon_2\zeta_2(t,x)\geq h_{\min}>0.
 \end{equation}
\end{Definition}
From the discussion above, it is legitimate to assume that such a smooth, uniformly bounded family of solutions of~\eqref{eqn:AdimEulerComplet} indeed exists.


\subsection{Description of the results, and the regimes under study}\label{sec:regimes}

The models displayed in the following sections are all justified with a consistency result; see definition below. When it is possible, that is when energy estimates are attainable, we provide a stronger justification with a convergence rate (the convergence theorems are precisely disclosed in Proposition~\ref{prop:ConvBouss} and~\ref{PROP:CONVKDV}).
\begin{Definition}[Consistency]\label{def:consistency}
The full Euler system~\eqref{eqn:AdimEulerComplet} is {\em consistent} with a system of two equations $(\Sigma)$ on $[0,T]$ if any {\em adapted solution} $U$ defines, {\em via} the changes of variables explained throughout the paper, a pair of functions satisfying $(\Sigma)$ up to a small residual, called the {\em precision} of the model. The order of the precision will be $\O(\varepsilon)$, if there exists $s_0,t_0\ge0$ such that if ${U\in W^{1,\infty}([0,T];H^{s+t_0})}$ with $s>s_0$, then the $L^\infty([0,T];H^s)$ norm of the residual is bounded by $C_0\ \varepsilon$, with $C_0$ a constant independent of $\varepsilon$.    
\end{Definition}
\begin{Definition}[Convergence]\label{def:convergence}
The full Euler system~\eqref{eqn:AdimEulerComplet} and a well-posed system $(\Sigma)$ of two equations are {\em convergent} at order $\O(\varepsilon)$ on $[0,T]$ if any adapted solution with small initial data $U$ defines, {\em via} the changes of variables explained throughout the paper, a pair of functions $V$ such that $\t V$ the solution of $(\Sigma)$ with same initial data satisfies 
\[ \big| V - \t V \big|_{L^\infty([0,T];H^s)} \ \leq \ C_0\ \varepsilon,\]
with $C_0$ independent of $\varepsilon$.
\end{Definition}

\bigskip

The small parameter $\varepsilon$ in these definitions is a function of some of the dimensionless parameters of the system that are assumed to be small. The regimes that we study throughout the paper have been briefly presented in the introduction; let us describe them here in more details. When nothing is specifically said in the regimes below, we assume the parameters to be fixed as
\[\gamma\in(0,1), \quad \epsilon_1,\epsilon_2\in(0,1),\quad 0<\delta_{\text{min}}\le \delta\le\delta_{\text{max}}, \quad \Fr\in[0,\Fr_{\text{max}}], \quad 0<\Bo_{\text{min}}\le\Bo.\]
In particular, the two layers are assumed to be of finite, comparable depth.

\medskip

The mutual and crucial assumption between the two regimes is that the depth of the two layers of fluids is small when compared with the internal wavelength. This is commonly called {\em shallow water regime} and is simply given by (with our notation)
\[\mu \ \ll \ 1.\]
The shallow water regime has been widely used in the framework of gravity waves. In the one layer, free surface case, it leads at order $\O(\mu)$ to the shallow water equations~\cite{Saint-Venant71}, and at order $\O(\mu^2)$ to the Green-Naghdi equations\cite{GreenNaghdi76}. The analysis has been extended to the case of two layers with a free surface in~\cite{Duchene10}, and the models presented here can be recovered from models in this situation, when fixing the surface as some data of the problem.

\bigskip

\begin{Regime}[Small sized boats]\label{regimeRL}
\[\mu \ \ll \ 1 \ ; \qquad \alpha\equiv\frac{\epsilon_1}{\epsilon_2} \ = \ \O(\mu)\ , \ 1-\gamma \ = \ \O(\mu).\]
\end{Regime}
\noindent The two additional smallness assumptions of this regime are very natural. First, we assume that the depth of the submerged part of the ship is small when compared with the depth of the fluid, and the attainable size of the deformation. The numerical simulations of our model show that even with this assumption, the waves generated by the small disturbance can be very large (of the order of the depth of the two layers, and therefore very large when compared with the variation of the surface induced by the ship). This explains why the dead-water phenomenon can be so powerful. What is more, we assume that the densities of the two fluids are almost equal. This is known as the Boussinesq approximation, and is valid most of the time (the value of $1-\gamma$ reported in Celtic Sea in~\cite{OstrovskyGrue03} is $\approx 10^{-3}$, and the value in the experiments of~\cite{VasseurMercierDauxois11} is $\approx 10^{-2}$). 

\bigskip

\begin{Regime}[Small wave amplitudes]\label{regimeSA}
\[\mu \ll 1\ ; \qquad \epsilon_2\ = \ \O(\mu)\ , \ \alpha\equiv\frac{\epsilon_1}{\epsilon_2}\ = \ \O(\mu).\]
\end{Regime}
\noindent In this regime, we assume that the internal wave generated by the disturbance will remain small when compared with the depth of the two fluids. As previously, we assume that the waves generated by the ship are large when compared with the depth of the  disturbance, so that the ship suffers from a relatively large wave resistance. In that way, the models we obtain involve only weak nonlinearities, and will be much easier to study. In particular, we are able to obtain well-posedness and convergence results for both of the models presented here. The counterpart is that these models remain valid only for relatively small waves. Again, this regime has been widely used in the literature. The one layer, free surface equivalent systems of the models presented here are the classical Boussinesq system~\cite{Boussinesq71,Boussinesq72} and the KdV approximation~\cite{KortewegDe95}. These models have been rigorously justified in~\cite{SchneiderWayne00,BonaColinLannes05,Alvarez-SamaniegoLannes08}, and the extension to the case of two layers, with a free surface, can be found in~\cite{Duchene}.

\subsection{Expansion of the Dirichlet-Neumann operators} \label{sec:expansion}
The key assumption in our models, which is shared by both regimes~\ref{regimeRL} and~\ref{regimeSA}, is the so-called shallow water assumption
\[\mu \ \ll \ 1,\]
which states that the depth of the two layers of fluids are small when compared with the characteristic wavelength of the system.
The main ingredient in the construction of asymptotic models in the shallow water regime, then lies in the construction and the justification of the following expansion of the operators $G_1$, $G_2$ and $H$:
\begin{Lemma}\label{Lem:extendVD1} Let $\zeta_1,\zeta_2,\psi_1,\psi_2$, such that $(\zeta_1,\zeta_2,\partial_x\psi_1,\partial_x\psi_2) \in W^{1,\infty}([0,T];H^{s+t_0}(\RR))$, with $s>1$ and $t_0\ge 9/2$, and such that~\eqref{eqn:h} is satisfied. Then one has
\begin{align*}
\Big\vert G_1(\psi_1,\psi_2)+\mu \partial_x(h_1\partial_x\psi_1+h_2\partial_x\psi_2) -\mu^2\partial_x\Big( \T[h_1,h_2]\partial_x\psi_1 + \T[h_2,0]\partial_x\psi_2 \hspace{3cm} &\\ 
-\frac{1}{2}(h_1^2\partial_x^2(h_2\partial_x\psi_2))-(h_1\epsilon_1\frac{\dd}{\dd x} \zeta_1 \partial_x(h_2\partial_x\psi_2)) \Big)\Big\vert_{L^{\infty}H^s}\ \leq\ \mu^3 C_0, &\\
\Big\vert G_2\psi_2+\mu\partial_x  (h_2\partial_x\psi_2)-\mu^2 \partial_x \T[h_2,0]\partial_x\psi_2\Big\vert_{L^{\infty}H^s} \ \leq \ \mu^3 C_0, &\\
\Big\vert H(\psi_1,\psi_2)-\partial_x\psi_1-\mu\partial_x\Big(h_1(\partial_x(h_1\partial_x\psi_1)+\partial_x(h_2\partial_x\psi_2))\hspace{6cm} & 
\\-\frac{1}{2}h_1^2\partial_x^2\psi_1  -h_1\epsilon_1(\partial_x\psi_1)\frac{\dd}{\dd x} \zeta_1(x)\Big) \Big\vert_{W^{1,\infty}H^s} \ \leq \ \mu^2 C_0,&
 \end{align*}
with  $C_0=C_0\left(\frac{1}{h_{\min}}, \big\vert U\big\vert_{W^{1,\infty}H^{s+t_0}}\right)$, and the operator $\T$ defined by
\[\T[h,b]V\ \equiv -\frac{1}{3}\partial_x(h^3\partial_x V)+\frac{1}{2}\big(\partial_x(h^2(\partial_x b) V))-h^2(\partial_x b)(\partial_x V)\big) +h(\partial_x b)^2  V.\]
\end{Lemma}
 These estimates have been proved for $L^{\infty}([0,T];H^s)$ norms in Propositions 2.2, 2.5 and 2.7 of~\cite{Duchene10}. The proof can easily be adapted to work with the time derivative of the functions, following the proof of~\cite[Proposition 2.12]{Duchene10}.
 
 \medskip

The idea is then simply to plug these expansions into the full Euler system~\eqref{eqn:AdimEulerComplet}, and drop all $\O(\mu^2)$ terms.
Then, using the fact that $\zeta_1$ is a forced parameter of our problem, the system reduces to two evolution equations for $(\zeta_2,v)$, with $v$ the {\em shear velocity} defined by
\begin{equation}\label{eqn:defv}
v \ \equiv\ \partial_x\left(\big(\phi_2 - \gamma\phi_1 \big)\id{z=\epsilon_2\zeta_2}\right)\ =\ \partial_x\psi_2-\gamma H(\psi_1,\psi_2).\end{equation}
These calculations are postponed to Appendix~\ref{sec:fullyscheme} for the sake of readability, and we directly present here the system thus obtained:
 \begin{equation}\label{eqn:Sbegin} \left\{ \begin{array}{l}
(\partial_{ t}-\Fr\partial_x) \zeta_2 +\partial_x\left(\dfrac{h_2}{h_1+\gamma h_2}(h_1  v+\gamma\alpha \Fr  \zeta)\right)+\mu\partial_x\big( \L(v_1,v_2)\big)  = 0,\\ \\
(\partial_{ t}-\Fr\partial_x) v + (\gamma+\delta)\partial_x \zeta_2 +\dfrac{\epsilon_2}{2}\partial_x \left(\dfrac{|h_1  v+\gamma\alpha \Fr  \zeta|^2 -\gamma |h_2 v -\alpha \Fr  \zeta|^2 }{(h_1+\gamma h_2)^2} \right)\vspace{10pt}\\
\multicolumn{1}{r}{+\mu \epsilon_2 \partial_x\big(\Q[v_1,v_2]\big) \ = \ \dfrac{1}{\Bo} \partial_x^2\bigg(\dfrac{\partial_x\zeta_2}{\sqrt{1+\mu\epsilon_2^2|\partial_x\zeta_2|^2}}\bigg),}
\end{array} \right. \end{equation}
where $\L$ and $\Q$ are respectively linear and quadratic in $(v_1,v_2)$, the latter being the approximation at order $\O(\mu)$ of $(\partial_x\psi_1,\partial_x\psi_2)$ given by 
\begin{equation}\label{eqn:v1v2v} v_1 \ \equiv -\dfrac{h_2v-\alpha \Fr  \zeta_1}{h_1+\gamma h_2}, \quad \text{ and } \quad v_2 \ \equiv \ \dfrac{h_1v+\gamma\alpha \Fr \zeta_1}{h_1+\gamma h_2}. \end{equation}
The operators $\L$ and $\Q$ are precisely disclosed in Appendix~\ref{sec:fullyscheme}.
The full Euler system~\eqref{eqn:AdimEulerComplet} is consistent with this system at order $\O(\mu^2)$ on $[0,T]$, $T>0$, as stated in Proposition~\ref{prop:ConsFullyNL}, page~\pageref{prop:ConsFullyNL}.

\bigskip

The simplified models we study are deduced from system~\eqref{eqn:Sbegin}, using the additional assumptions of Regime~\ref{regimeRL} and Regime~\ref{regimeSA}. The full Euler system~\eqref{eqn:AdimEulerComplet} is still consistent at order $\O(\mu^2)$ with the models thus obtained, that is (see below) the strongly nonlinear systems~\eqref{eqn:D1} and~\eqref{eqn:Dfinal}, the Boussinesq-type system~\eqref{eqn:Bouss} and its symmetrized version~\eqref{eqn:SymBouss}. The KdV approximation~\eqref{eqn:KdV1} is then deduced from the symmetric Boussinesq-type system.

\section{Strongly nonlinear models}\label{sec:stronglynonlinear}
In this section, we introduce different strongly nonlinear asymptotic models for the full Euler system~\eqref{eqn:AdimEulerComplet}. These models are obtained starting from system~\eqref{eqn:Sbegin}, using the additional properties of Regime~\ref{regimeRL}. The weakly nonlinear models (Regime~\ref{regimeSA}) are presented in following Section~\ref{sec:weaklynonlinear}. Indeed, the system~\eqref{eqn:Sbegin} only presumes some smallness on the shallowness parameter $\mu$, that is to say the depth of the fluids are small when compared with the characteristic wavelength of the system. In the framework of Regime~\ref{regimeRL}, we assume that
\[\mu \ \ll \ 1 \ ; \qquad \alpha\equiv\frac{\epsilon_1}{\epsilon_2} \ = \ \O(\mu)\ , \ 1-\gamma \ = \ \O(\mu),\]
where we recall that $\gamma$ is the density ratio and $\alpha$ the ratio between the amplitudes of the deformations at the surface and at the interface.
As it has been said, page~\pageref{regimeRL}, these assumptions are very natural in the framework of our study. The first one supposes that the deformation induced by the presence of the ship at the surface is small when compared with the depth of the two layers of fluid, and small when compared with the attainable size of the wave generated at the interface. The second one is the classical Boussinesq approximation. The simplified systems we obtain remain consistent at order $\O(\mu^2)$, as stated in Proposition~\ref{prop:ConsFullyRL}. In Section~\ref{sec:numericsFNL}, we present and discuss numerical computations from the constructed models, that reproduce the key aspects of the dead-water phenomenon. 

\subsection{The fully nonlinear model in Regime~\ref{regimeRL} }\label{sec:regime2}
The first obvious observation is that the total depth of the fluid is approximatively constant in Regime~\ref{regimeRL}, {\em i.e.} equal to ${h \ \equiv\ 1+\frac1\delta}$ at order $\O(\mu)$. What is more, one has 
\[ h_1+\gamma h_2 \ =\  h \ +\  \O(\mu)\ \equiv\ 1+\frac1\delta \ + \ \O(\mu).\]
Therefore, the approximations of $(\partial_x \psi_1,\partial_x\psi_2)$ at order $\O(\mu)$ given in~\eqref{eqn:v1v2v} are now simply
\[h_2 v = -h\ v_1 + \O(\mu) , \quad \quad  h_1 v =  h\ v_2 + \O(\mu)  .\]
 It follows then some substantial simplifications, and one obtains in the end 
the system
 \begin{equation}\label{eqn:D1} \left\{ \begin{array}{l}
(\partial_{ t}-\Fr\partial_x) \zeta_2\ +\ \partial_x\left(\dfrac{h_1h_2}{h_1+\gamma h_2} v\ +\ h_2 \dfrac{\alpha \Fr  }{h}\zeta_1\right)\ +\ \mu \partial_x \big(\P_1v\big) \ =\ 0,\\ \\
(\partial_{ t}-\Fr\partial_x) v\ + \ (\gamma+\delta)\partial_x \zeta_2 \ +\ \epsilon_2\partial_x \left(\dfrac12\dfrac{|h_1  v|^2 -\gamma |h_2 v|^2 }{(h_1+\gamma h_2)^2} +\dfrac{\alpha \Fr }h \zeta_1 v \right)\vspace{10pt} \\
\multicolumn{1}{r}{+\ \mu \epsilon_2\partial_x\big(\P_2[v]\big)\ =\  \dfrac{1}{\Bo} \partial_x^2\bigg(\dfrac{\partial_x\zeta_2}{\sqrt{1+\mu\epsilon_2^2|\partial_x\zeta_2|^2}}\bigg),}
\end{array} \right. \end{equation}
where $h\ \equiv\ 1+\frac1\delta$ and the operators $\P_1$ and $\P_2$ are defined by
\begin{align*}
h^2\ \P_1v & \ =\ \frac13\left(h_2 \partial_x(h_1^3 \partial_x (h_2 v))+h_1 \partial_x(h_2^3 \partial_x (h_1 v)) \right), \\
h^2\ \P_2[v] &\ = \ \frac{v}3\partial_x\left[h_1^3 \partial_x (h_2 v) -h_2^3 \partial_x (h_1 v)\right] + \frac12 \left( (h_1\partial_x(h_2v))^2-(h_2\partial_x(h_1v))^2\right).
\end{align*}

\bigskip

Linearizing the system~\eqref{eqn:D1} around the rest state, leads to
 \[\left\{ \begin{array}{l}
(\partial_{ t}-\Fr\partial_x) \zeta_2\ +\ \dfrac{1}{\gamma+\delta}\partial_x v\  +\ \mu \dfrac1{3\delta(1+\delta)}\partial_x^3 v \ =\ \dfrac{\alpha \Fr  }{1+\delta}\partial_x \zeta_1,\\ 
(\partial_{ t}-\Fr\partial_x) v\ + \ (\gamma+\delta)\partial_x \zeta_2 \ =\ \dfrac{1}{\Bo}\partial_x^3\zeta_2,
\end{array} \right. \]
from which we can easily deduce dispersion relation of the system, without forcing. Indeed, setting $\alpha\equiv0$, the wave frequency $\omega(k)$, corresponding to plane-wave solutions $e^{ik\cdot X-i\omega t}$, is solution of the quadratic equation
\begin{equation}\label{eqn:dispersion1}
(\Fr k \ + \ \omega)^2 \ = \ k^2\ \Big(\ 1 \ + \ \frac{k^2}{\Bo(\gamma+\delta)}\ \Big)\Big(\ 1 \ -\ \mu k^2 \frac{\delta+\gamma}{3\delta(1+\delta)}  \ \Big). \end{equation}
Therefore, for high wave numbers, $k$, there is no real-valued solution $\omega(k)$ of~\eqref{eqn:dispersion1}, which means that the system~\eqref{eqn:D1} is linearly ill-posed, and leads to instabilities. In order to deal with this issue, we use a nonlinear change of variables on the shear velocity $v$, that leads to a formally equivalent system. This idea is not new (see~\cite{BonaChenSaut02,NguyenDias08,ChoiBarrosJo09,Duchene10} for example), and usually relies on a ``shear velocity'' constructed from horizontal velocities {\em at any depth} in the upper and lower layers. The choice of the depth allows then to control properties of the linear relation dispersion. Here, we define $w$ with
\[ \frac{h_1h_2}{h_1+\gamma h_2} w \ \equiv\ \frac{h_1h_2}{h_1+\gamma h_2} v\ +\ \mu \P_1v \  \Longrightarrow\ v\ =\ w\ -\ \mu \ \frac{h\P_1w}{(h-h_2)h_2} \ + \O(\mu^2).\]
This leads immediately to the following system, equivalent at order $\O(\mu^2)$:
\begin{equation}\label{eqn:Dfinal}
 \left\{ \begin{array}{l}
(\partial_{ t}-\Fr\partial_x) \zeta_2  \ +\  \partial_x\left(\dfrac{h_1h_2}{h_1+\gamma h_2} w+h_2 \dfrac{\alpha \Fr }h \zeta_1\right) \ = 0,\\ \\
(\partial_{ t}-\Fr\partial_x)\left(w-\mu\S_1w\right) \ +\ (\gamma+\delta)\partial_x \zeta_2\  +\ \epsilon_2\partial_x \left(\dfrac12\dfrac{h_1^2 -\gamma h_2^2 }{(h_1+\gamma h_2)^2}w^2 + w\dfrac{\alpha \Fr }h \zeta_1\right)\vspace{10pt} \\
\multicolumn{1}{r}{+ \ \mu\epsilon_2\partial_x \left(w\ \S_2 w\right)\ =\ \dfrac{1}{\Bo} \partial_x^2\bigg(\dfrac{\partial_x\zeta_2}{\sqrt{1+\mu\epsilon_2^2|\partial_x\zeta_2|^2}}\bigg),}
\end{array} \right.
\end{equation}
with $h\ \equiv\ 1+\frac1\delta$, $h_1\ \equiv\ 1+\epsilon_1\zeta_1-\epsilon_2\zeta_2$, $h_2\ \equiv\ \frac1\delta+\epsilon_2\zeta_2$, and the operators
\begin{align*}
\S_1w &\equiv \ \frac13\left(\partial_x^2\big(h_1h_2 w\big)-(\partial_x h_2)^2w\right),\\
\S_2w &\equiv \ \frac{h_1h_2}{3h} \left( (\partial_x^2 h_2)w +2 (\partial_x w )(\partial_x h_2)\right)+\frac{h_1-h_2}{2h}(\partial_x h_2)^2 w.
\end{align*}
Under this form, our system corresponds to the Green-Naghdi model presented in the one-layer case in~\cite{Alvarez-SamaniegoLannes08,Alvarez-SamaniegoLannes08a}, and proved to be well-posed and convergent, in the sense that their solutions provide an approximation at order $\O(\mu^2)$ in $L^\infty([0,T];H^s)$ to the solutions of the one layer water wave equations. When dropping all $\O(\mu)$ terms, one obtains a shallow water model, similar to the ones derived in~\cite{ChoiCamassa99} and~\cite{CraigGuyenneKalisch05} in the flat rigid lid case (the forcing terms induced by the presence of the body do not appear). Such a system has been studied in details and analyzed as a hyperbolic system (leading to well-posedness results under reasonable assumptions on the initial data) in~\cite{MilewskiTabakTurnerEtAl04,GuyenneLannesSaut10,BreschRenardy}.

\medskip

One can check that the issue of linear ill-posedness is now solved. Indeed, linearizing the system~\eqref{eqn:Dfinal} around the rest state, leads to
 \[\left\{ \begin{array}{l}
(\partial_{ t}-\Fr\partial_x) \zeta_2\ +\ \dfrac{1}{\gamma+\delta}\partial_x w\  \ =\ \dfrac{\alpha \Fr  }{1+\delta}\partial_x \zeta_1,\\ 
(\partial_{ t}-\Fr\partial_x) (1-\dfrac{\mu}{3\delta}\partial_x^2)w\ + \ (\gamma+\delta)\partial_x \zeta_2 \ =\ \dfrac{1}{\Bo}\partial_x^3\zeta_2.
\end{array} \right. \]
The dispersion relation of this linear system, when setting $\alpha\equiv0$, is
\begin{equation}\label{eqn:dispersion2}
 (\Fr k \ + \ \omega)^2\Big(1\ + \ \frac{\mu}{3\delta} k^2\Big) \ = \ k^2\Big(1 \ + \ \frac1{(\gamma+\delta)\Bo} k^2\Big),
\end{equation}
which leads to real-valued waves frequencies $\omega$, for any values of $k\in\RR$.

Let us remark that under the assumptions of Regime~\ref{regimeRL}, the relations~\eqref{eqn:dispersion1} and~\eqref{eqn:dispersion2} are asymptotically equivalent at order $\O(\mu^2k^4)$, so that the effect of the nonlinear change of variable only affects high frequencies.

\bigskip

We now state that the two systems~\eqref{eqn:D1} and~\eqref{eqn:Dfinal} are equivalently justified as models for the full Euler system, with the following consistency result. 
\begin{Proposition}\label{prop:ConsFullyRL}
Assuming that $\alpha=\O(\mu)$ and $1-\gamma =\O(\mu)$, the full Euler system~\eqref{eqn:AdimEulerComplet} is consistent with the models~\eqref{eqn:D1} and~\eqref{eqn:Dfinal}, both at precision $\O(\mu^2)$ on $[0,T]$, with $T>0$.
\end{Proposition}
\begin{proof}
 Let $U\equiv(\zeta_1,\zeta_2,\psi_1,\psi_2)$ be a strong solution of~\eqref{eqn:AdimEulerComplet}, bounded in $W^{1,\infty}([0,T];H^{s+t_0})$ with $s>1$ and $t_0\ge 9/2$, and such that~\eqref{eqn:h} is satisfied with $\zeta_1(t,x) \ \equiv \ \zeta_1(x)$. The consistency result of Proposition~\ref{prop:ConsFullyNL} states that $(\zeta_2,v)$, with $v\equiv  \partial_x\psi_2-\gamma H(\psi_1,\psi_2)$, satisfies~\eqref{eqn:Sbegin}, up to $R_1=(r_1,r_2)^T\in L^\infty([0,T];H^{s})^2$, with (for $i=1,2$)
\[ \big\vert r_i \big\vert_{L^\infty H^{s}} \  \leq \mu^2\ C_0\left(\frac{1}{h_{\min}}, \big\vert U\big\vert_{W^{1,\infty}H^{s+t_0}}\right).\]
Since $\R_1,\ \R_2, \ \T$ and $\partial_x \H$ (defined in Appendix~\ref{sec:fullyscheme}) involve two spatial derivatives, and thanks to straightforward calculations using the smallness assumptions of Regime~\ref{regimeRL}, one has for $i=1,2$ 
\[\big\vert \L_i \ - \ \P_i \big\vert_{L^\infty([0,t); H^{s})} \ \leq \  \mu^2\ C_0\left(\frac{1}{h_{\min}}, \big\vert (h_1, h_2 ,v) \big\vert_{L^{\infty}H^{s+2}}\right) \leq \ \mu^2\ C_0\left(\frac{1}{h_{\min}}, \big\vert U\big\vert_{W^{1,\infty}H^{s+t_0}}\right).\]
The consistency of~\eqref{eqn:AdimEulerComplet} with~\eqref{eqn:D1} is therefore proved.

Now, we set $ w \ \equiv\ v\ +\ \mu \frac{h_1+\gamma h_2}{h_1h_2}\P_1[h_2]v$, and one has immediately, using the fact that $H^s(\RR)$ is an algebra for $s>1/2$,
\begin{align*} \left\vert v\ -\ w\ +\ \mu \ \frac{h\P_1[h_2]w}{(h-h_2)h_2} \right\vert_{W^{1,\infty}H^{s}} & \leq \ \mu^2\ C_0\left(\frac{1}{h_{\min}}, \big\vert (h_1, h_2 ,v) \big\vert_{W^{1,\infty}H^{s+2}}\right)\\ & \leq \ \mu^2\ C_0\left(\frac{1}{h_{\min}}, \big\vert U\big\vert_{W^{1,\infty}H^{s+t_0}}\right).
\end{align*}
It is then straightforward to deduce from the previous consistency result, the consistency of~\eqref{eqn:AdimEulerComplet} with~\eqref{eqn:Dfinal}.
\end{proof}

\subsection{Numerical simulations} 
\label{sec:numericsFNL}
In figures \ref{fig:FNLd1} and \ref{fig:FNLd2}, we plot the behavior of the flow predicted by the model~\eqref{eqn:Dfinal}, for different values of the parameters. Each figure contains three panels. The left panel represents the interface deformation, depending on space ($x\in[-20,20]$) and time ($t\in[0,15]$) variables. The right panel contains the time evolution of the wave resistance coefficient, computed thanks to formula~\eqref{eqn:WaveResistanceRL}. Finally, we plot in the bottom panel the situation of the system ({\em i.e.} the surface and interface deformations) at final time $t=15$.

In these simulations, and throughout the paper, we use zero initial data conditions, and a function $\zeta_1(x)$ defined by
\[\zeta_1(x) \ \equiv \ \left\{\begin{array}{ll}
                                -\exp(-\frac{x^2}{(1-x)(x+1)}) & \quad \text{if } x\in(-1,1),\\
                                0 & \quad \text{otherwise}.
                               \end{array}\right.
\]
The parameters of the system $\epsilon_2$, $\mu$, $\alpha\equiv \epsilon_1/\epsilon_2$, $\delta$, $\gamma$, $\Fr$ and $\Bo$, are specified below each of the figures. The schemes used are described and justified in Appendix~\ref{sec:numericalschemes}. 

\bigskip

Here, we decide to study the effect of the depth ratio coefficient $\delta$. We choose two different values for the depth ratio: $\delta=5/12$ (corresponding to the experiments of~\cite{VasseurMercierDauxois11}), or $\delta=12/5$ for a thicker upper layer. We plot in figures~\ref{fig:FNLd1} and~\ref{fig:FNLd2} the outcome of our scheme~\eqref{SchemeFNL}, setting the Froude number as $\Fr=1$. Away from this critical value, as predicted by experiments~\cite{Ekman04,VasseurMercierDauxois11} and confirmed by our numerical simulations, the amplitude of the generated waves (and therefore the magnitude of the wave resistance coefficient) are significantly smaller. In that way, the outcome is similar to the one predicted by the weakly nonlinear models, and we do not present these results for the sake of brevity. We let the reader refer to Section~\ref{sec:noncritical} for a discussion on the case of subcritical and supercritical Froude values, in Regime~\ref{regimeSA}. 

\begin{figure}[hptb]\centering
\includegraphics [width=0.9\textwidth]{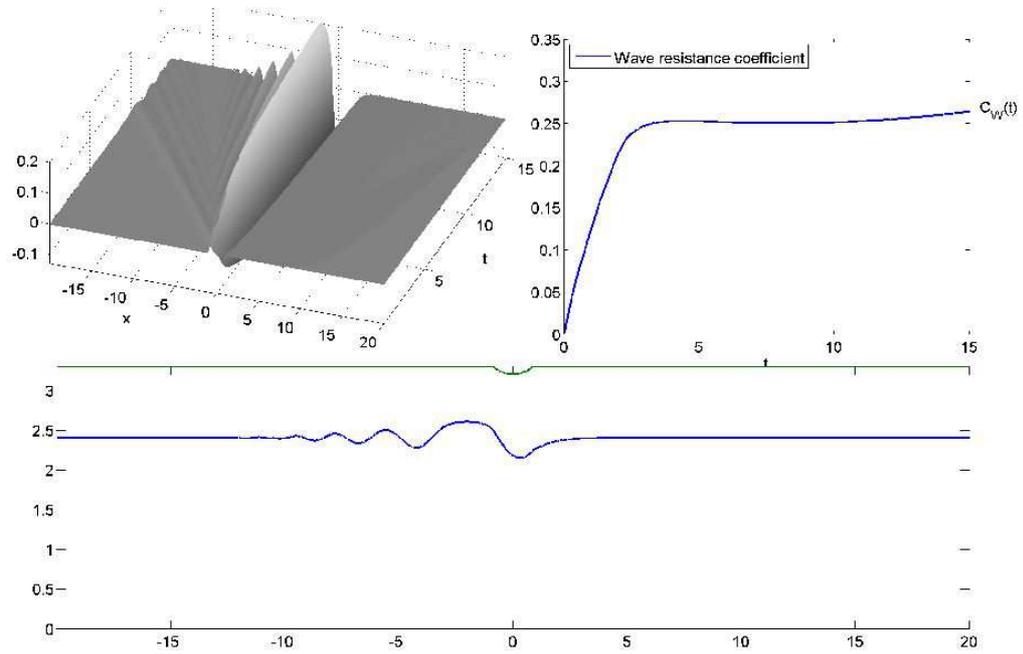}
 \caption{Flow predicted by \eqref{eqn:Dfinal}, with steady initial data and a thicker lower layer.\newline $\epsilon_2=1$, $\alpha=\mu=0.1$, $\gamma=0.99$, $\Fr=1$, $\Bo=100$, $\delta=5/12$.}
\label{fig:FNLd1}
\end{figure}
\begin{figure}[hptb]\centering
\includegraphics [width=0.9\textwidth]{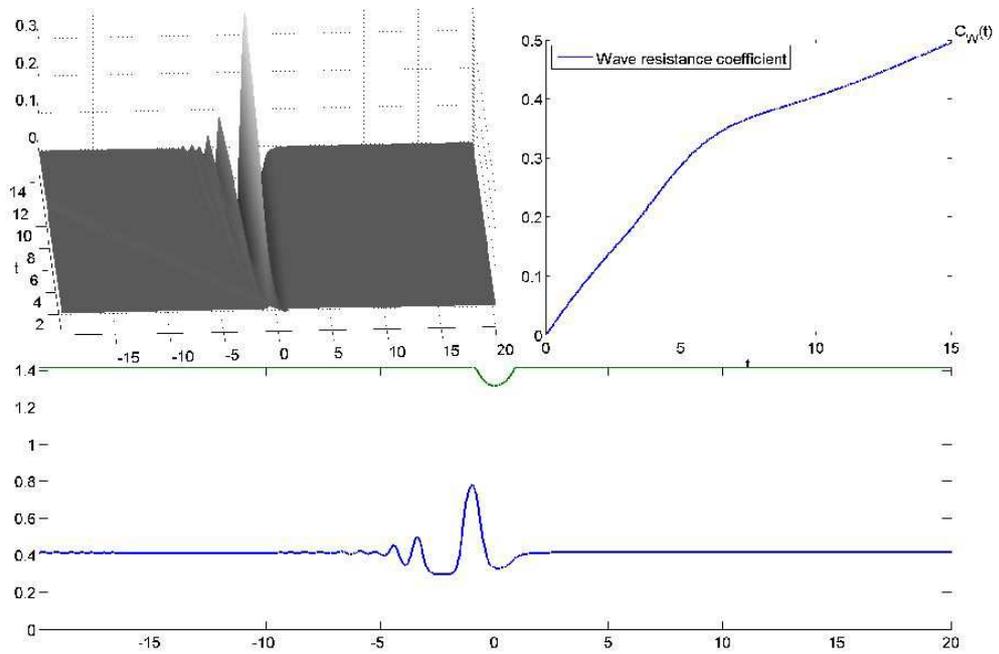}
 \caption{Flow predicted by \eqref{eqn:Dfinal}, with steady initial data and a thicker upper layer.\newline $\epsilon_2=1$, $\alpha=\mu=0.1$, $\gamma=0.99$, $\Fr=1$, $\Bo=100$,  $\delta=12/5$.}
\label{fig:FNLd2}
\end{figure}

Let us express a few remarks on these results. There are some similarities for each of the situations. First, one sees that there is no deformation on the right-hand side of the plots, which corresponds to the up-stream propagating part. This is due to the fact that we have set $\Fr=1$, which corresponds to the maximum gravitational wave velocity in the flat rigid lid case. On the contrary, on the left-hand side (down-stream propagating part), one remarks a small elevation wave, followed by even smaller perturbations, progressing with velocity $c_-=-2$ (in the frame of the ship). This corresponds to the $\eta_-$ part of the KdV approximation decomposition, and is studied with more details in Section~\ref{sec:KdVana}. Finally, the most important part takes place just behind the location of the body. Each time, an important wave of elevation is generated just behind the body, producing a severe wave resistance (see Remark~\ref{rem:waveresistance} page~\pageref{rem:waveresistance}). This wave comes with a tail of smaller waves, that are located away from the boat, and therefore do not produce any drag.

However, one can see that the shape and time-behavior of this elevation wave is quite different, depending on the value of the depth ratio $\delta$. When the upper fluid domain is thinner than the lower fluid's, the generated wave is flattening as it is growing up. Its height is relatively low, but the deformation carries on with a depression wave, located just below the body. On the contrary, when the upper fluid domain is the thicker one, then the generated wave remains sharp and is continuously growing. It is separated with its tail, while all of the energy produced by the ship contributes to the elevation of the wave. The wave produced wave resistance is greater in this case, and the dead-water effect appears therefore much stronger. The fact that the wave resistance is stronger when the upper fluid domain is thicker is counter-intuitive, as the distance between the ship and the interface is larger. However, as we can see in Appendix~\eqref{sec:deadwater}, the wave resistance coefficient does not depend on the distance between the two interfaces ($\zeta_1-\zeta_2$), but rather on the amplitude of the generated wave. It is therefore not surprising that, when the upper layer is thicker, the generated elevation wave is able to reach larger amplitudes, and the drag suffered by the body is stronger.

\bigskip


\begin{figure}[!hptb]\centering
\includegraphics [width=0.85\textwidth]{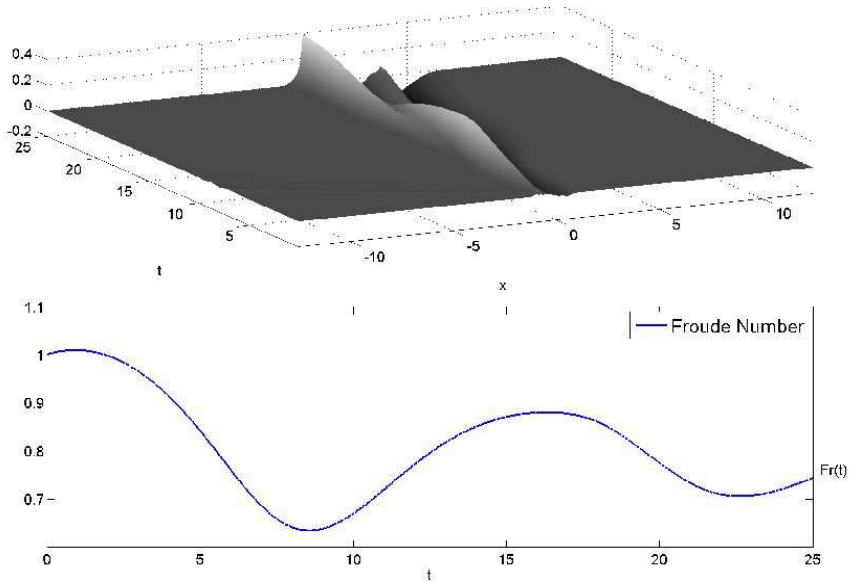}
 \caption{Flow predicted by \eqref{eqn:Dfinal}, when the velocity of the body, $\Fr$, is not constant.\newline $\epsilon_2=1$, $\delta=1$, $\alpha=\mu=0.1$, $\gamma=0.99$, $\Bo=100$.}
\label{fig:hysteresis}
\end{figure}
\label{Hyst}

The experiments conducted in~\cite{VasseurMercierDauxois11} (this has also been reported in~\cite{MaasHaren06} for example) exhibit an interesting behavior, that cannot be seen in our simulations in figures~\ref{fig:FNLd1} and \ref{fig:FNLd2}. Indeed, with their setting, the dead-water phenomenon appears to be periodic in some sense, as the generated wave increases its amplitude, slows down the boat, breaks, and the process repeats. As a matter of fact, even when running simulations for much longer (around ten times) time intervals, the solution of~\eqref{eqn:Dfinal} does not display such a phenomenon. It is interesting to see that this is a major difference with the solutions of the KdV approximation, which generate periodically up-stream propagating waves, inducing an oscillation in the related wave resistance (see Figure~\ref{fig:critical} page~\pageref{fig:critical}). However, even with this model, the time-period of generation of these solitons ($\Delta T \approx 100$ in our situation) is too large to explain positively the phenomenon. 

This discrepancy may be due to our assumption of a constant velocity for the body, as the setting in~\cite{VasseurMercierDauxois11} constrains a constant force to move the boat. As a numerical experiment, we performed simulations with a Froude number $\Fr$ adjusted at each step: \[\Fr((n+1)\Delta t)\equiv \Fr(n\Delta t) - \Delta t\text{Cstt}_1 (C_W(n\Delta t)-\text{Cstt}_2).\] This roughly corresponds to a case where the acceleration of the ship is given by a constant force imposed to the ship, minus the resistance suffered by the body.\footnote{Of course, this modification is careless, and the produced scheme reflects only roughly what would be a model where a constant force is imposed to the body. To construct and rigorously justify a constant-force model would be much more difficult, as we explicitly use the constant-velocity hypothesis, as early as in~\eqref{eqn:vv1v2}.} We present in Figure~\ref{fig:hysteresis} the result of this computation.
A periodicity can clearly be seen in the wave resistance (with a time period of order $\Delta T\approx 10$). One can explain the phenomenon as follows. The velocity of the ship, suffering from the drag generated by the wave resistance, decreases down to a value where the generated wave resistance is very small. Released from its drag, the ship speeds up, and the phenomenon repeats periodically. This corresponds to the observations of~\cite{VasseurMercierDauxois11}, and is not recovered by constant velocity models (see~\cite{MilohTulinZilman93,MotyginKuznetsov97,YeungNguyen99}). 

\section{Weakly nonlinear models}\label{sec:weaklynonlinear}
The aim of this section is to introduce simple weakly nonlinear models for our problem, using the smallness assumptions of Regime~\ref{regimeSA}:
\[\mu \ll 1\ ; \qquad \epsilon_2\ = \ \O(\mu) , \quad \alpha\ = \ \O(\mu).\]
We recall that $\mu$ stands for the shallowness parameter, that is the ratio between the depth of the upper fluid layer and the internal wavelength. The parameter $\epsilon_2$ measures the magnitude of the deformation at the interface when compared with the upper fluid layer, and $\alpha$ is the ratio between the amplitudes of the deformations at the surface and at the interface. As discussed in Remark~\ref{rem:surfacetension}, the effect of the surface tension term will be relevant only if Kelvin-Helmholtz instabilities, located at high frequencies, appear. The models we obtain in this section do not give rise to these instabilities, and we therefore neglect the surface tension term, and set $\Bo^{-1} \ \equiv \ 0$ ($\Bo^{-1} \ = \ \O(\mu^2)$ would in fact suffice).

Starting with the strongly nonlinear model~\eqref{eqn:Sbegin}, we deduce easily a Boussinesq-type model and its symmetrized version. Then, using a classical WKB expansion, we obtain a rougher approximation, that consists in two {\em uncoupled Korteweg-de Vries equation, with a forcing term}. These two equations are studied in details, and numerically computed, in Section~\ref{sec:KdVana}. The symmetric Boussinesq-type model, as well as the KdV approximation, are justified thanks to convergence results (see Propositions~\ref{prop:ConvBouss} and~\ref{PROP:CONVKDV}, respectively).

\subsection{The Boussinesq-type models}\label{sec:Bouss}
Let us drop $\O(\mu^2)$ terms in system~\eqref{eqn:Sbegin}, using the assumptions of Regime~\ref{regimeSA}. One obtains straightforwardly the following system
\begin{equation}\label{eqn:vv1v2bouss}\left\{\begin{array}{l}
          (\partial_t -\Fr\partial_x)\zeta_2  \ + \ \dfrac1{\delta+\gamma}\partial_x v\ + \ \epsilon_2\dfrac{\delta^2-\gamma}{(\gamma+\delta)^2} \partial_x(\zeta_2 v) \ +\ \mu\dfrac{1+\gamma\delta}{3\delta(\delta+\gamma)^2} \partial_x^3 v \ =\ -\alpha\dfrac{\Fr\gamma}{\delta+\gamma}\dfrac{\dd}{\dd x} \zeta_1,\\ \\
	  (\partial_t -\Fr\partial_x) v \ +\  (\gamma+\delta) \partial_x \zeta_2 \ +\ \dfrac{\epsilon_2}{2}\dfrac{\delta^2-\gamma}{(\gamma+\delta)^2}\partial_x\left(| v|^2  \right) \ =\ 0.
         \end{array}
\right.\end{equation}
This Boussinesq-type system can be written in a compact form as
\begin{equation} \label{eqn:Bouss}
\partial_t U \ + \ A_0\partial_x U \ +\ \epsilon_2 A_1(U)\partial_x U\ -\ \mu A_2 \partial_x^3 U\  =\ \alpha b_0(x),\end{equation}
with $U \ = \ (\zeta_2,v)^T$, $b_0\ =\ -\Fr\frac\gamma{\delta+\gamma} (\frac{\dd}{\dd x} \zeta_1,0)^T$, and
\[A_0=\begin{pmatrix}
       -\Fr&\frac1{\delta+\gamma}  \\ \delta+\gamma & -\Fr
      \end{pmatrix},\quad
A_1(U)=\frac{\delta^2-\gamma}{(\gamma+\delta)^2}\begin{pmatrix}
       v&\zeta_2 \\ 0 &v
      \end{pmatrix},\quad
      A_2=\begin{pmatrix}
       0&-\frac{1+\gamma\delta}{3\delta(\delta+\gamma)^2} \\ 0 & 0
      \end{pmatrix}.
 \]
 Following the classical theory of hyperbolic quasilinear equations, our aim is now to obtain an appropriate symmetrizer of this system. 
  Let us define
 \begin{align*} S(U) \ &\equiv \ S_0\ +\ \epsilon_2 S_1(U)\ -\ \mu S_2 \partial_x^2 \\
  &= \ \begin{pmatrix}
       \delta+\gamma & 0 \\ 0 & \frac1{\delta+\gamma}
      \end{pmatrix} \ +\epsilon_2 \
\frac{\delta^2-\gamma}{(\gamma+\delta)^2}\begin{pmatrix}
       0&-v \\ -v &\zeta_2
      \end{pmatrix} \ -\mu\ 
      \begin{pmatrix}
       0 & 0 \\ 0 & -\frac{1+\gamma\delta}{3\delta(\delta+\gamma)^2} 
      \end{pmatrix}\ \partial_x^2 \ .
 \end{align*}
Multiplying~\eqref{eqn:Bouss} on the left by $S(U)-  \mu K  S_0  \partial_x^2$, and dropping the $\O(\mu^2)$ terms, leads to the following equivalent system
\begin{equation} \label{eqn:SymBouss}
 \left(\ S_0\ +\ \epsilon_2 S_1(U)\ -\ \mu (S_2 +K S_0) \partial_x^2\ \right)\partial_t U\ +\ \left(\ \Sigma_0\ +\ \epsilon_2 \Sigma_1(U)\ -\ \mu \Sigma_2\ \right)\partial_x U \ =\ \alpha b(x) ,
\end{equation}
with $b\ =\ -\Fr\gamma (\frac{\dd}{\dd x} \zeta_1,0)^T$, and
 \begin{align*}&\Sigma_0\ =\ \begin{pmatrix}
       -\Fr(\gamma+\delta) & 1 \\ 1 & \frac{-\Fr}{\gamma+\delta}
      \end{pmatrix}, \qquad
\Sigma_1(U)\ =\ \frac{\delta^2-\gamma}{\gamma+\delta}\begin{pmatrix}
       0&\zeta_2+\frac{\Fr\ v}{\gamma+\delta} \\ \zeta_2+\frac{\Fr\ v}{\gamma+\delta} &\frac{-\Fr\ \zeta_2}{\gamma+\delta}
      \end{pmatrix}, \\
      &\Sigma_2\ =\ -\frac{1+\gamma\delta}{3\delta(\delta+\gamma)} \begin{pmatrix}
       0&1 \\ 1 & 0
      \end{pmatrix} \ + \ K\ \Sigma_0.
 \end{align*}
 As we can see, the system~\eqref{eqn:SymBouss} is perfectly symmetric, and $K$ can be chosen so that $S_0$ and $S_2+K S_0$ are definite positive. With these properties, we are able to use energy methods in order to prove that~\eqref{eqn:SymBouss} is well-posed, and convergent with the full Euler system~\eqref{eqn:AdimEulerComplet}, at order $\O(\mu)$ up to times of order $\O(1/\mu)$. More precisely, one has
 \begin{Proposition}\label{prop:ConvBouss} 
 Let $s>3/2$, $ t_0\ge 9/2$ and $U=(\zeta_1,\zeta_2,\psi_1,\psi_2)$ be an {\em adapted} solution of the full Euler system~\eqref{eqn:AdimEulerComplet}, bounded in $W^{1,\infty}([0,T/\mu);H^{s+t_0})$. We define $V\equiv(\zeta_2,v)$ by 
 \[ v  \equiv\ \partial_x\left({\phi_2}\id{z=\epsilon_2\zeta_2} - \gamma{\phi_1}\id{z=\epsilon_2\zeta_2} \right)\ \equiv\ \partial_x\psi_2-\gamma H(\psi_1,\psi_2).\]  Moreover, let us assume that there exists a constant $C_0$ such that $\epsilon_2 \leq C_0 \mu$ and $\alpha \leq C_0\mu$.  
 Then there exists a constant ${C_1=C(\frac{1}{\gamma+\delta},\gamma+\delta,C_0)>0}$ such that if $\epsilon_2 \big\vert {V}\id{t=0}\big\vert_{H^{s+1}_\mu} \ \leq \ \frac{1}{C_1}$, there exists ${T>0}$, independent of $\mu$, and a unique solution ${V_B\in C^0([0,T/\mu);H^{s+1}_\mu)\cap C^1([0,T/\mu);H^{s}_\mu)}$, bounded in $W^{1,\infty}([0,T/\mu);H^{s}_\mu)$, of the Cauchy problem~\eqref{eqn:SymBouss} with ${{V_B}\id{t=0}={V}\id{t=0}}$. Moreover, one has for all $t\in[0,T/\mu)$, 
 \[\big|V-V_B\big|_{L^\infty([0,t] ; {H}^{s})}\ \leq\  \mu^2 t C,\]
with $C=C(\frac{1}{h_{\min}},\frac{1}{\gamma+\delta},\gamma+\delta,\big|V\big|_{W^{1,\infty} H^{s+t_0}},T)$. 
 \end{Proposition}
 Before starting with the proof, let us remark that the proposition is not empty only if there actually exists a family of solutions of~\eqref{eqn:AdimEulerComplet}, smooth and bounded in $W^{1,\infty}([0,T/\mu);H^{s+t_0})$. As discussed in Remark~\ref{rem:surfacetension} page~\pageref{rem:surfacetension}, this requires adding a surface tension term. However, this surface tension term is very small in practical cases, so that we can assume $\frac1{\Bo}=\O(\mu^2)$, and the result is obtained as in the proof presented below. 
 \begin{proof}
{\em Step 1: Well-posedness.} In order to prove the well-posedness of the symmetric system~\eqref{eqn:SymBouss}, we use techniques of~\cite{Duchene}. It is proved in Proposition~2.4 therein that systems of the form~\eqref{eqn:SymBouss} (with four equations instead of two and without the right-hand size), satisfying
 \begin{enumerate}
 \item The matrices $S_0$, $\Sigma_0$, $S_2$, $\Sigma_2$ are symmetric,
 \item $S_1(\cdot)$ and $\Sigma_1(\cdot)$ are linear mappings, and for all $U\in \RR^4$, $S_1(U)$ and $\Sigma_1(U)$ are symmetric,
 \item $S_0$ and $S_2$ are definite positive,
\end{enumerate}
 are well-posed and satisfy an energy estimate. The proof is easily adapted for the case of a non-zero right-hand side, and we briefly give the arguments here. The key point of the proof relies on a differential inequality, satisfied by the following energy (with ${\Lambda^s\equiv(1-\partial_x^2)^{s/2}}$)
 \[E_s(U)\ \equiv\ 1/2(S_0 \Lambda^s U,\Lambda^s U)+\epsilon_2/2(S_1(U) \Lambda^s U,\Lambda^s U)+\mu/2(S_2 \Lambda^s \partial_x U,\Lambda^s \partial_x U).\]
This energy is easily proved, thanks to the positiveness of $S_0$ and $S_2$ and using the smallness assumptions of Regime~\ref{regimeSA}, to be equivalent to the $\big\vert\cdot\big\vert_{H^{s+1}_\mu}$ norm, that is to say there exists $C_0>0$ such that
\[
 \frac{1}{C_0}\left(\big\vert U\big\vert_{H^{s}}^2 + \mu\big\vert U\big\vert_{H^{s+1}}^2\right)\ \leq \  E_s(U)\  \leq\  C_0 \left(\big\vert U\big\vert_{H^{s}}^2 + \mu\big\vert U\big\vert_{H^{s+1}}^2\right).\]
 Then, we prove that there exists ${C_1=C_1(\frac{1}{\gamma+\delta},\gamma+\delta)>0}$ such that if $\epsilon_2\big\vert U\big\vert_{H^{s+1}_\mu} \leq \frac1{C_1}$, then the operator defined by ${P(U,\partial_x)=S_0+\epsilon_2 S_1(U)-\mu S_2\partial_x^2:H^{s+1}\to H^{s-1}}$ is one-to-one and onto, and that $P(U,\partial_x)^{-1}\big(\Sigma_0+\epsilon_2 \Sigma_1(U)-\mu\Sigma_2\partial_x^2\big)$ is uniformly bounded $H^{s}_\epsilon\to H^{s}_\epsilon$, so that any solution $V$ of~\eqref{eqn:SymBouss} will satisfy the a priori estimate 
\[\big\vert\partial_t V\big\vert_{H^{s}_\epsilon} = \Big\vert P(V,\partial_x)^{-1}\Big(\big(\Sigma_0+\epsilon_2 \Sigma_1(V)-\mu\Sigma_2\partial_x^2\big)\partial_x V -\alpha b\Big)\Big\vert_{H^{s}_\epsilon}\leq C_2 (\big\vert V\big\vert_{H^{s+1}_\epsilon}+\alpha\big\vert b \big\vert_{H^s}) ,\]
 with $C_2$ independent of $\epsilon_2$ and $\mu$, as long as $\epsilon_2 \big\vert V \big\vert_{H^{s+1}_\mu} \ \leq \ \frac{1}{C_1}$.
  It follows that $E_s(V)$ satisfies
 \begin{align*}\label{dEs/dt}
 \frac{d}{dt}E_s(V)&= \ \epsilon_2/2 (S_1(\partial_t V)\Lambda^s V,\Lambda^s V)-\epsilon_2 ([\Lambda^s,S_1(V)]\partial_t V,\Lambda^s V)+\epsilon_2/2 ((\Sigma_1(\partial_x V) \Lambda^s V),\Lambda^s V)\\
 & \quad -\epsilon_2 ([\Lambda^s,\Sigma_1(V)]\partial_x V,\Lambda^s V)+\alpha(\Lambda^{s} b, \Lambda^{s} V)\\
 & \leq C_3\left(\mu \big|V\big|_{H^s}^2\big( \big|V\big|_{H^{s+1}_\mu}+\alpha\big\vert b \big\vert_{H^s}\big) +\alpha \big\vert b \big\vert_{H^s}\big|V\big|_{H^s}\right) \\ 
 & \leq C_4\left(\ \epsilon_2  E_s(V)^{3/2} +\alpha \big\vert b \big\vert_{H^s} E_s(V)^{1/2}\right).
\end{align*}
The last inequalities come from Cauchy-Schwarz inequality, Sobolev embeddings, Kato-Ponce commutator estimates and the above a priori estimate (see Appendix A of~\cite{Duchene}). The Gronwall-Bihari's Lemma allows to conclude that as long as $\epsilon_2 \big\vert V \big\vert_{H^{s+1}_\mu} \ \leq \ \frac{1}{C_0}$, one has
\begin{equation}\label{eqn:energy}
  \big|V\big|_{H^{s+1}_\mu} \ \leq \ C_0 \ E(R_s)^{1/2} \leq C_5 \alpha \big\vert f \big\vert_{H^s} \ t.
\end{equation}

Using the assumptions of Regime~\ref{regimeSA} ($\alpha=\O(\mu),\epsilon_2=\O(\mu)$), one can then follow the classical Friedrichs proof, and obtain the existence of $T(\gamma+\delta,\frac1{\gamma+\delta},(\epsilon_2 \big\vert V^0 \big\vert_{H^{s+1}_\mu})^{-1})>0$ and a solution $V$ of~\eqref{eqn:SymBouss}, defined over times $[0,T/\mu)$, such that $V\in C^0([0,T/\mu) ; H^{s+1}) \cap C^1([0,T/\mu) ; H^{s})$.

Finally, the uniqueness of the solution is obtained in the same way, applying the energy estimate to the difference of two solutions. Indeed, let $V_1, V_2 \in  C^0([0,T/\epsilon) ; H^{s+1}) \cap C^1([0,T/\epsilon) ; H^{s})$ be two solutions of the Cauchy problem~\eqref{eqn:SymBouss} with same initial value ${{V_1}\id{t=0}={V_2}\id{t=0}=V^0}$. The functions $V_1$ and $V_2$ are uniformly bounded thanks to~\eqref{eqn:energy}. One can immediately check that ${R\equiv V_1-V_2}$ satisfies
\begin{align} \label{eqn:EnDif}\Big(S_0+\epsilon_2 S_1(V_1)-\mu S_2\partial_x^2\Big)\partial_t \Lambda^s R+\Big(\Sigma_0+\epsilon_2 \Sigma_1(V_1)-\mu \Sigma_2 \partial_x^2\Big)\partial_x \Lambda^s R \nonumber \\
+\epsilon_2[\Lambda^s, S_1(V_1)]\partial_t R + \epsilon_2 [\Lambda^s, \Sigma_1(V_1)] \partial_x R=\epsilon_2 F,\end{align}
with $F=-\Lambda^s \Big(S_1(R)\partial_t V_2+\Sigma_1(R)\partial_x U_2\Big)$. Then, we can carry out the above method on $R$ with a modified right-hand side, and obtain the equivalent energy estimate
\[\frac{d}{dt}E_s(R) \leq \epsilon_2  C_6 ( \big|V_1 \big|_{H^{s}}+ \big|V_2 \big|_{H^{s}})E_s,\] 
with $C_6=C(\frac{1}{\gamma+\delta},\delta+\gamma,\big| U^0\big|_{H^{s+1}_\epsilon})$.
From Gronwall-Bihari's inequality, using the uniform boundedness of $V_1$ and $V_2$ on $[0,T/\epsilon)$, and since $E_s(R)\id{t=0}=0$, one has immediately $E_s(R)=0$ on $[0,T/\epsilon)$, and finally $V_1=V_2$.

\medskip

\noindent{\em Step 2: Consistency.} We prove that, assuming that $\alpha=\O(\mu)$ and $\epsilon_2 =\O(\mu)$, the full Euler system~\eqref{eqn:AdimEulerComplet} is consistent with the models~\eqref{eqn:Bouss} and~\eqref{eqn:SymBouss}, both at precision $\O(\mu^2)$ on $[0,T]$, $T>0$.

 Let $U\equiv(\zeta_1,\zeta_2,\psi_1,\psi_2)$ be a strong solution of~\eqref{eqn:AdimEulerComplet}, bounded in $W^{1,\infty}([0,T];H^{s+t_0})$ with $s>1$ and $t_0\ge 9/2$, and such that~\eqref{eqn:h} is satisfied with $\zeta_1(t,x) \ \equiv \ \zeta_1(x)$. The consistency result of Proposition~\ref{prop:ConsFullyNL} states that $(\zeta_2,v)$, with $v\equiv  \partial_x\psi_2-\gamma H(\psi_1,\psi_2)$, satisfies~\eqref{eqn:Sbegin}, up to $R_1=(r_1,r_2)^T\in L^\infty([0,T];H^{s})^2$, satisfying (for $i=1,2$)
\[ \big\vert r_i \big\vert_{L^\infty H^{s}} \  \leq \mu^2\ C_0\left(\frac{1}{h_{\min}}, \big\vert U\big\vert_{W^{1,\infty}H^{s+t_0}}\right).\]
It is the straightforward, as in the proof of Proposition~\ref{prop:ConsFullyRL}, to show that the terms dropped in~\eqref{eqn:vv1v2bouss}, and later in~\eqref{eqn:SymBouss}, are all bounded by $\mu^2\ C_0\left(\frac{1}{h_{\min}}, \big\vert U\big\vert_{W^{1,\infty}H^{s+4}}\right)$ in $L^\infty([0,T];H^{s})$ norm. Finally, $(\zeta_2,v)$ satisfies~\eqref{eqn:Bouss} and~\eqref{eqn:SymBouss} with modified right-hand sides, {\em i.e. }, respectively,
\[ \alpha \t b_0 \ = \ \alpha b_0 \ + \ \mu^2 f_0, \quad \text{and}\quad  \alpha \t b \ = \ \alpha b \ + \ \mu^2 f,\]
with $f_0,f$ uniformly bounded in $L^\infty([0,T];H^{s})$.

\medskip

\noindent{\em Step 3: Convergence.} The convergence estimates of the proposition follow easily from the calculations of {\em Step 1}, together with {\em Step 2}. Indeed, thanks to the consistency result, $V-V_B$ satisfy~\eqref{eqn:EnDif}, with modified right-hand side 
\[ \t F \ \equiv F + \mu^2 f,\qquad f\in L^\infty([0,T];H^{s}).\]
It follows that $E_s(V-V_B)$ satisfies 
\[ \frac{d}{dt}E_s(V-V_B)\leq \epsilon_2 C_6  E_s(V-V_B)+\mu^2 C_7  \big\vert f \big\vert_{H^s} E_s(V-V_B)^{1/2},\]
and the Gronwall-Bihari's Lemma leads to
\[\big|V-V_B\big|_{L^\infty([0,T/\mu) ; {H}^{s+1}_\mu)} \ \leq \ C_0 E_s(V-V_B)^{1/2} \ \leq\  \frac{C_7 \mu^2}{C_6 \epsilon_2 }\big\vert f \big\vert_{H^s}(e^{C_6 \epsilon_2 t}-1).\]
The convergence estimate of the proposition follow from $\epsilon_2 =\O(\mu)$.
\end{proof}

\subsection{The Korteweg-de Vries approximation} \label{sec:KdV}
In this section, we use a WKB (Wentzel-Kramers-Brillouin) approximation, in order to deduce from the symmetric Boussinesq-type system~\eqref{eqn:SymBouss}, an approximated model that consists in two uncoupled forced Korteweg-de Vries (fKdV) equations. This method has been used, for example, in~\cite{Duchene,LannesSaut06}, and is briefly presented in the following.

The idea is to look for an approximate solution of the Cauchy problem~\eqref{eqn:SymBouss} with initial data $U^0$, under the form
\[U_{\text{app}}(t,x)=U_0(\mu t,t,x)+\mu U_1(\mu t,t,x),\]
with the profiles $U_0(\tau,t,x)$ and $\mu U_1(\tau,t,x)$ satisfying ${U_0}\id{t=\tau=0}=U^0$ and ${U_1}\id{t=\tau=0}=0$.

Plugging the {\em Ansatz} into~\eqref{eqn:SymBouss} leads to the following equation
\begin{align}\label{eqn:AnsatzInBouss}
 (S_0\partial_t +\Sigma_0\partial_x) U_0+\mu S_0\partial_\tau U_0+\epsilon_2\big(S_1(U_0)\partial_t U_0+\Sigma_1(U_0)\partial_x U_0\big)-\mu\big(S_2\partial_x^2\partial_t U_0 +\Sigma_2\partial_x^3 U_0\big)  \nonumber \\ +\mu(S_0\partial_t +\Sigma_0\partial_x) U_1+\mu^2 R =0. 
\end{align}
We now deduce $U_0(\tau,t,x)$ and $U_1(\tau,t,x)$, by solving~\eqref{eqn:AnsatzInBouss} at each order.

\bigskip

\para{At order $\mathcal{O}(1)$} We solve
\begin{equation}\label{eqn:TRANSPORT}(\ S_0\partial_t\ +\ \Sigma_0\partial_x\ ) U_0\ =\ 0.\end{equation}
Let us define $\e_\pm\ \equiv\ \frac1{\sqrt2}( \pm \frac{1}{\sqrt{\gamma+\delta}}, \sqrt{\gamma+\delta})^T$.
One can check that the basis satisfies \[\e_i\cdot \Sigma_0 \e_j=c_i\delta_{i,j},\quad \text{ and }\quad \e_i\cdot S_0 \e_j=\delta_{i,j},\]
with $\delta_{i,j}$ the classical Kronecker delta symbol, and $c_\pm = \pm 1 - \Fr$. 

Therefore, when we define $u_\pm\equiv \e_\pm\cdot S_0 U_0$ (and hence $U_0=u_+\e_+ \ + \ u_- \e_-$), the scalar product of~\eqref{eqn:TRANSPORT} with $\e_\pm$ leads to 
\[(\ \partial_t \ +\ c_\pm\partial_x\ ) u_\pm\ =\ 0.\]
Finally, since $u_\pm$ satisfies a scalar transport equation, we use the notation
\begin{equation}
 \label{eqn:transport} u_\pm(\tau,t,x)\ =\ u_\pm(\tau,x-c_\pm t)\ =\ u_\pm(\tau,x_\pm),
\end{equation}
with initial data ${u_\pm}(0,x_\pm)\ =\ \e_\pm\cdot S_0 U^0(x_\pm)$.

\bigskip

\para{At order $\mathcal{O}(\mu)$} We solve
\begin{equation}\label{eqn:Order2} S_0\partial_\tau U_0+\frac{\epsilon_2}{\mu}\left(\Sigma_1(U_0)\partial_x U_0+S_1(U_0)\partial_t U_0\right)-\Sigma_2\partial_x^3 U_0-S_2\partial_x^2\partial_t U_0 +(S_0\partial_t +\Sigma_0\partial_x) U_1=\frac{\alpha}{\mu} b(x),\end{equation}
that we can split in
\begin{equation}\label{eqn:KdV0}\partial_\tau u_\pm\ +\ \lambda_\pm u_\pm\partial_{x_\pm} u_\pm\ +\ \nu_\pm \partial_{x_\pm}^3 u_\pm\ =\ \beta_\pm(x),\end{equation}
with $\lambda_\pm\ \equiv \ \frac{\epsilon_2}{\mu}\e_\pm\cdot(\Sigma_1-c_\pm S_1)(\e_\pm))\e_\pm$, $\nu_\pm \ \equiv \ \e_\pm\cdot (-\Sigma_2+c_\pm S_2)\e_\pm$ and $\beta_\pm\ \equiv \ \frac{\alpha}{\mu}\e_\pm\cdot b$; and on the other hand,
\begin{equation}\label{eqn:rest} (\partial_t +c_i\partial_x)\e_i\cdot S_0 U_1 + \sum_{(j,k)\neq(i,i)}\alpha_{ijk} u_k(\tau,x- c_k t)\partial_x u_j(\tau,x- c_j t)=\sum_{j\neq i} \beta_{ij}\partial_x^3 u_j(\tau,x- c_j t),\end{equation}
with $\alpha_{ijk}\ \equiv\ \e_i\cdot (\Sigma_1(\textbf{e} _k) -c_j S_1(\textbf{e} _k)) \e_j$ and $\beta_{ij}\ \equiv\ \e_i \cdot (\Sigma_2-c_j S_2)\e_j$.

It is clear that $u_i$ satisfies~\eqref{eqn:transport} and~\eqref{eqn:KdV0}, if and only if $u_i(\epsilon t,t,x)$ satisfies the Korteweg-de Vries equation:
\[
 \partial_t u_\pm\ +\ c_\pm\partial_x u_\pm\ +\ \mu \left(\ \lambda_\pm u_\pm\partial_x u_\pm\ +\ \nu_\pm \partial_x^3 u_\pm \ \right)\ =\ \mu\beta_\pm(x).
\]
Finally, simple calculations show that in our case, we can decompose 
\[\zeta_2 \ \equiv\ \eta_+ \ +\ \eta_- , \quad \text{with}\ \ \eta_\pm \ \equiv \ \pm \frac{1}{\sqrt{2(\gamma+\delta)}} u_\pm \] satisfying precisely the following KdV equation
\[\partial_t \eta_\pm \ +\ (-\Fr \pm 1) \partial_x \eta_\pm\ \pm\ \epsilon_2\frac32\frac{\delta^2-\gamma}{\gamma+\delta} \eta_\pm\partial_x \eta_\pm\ \pm\ \mu\frac16\frac{1+\gamma\delta}{\delta(\gamma+\delta)} \partial_x^3 \eta_\pm\ =\ -\alpha \Fr \gamma\frac{\dd}{\dd x} \zeta_1.\]
Unsurprisingly, we recover the KdV approximation with a flat rigid lid, when $\alpha \ = \ 0$ (see~\cite{Duchene} and references therein).

\bigskip

Using these forced Korteweg-de Vries equations as an approximation of the full problem is justified up to times of order $\O(1/\mu$) by the following proposition.
 \begin{Proposition}\label{PROP:CONVKDV} Let $s>1/2$, $ t_0\ge 5+5/2$ and $U=(\zeta_1,\zeta_2,\psi_1,\psi_2)$ be an {\em adapted} solution of the full Euler system~\eqref{eqn:AdimEulerComplet}, bounded in $W^{1,\infty}([0,T/\mu);H^{s+t_0})$. We define $V\equiv(\zeta_2,v)$ by 
 \[ v \ \equiv\  \partial_x\left({\phi_2}\id{z=\epsilon_2\zeta_2} - \gamma{\phi_1}\id{z=\epsilon_2\zeta_2} \right)\ \equiv\ \partial_x\psi_2-\gamma H(\psi_1,\psi_2).\] 
Then there exists $\eta_+$ and $\eta_-$, the two solutions of the following forced Korteweg-de Vries equation
\begin{equation}\label{eqn:KdV1}\partial_t \eta_\pm \ +\  (-\Fr \pm 1) \partial_x \eta_\pm \ \pm\  \epsilon_2\frac32\frac{\delta^2-\gamma}{\gamma+\delta} \eta_\pm\partial_x \eta_\pm \ \pm\ \mu\frac16\frac{1+\gamma\delta}{\delta(\gamma+\delta)} \partial_x^3 \eta_\pm \ = \ -\alpha \Fr \gamma\frac{\dd}{\dd x} \zeta_1(x),\end{equation}
with ${\eta_\pm}\id{t=0}=\frac12 \big(\zeta_2\pm\frac{1}{\gamma+\delta}v\big)\id{t=0}$. Moreover, if there exists a constant $C_0$ such that $\epsilon_2 \leq C_0 \mu$ and $\alpha \leq C_0\mu$, then one has for all $t\in[0,T/\mu)$,
\[\big|\zeta_2-(\eta_+ + \eta_-) \big|_{L^\infty([0,t] ; {H}^{s}} \ + \ \big|v-(\gamma+\delta)(\eta_+ - \eta_-) \big|_{L^\infty([0,t] ; {H}^{s})} \ \leq\ \mu \sqrt t C,\]
with $C=C(\frac{1}{h_{\min}},\frac{1}{\gamma+\delta},\gamma+\delta,\big|U\big|_{W^{1,\infty}H^{s+t_0}},C_0)$.

Additionally, if $V$ satisfies ${(1+x^2) V\id{t=0} \in H^{s+5}}$, then one has the improved estimate
\[\big|\zeta_2-(\eta_+ + \eta_-) \big|_{L^\infty([0,t] ; {H}^{s}} \ + \ \big|v-(\gamma+\delta)(\eta_+ - \eta_-) \big|_{L^\infty([0,t] ; {H}^{s})} \ \leq \ \epsilon C',\]
with $C'=C(\frac{1}{h_{\min}},\frac{1}{\gamma+\delta},\gamma+\delta,\big|U\big|_{W^{1,\infty} H^{s+t_0}},\big\vert(1+x^2) V\id{t=0}\big\vert_{H^{s+5}})$.
\end{Proposition}
The proposition is obtained by a simple adaptation of the techniques presented in~\cite{Duchene}, with additional forcing terms in the Korteweg-de Vries equations. The proof is given in Appendix~\ref{sec:proof}, for the sake of completeness.

\subsection{Analysis of the forced Korteweg-de Vries equation}\label{sec:KdVana}
The forced Korteweg-de Vries equation 
\begin{equation}\label{eqn:fKdV}\partial_t u\ +\ c \partial_x u\ +\ \varepsilon\lambda u \partial_x u \ +\ \varepsilon\nu \partial_x^3 u \ =\ \varepsilon\frac{\dd}{\dd x} f(x)\end{equation}
has attracted a lot of interests, especially in the framework of the one layer water wave problem (where a moving topography, or pressure, is the forcing term that generates waves). Of particular interest is the problem of the generation of solitons, that have first been numerically discovered by Wu and Wu~\cite{WuWu83}, and validated with experiments by Lee~\cite{Lee85}. Using the Boussinesq-type system or the KdV approximation, they found that starting with a zero initial data, the solution can generate periodically an infinite number of solitary waves. Numerous work have have then tackled this problem, including~\cite{Wu87,LeeYatesWu89,Shen92,Protopopov93}. It appeared that the Froude number (which is given by $1-c$ in~\eqref{eqn:fKdV}) is playing a predominant role in this phenomenon, as the generation of solitons only occurs for a narrow band of its values. One could roughly summarize the observations by the existence of $F_c>1$ such that
\begin{enumerate}
 \item if $\Fr>F_c$, then the flow approaches a steady state, symmetric and localized at the site of forcing;
 \item for $\Fr< F_c$, solitons are periodically generated at the site of forcing and radiated up-stream;
 \item the amplitude of the generated solitons goes to zero as $\Fr \to -\infty$. 
\end{enumerate}
The existence of steady solitary waves of~\eqref{eqn:fKdV}, and their stability is therefore essential. This issue has been studied for specific forcing terms by Camassa and Wu in~\cite{CamassaWu91a,CamassaWu91}, and for general forcing terms by Choi, Lin, Sun and Whang in~\cite{ChoiSunWhang08,ChoiLinSunEtAl10}. They prove that for $\Fr$ sufficiently large, there exists a unique small steady solution. Moreover, this solution is proved to be symmetric and localized at the site of forcing, and stable in $H^1(\RR)$ (in the Lyapunov sense). The behavior when the Froude number $\Fr$ approaches unity is more peculiar and depends on the sign of the forcing term. Subcritical Froude numbers are not studied.

In the framework of our analysis, the values of the coefficients $\lambda$ and $\nu$ depend on the parameters $\gamma$ and $\delta$, and their order of magnitude depends on $\epsilon_2$ and $\mu$ (the magnitude of the forcing term depends mostly on $\alpha$). We use the factor $\varepsilon$ to keep in mind that all of these coefficients are assumed to be small. More precisely, we recall
\[c_\pm\ =\ \pm1 - \Fr,\qquad \varepsilon\lambda_\pm\ =\ \pm\epsilon_2\frac32\frac{\delta^2-\gamma}{\gamma+\delta},\qquad \varepsilon\nu_\pm\ =\ \pm\mu\frac16\frac{1+\gamma\delta}{\delta(\gamma+\delta)}, \qquad \varepsilon f(x)\ =\  - \Fr \alpha\gamma\zeta_1(x).\]
When $\delta^2-\gamma \neq0$, a simple change of parameters allows to recover values of the one layer problem, and one obtains straightforwardly similar results. Our approach is quite different. In the following section, we prove that away from the critical speed ($c=0$), the solution of~\eqref{eqn:fKdV} with small initial data will remain small (in $H^s(\RR)$ norm) for long times, using smallness assumptions on the forcing terms and the coefficients $\varepsilon\lambda$ and $\varepsilon\nu$. Numerically, it appears that the solution converges locally towards a negative steady state in the supercritical case, and that small solitons are continuously generated otherwise. Then, in Section~\ref{sec:critical}, we study numerically the generation of up-stream propagating solitons for Froude numbers around the critical value.

The consequences of this study to the dead-water effect are the following:
\begin{enumerate}
 \item away from the critical Froude number, the drag suffered by the ship is always small;
 \item the peak of wave resistance occurs for Froude numbers just below the critical value;
 \item the time-period of the generation of solitons predicted by the KdV approximation is of the same order as the time-scale of relevance to the model, so that in cannot clearly be held responsible for the periodic aspect described among others in~\cite{VasseurMercierDauxois11}.
\end{enumerate}

\subsubsection{Non-critical Froude numbers}\label{sec:noncritical}
In the following proposition, we obtain an improved growth rate for the solution of the forced Korteweg-de Vries equation~\eqref{eqn:fKdV}, if the velocity coefficient $c$ is away from zero, when assuming smallness on the nonlinearity and dispersion coefficients $\lambda$ and $\nu$, and on the initial data and the forcing term.
This phenomenon is easily explained, when looking at the linear transport equation related to~\eqref{eqn:fKdV}:
\begin{equation}\label{eqn:ftransport}\partial_t v \ +\  c \partial_x v\ =\ \varepsilon\frac{\dd}{\dd x} f(x).\end{equation}
The Cauchy problem, with initial data $v\id{t=0}=\varepsilon u_0$, is solved by
\[v \ \equiv \ \varepsilon\ u_0(x-ct) \ + \ \frac{\varepsilon}{c}\ \Big(\ f(x)-f(x-ct)\ \Big).\]
The solution is therefore bounded for all times as soon as $c\neq0$, and small if the initial data and the forcing terms are small. We will obtain improved bound estimates on $u$ the solution of the forced KdV equation, by estimating the difference $\big\vert u-v\big\vert_{H^s}$.
\begin{Proposition}\label{prop:smallKdV}
Let $s>3/2$ and $u$ be the solution of 
\[ \partial_t u \ + \ c \partial_x u \ + \ \varepsilon\lambda u \partial_x u \ +\ \varepsilon\nu \partial_x^3 u\ =\ \varepsilon\frac{\dd}{\dd x} f(x),\]
such that $u\id{t=0}=\varepsilon u_0$. Let us assume that there exists $M>0$ such that 
\[ \left|\frac1c\right|,\ \left|\lambda\right|, \ \left|\nu\right|, \ \big\vert u_0\big\vert_{H^{s+3}}, \ \big\vert f\big\vert_{H^{s+3}} \ \leq M.\]
Then there exists $T(M,s)>0$ and $C=C(M,s)>0$, such that the function $u$ is bounded on $[0,T/\varepsilon]$ by
\[\big\vert u(t)\big\vert_{L^\infty([0,T/\varepsilon];H^s)} \ \leq\  C \varepsilon.\]
\end{Proposition}
\begin{proof}
Thanks to~\cite{BonaZhang96}, we know that the function $u$ exists and is unique. We define as above
\[v \ \equiv \ \varepsilon\ u_0(x-ct)\ + \ \frac{\varepsilon}c\ \Big(\ f(x)-f(x-ct)\ \Big),\]
the solution of the transport equation 
\[\partial_t v \ +\  c \partial_x v\ =\ \varepsilon\frac{\dd}{\dd x} f(x),\]
with same initial data $v\id{t=0}=\varepsilon u_0$. Since the function $v$ is uniformly bounded by
\[\big\vert v(t)\big\vert_{L^\infty(\RR;H^s)} \ \leq\  C \varepsilon,\]
the estimate of the proposition follows from the same estimate on the difference $r \ \equiv \ u-v$. It is easy to check that $r$ satisfies
\begin{equation}\label{eqn:KdVdiff}\partial_t r \ +\  c \partial_x r\ + \ \varepsilon\lambda (r+v) \partial_x (r+v) \ +\ \varepsilon\nu \partial_x^3 (r+v) = \ 0,\end{equation}
and $r\id{t=0}=0$. When multiplying~\eqref{eqn:KdVdiff} by $\Lambda^{2s}r$, and integrating, one obtains
\[\frac12\frac{d}{dt} (\Lambda^s r,\Lambda^s r) \ +\  c (\Lambda^s \partial_x r,\Lambda^s r)\ + \ \varepsilon\lambda (\Lambda^s \left((r+v) \partial_x (r+v)\right),\Lambda^s r) \ -\ \varepsilon\nu (\Lambda^s \partial_x^2 (r+v),\Lambda^s\partial_x r) = \ 0,\]
denoting $(f,g) \ \equiv \int_\RR fg$ the $L^2$-based inner product. It follows then
\[\frac12\frac{d}{dt} \left(\big\vert r \big\vert_{H^s}^2\right) \ = \ - \varepsilon\lambda(\Lambda^s \left((r+v) \partial_x (r+v)\right),\Lambda^s r) \ - \ \varepsilon\nu (\Lambda^s \partial_x^3 v,\Lambda^s r).\]
Some of the terms of the right-hand side are straightforwardly estimated using Cauchy-Schwarz inequality, and the algebraic properties of $H^s(\RR)$, $s>1/2$:
\begin{align*}
\big\vert (\Lambda^s \partial_x^3 v,\Lambda^s r)\big\vert & \leq  \ \big\vert v \big\vert_{H^{s+3}}\big\vert r \big\vert_{H^s},\\
\big\vert (\Lambda^s (v\partial_x v),\Lambda^s r)\big\vert & \leq  \ C_0(s) \big\vert v \big\vert_{H^{s+1}}\big\vert v \big\vert_{H^{s}}\big\vert r \big\vert_{H^s},\\
\big\vert (\Lambda^s (r\partial_x v),\Lambda^s r) \big\vert & \leq  \ C_0(s) \big\vert v \big\vert_{H^{s+1}}\big\vert r \big\vert_{H^s}^2.
\end{align*}
As for the remaining terms, we integrate by part and obtain
\begin{align*}
 (\Lambda^s (r\partial_x r),\Lambda^s r) & = \ -\frac12( \partial_x r\Lambda^s r,\Lambda^s r)+([\Lambda^s,r]\partial_x r, \Lambda^s r),\\
 (\Lambda^s (v\partial_x r),\Lambda^s r) & = \ -\frac12( \partial_x v\Lambda^s r,\Lambda^s r)+([\Lambda^s,v]\partial_x r, \Lambda^s r),
\end{align*}
with $[T,f]$ the commutator defined by $[T,f]g\equiv T(fg)-fT(g)$. Since for $f\in H^s$, $s>3/2$, one has $\big\vert \partial_x f \big\vert_{L^\infty}\leq \big\vert f \big\vert_{H^s}$, one can use the classical Kato-Ponce Lemma~\cite{KatoPonce88} to deduce the following estimates
\begin{align*}
\big\vert (\Lambda^s (v\partial_x r),\Lambda^s r)\big\vert & \leq  \ C_0(s)\big\vert v \big\vert_{H^s}\big\vert r \big\vert_{H^s}^2,\\
\big\vert (\Lambda^s (r\partial_x r),\Lambda^s r) \big\vert & \leq  \ C_0(s)\big\vert r \big\vert_{H^s}^3.
\end{align*}
Altogether, it follows that $ \big\vert r \big\vert_{H^s}$ satisfies the differential inequality
\[\frac12\frac{d}{dt} \left(\big\vert r \big\vert_{H^s}^2\right) \ \leq \ C_0(s) \varepsilon\lambda \left(\big\vert v \big\vert_{H^{s+1}}^2\big\vert r \big\vert_{H^s}+\big\vert v \big\vert_{H^{s+1}}\big\vert r \big\vert_{H^s}^2 + \big\vert r \big\vert_{H^s}^3\right)+\varepsilon\nu \big\vert v \big\vert_{H^{s+3}}\big\vert r \big\vert_{H^s},\]
with $C_0(s)$ a constant depending only on the parameter $s>3/2$. Now, using the assumptions of the proposition, it is straightforward to see that one can rewrite
\[\frac{d}{dt} \big\vert r \big\vert_{H^s} \ \leq \ C_0\varepsilon^2 \ + \ \ C_0(s) \varepsilon \left(\varepsilon^2 + \varepsilon\big\vert r \big\vert_{H^s} + \big\vert r \big\vert_{H^s}^2\right).\]
Now, let us define $T^\star$ the maximum time such that $\big\vert r \big\vert_{H^s} \leq \varepsilon$ on $[0,T^\star]$ (this is true at $t=0$, so that $T^\star>0$ by a continuity argument). Then we deduce from the above inequality, and Gronwall's Lemma, that for all $t\in[0,T^\star]$, 
\[ \big\vert r \big\vert_{H^s} \ \leq \  C(s)\varepsilon^2 t,\]
with a constant $C(s)$ depending only on the parameter $s>3/2$. From this very a priori estimate, we deduce that $T^\star \geq (C(s)\varepsilon)^{-1}$, and the proposition is proved.
\end{proof}
\begin{Remark}
 One of the immediate consequences of Proposition~\ref{prop:smallKdV} to our problem is that, in the decomposition of the KdV approximation (Proposition~\ref{PROP:CONVKDV}), the function $\eta_-$ remains small on the time interval where the KdV approximation is a relevant model. Indeed, as $c_-\equiv-1-\Fr<-1$, then if ${\eta_-}\id{t=0}=0$ (small in $H^{s+3}$ would suffice), then for $t\in[0,T/\mu)$, one has 
 \[\big\vert \eta_-(t)\big\vert_{H^s} \ \leq\  C\Big(\frac{1}{\gamma+\delta},\gamma+\delta,\big\vert\zeta_1\big\vert_{H^{s+3}}\Big) \mu.\]
 Therefore, we focus in the following on $\eta_+$, the solution of
 \begin{equation}\label{eqn:KdV+}\partial_t \eta_+ \ +\ (1-
\Fr) \partial_x \eta_+ \ +\ \epsilon_2\frac32\frac{\delta^2-\gamma}{\gamma+\delta} \eta_+\partial_x \eta_+ \ +\ \mu\frac16\frac{1+\gamma\delta}{\delta(\gamma+\delta)} \partial_x^3 \eta_+ \ =\ -\alpha \Fr \gamma\frac{\dd}{\dd x} \zeta_1.\end{equation}
 \end{Remark}
In Figures~\ref{fig:subcritical} and~\ref{fig:supercritical}, we compute the solution of~\eqref{eqn:KdV+}, with zero initial data, in the supercritical ($\Fr=1.5$) and subcritical ($\Fr=0.5$) cases. For each of the figures, we plot the flow, depending on space and time variables, as well as the evolution of the related wave resistance coefficient $C_W$, calculated with formula~\eqref{eqn:WaveResistanceSA} page~\pageref{eqn:WaveResistanceSA}. We use two values for the depth ratio: $\delta\in\{5/12,12/5\}$.

\begin{figure}[hptb]
\subfigure[$\delta=5/12$.]{\includegraphics [width=0.49\textwidth]{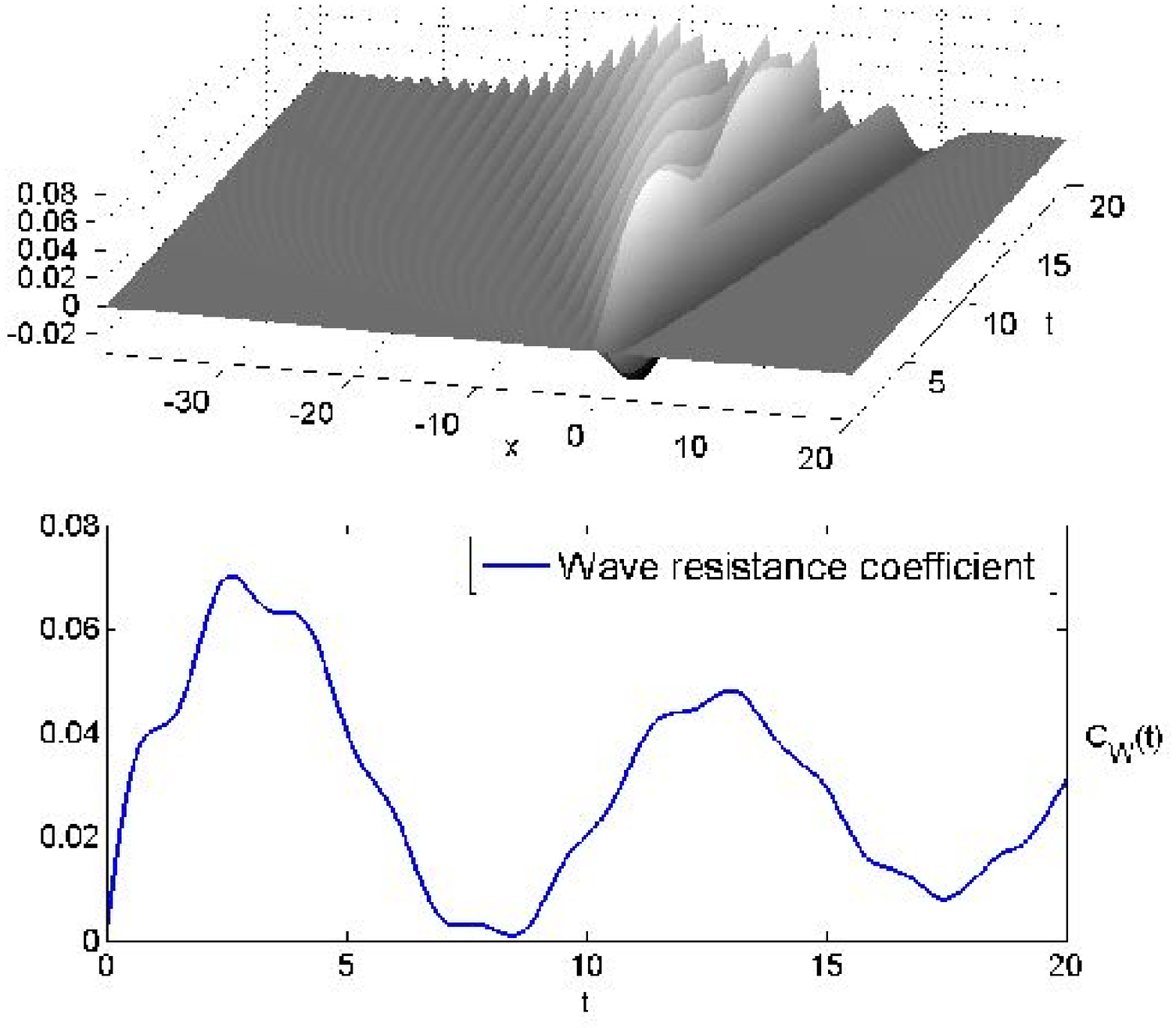}} 
\subfigure[$\delta=12/5$.]{\includegraphics [width=0.49\textwidth]{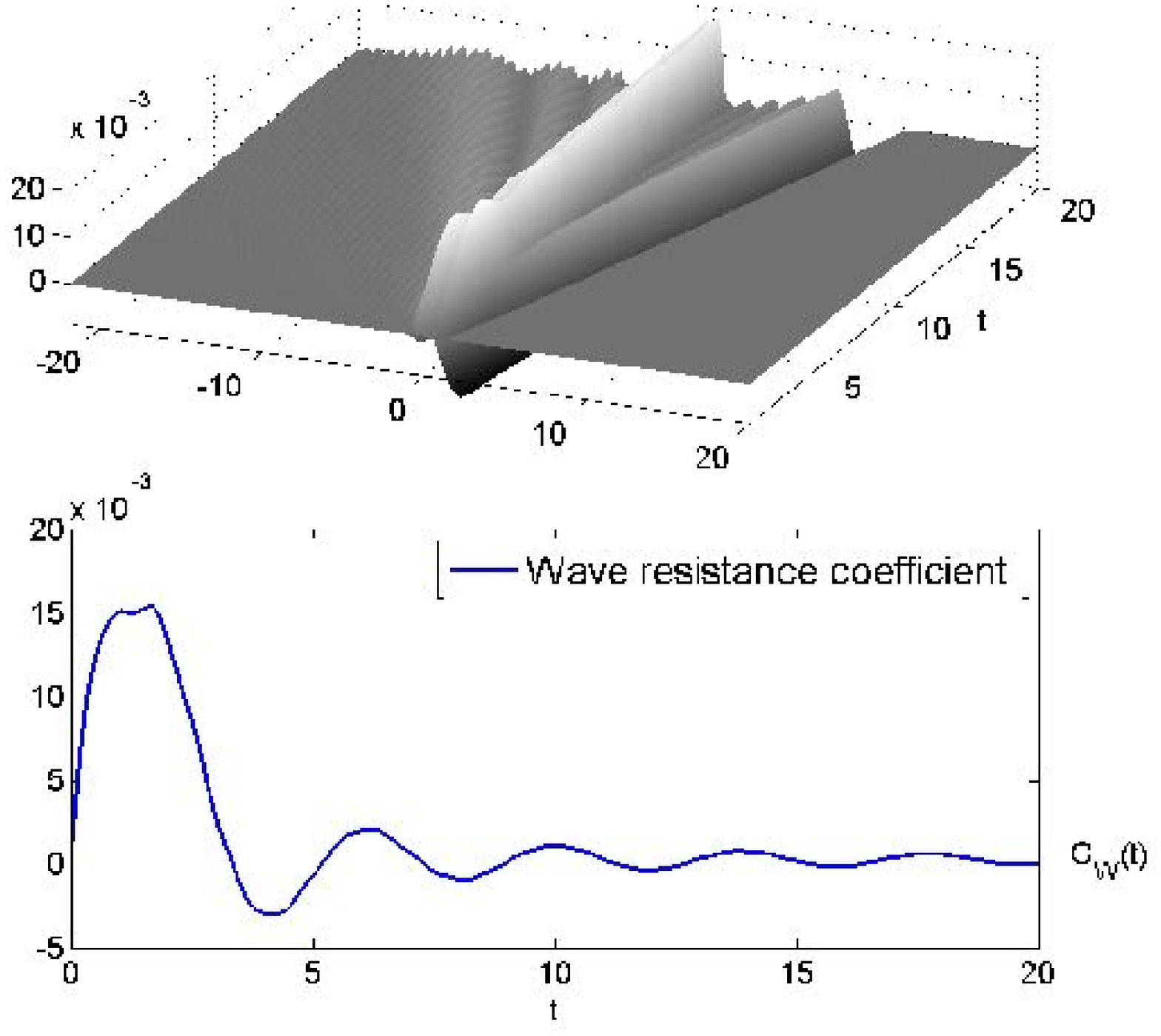}} 
 \caption{Subcritical flow: $\Fr=0.5$. $\alpha=\epsilon_2=\mu=0.1$, $\gamma=0.99$.}
\label{fig:subcritical}
\end{figure}
\begin{figure}[!hptb]
\subfigure[$\delta=5/12$.]{\includegraphics [width=0.49\textwidth]{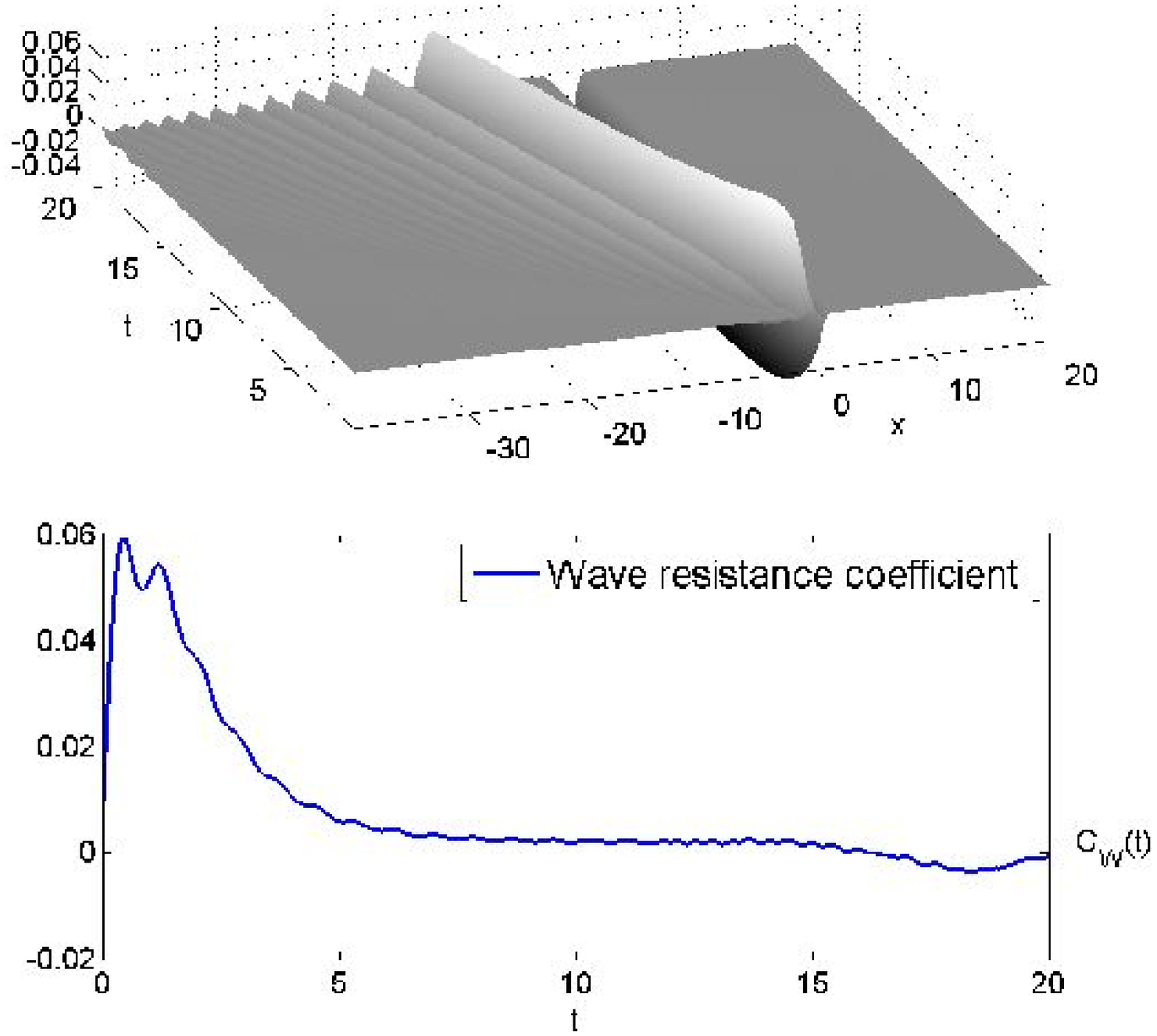}} 
\subfigure[$\delta=12/5$.]{\includegraphics [width=0.49\textwidth]{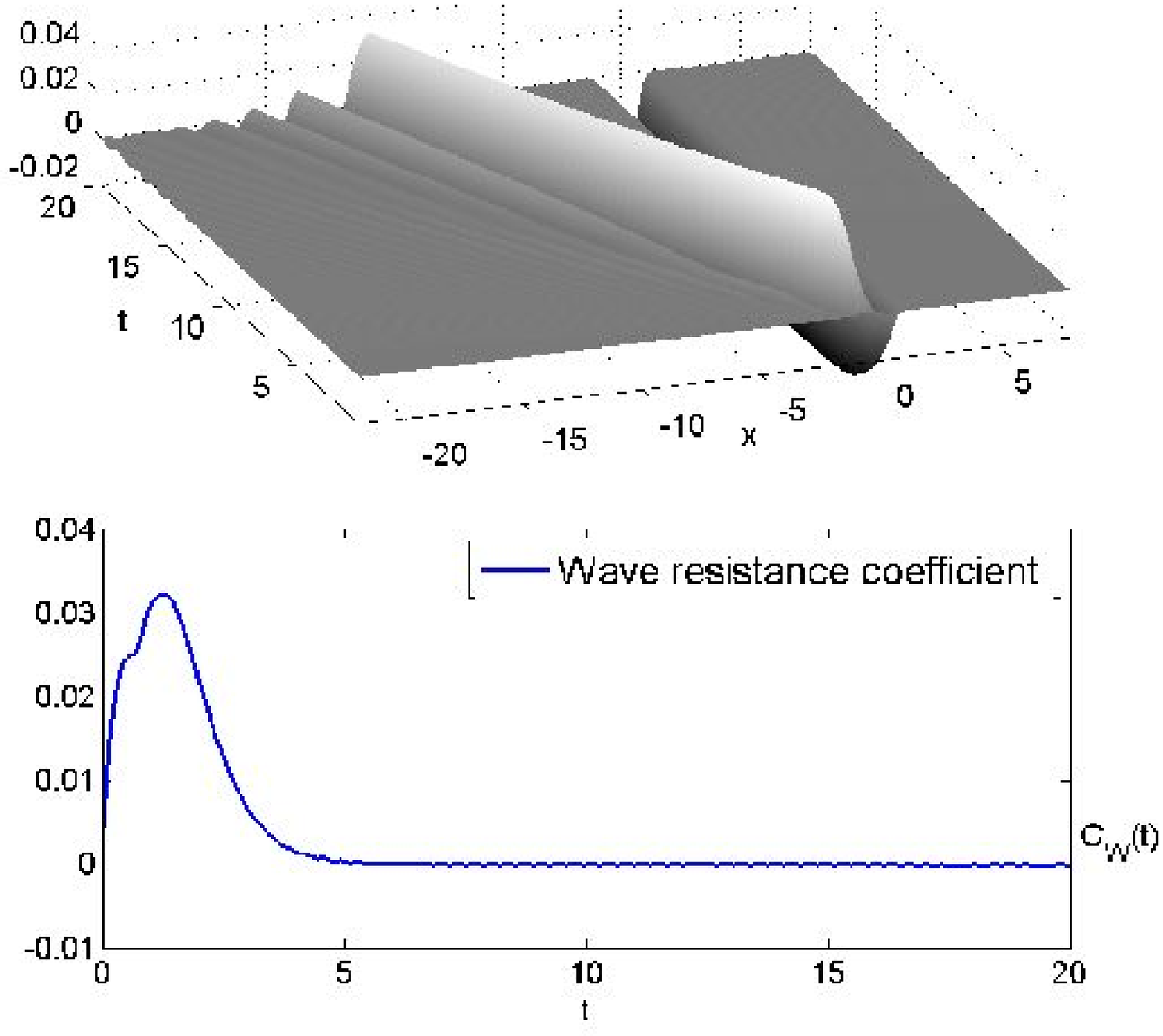}} 
 \caption{Supercritical flow: $\Fr=1.5$. $\alpha=\epsilon_2=\mu=0.1$, $\gamma=0.99$.}
\label{fig:supercritical}
\end{figure}

As we are away from the critical Froude number, the amplitudes of the generated waves, as well as the magnitude of the wave resistance coefficient, are small (at most of the order $\O(\mu)$). One remarks that the behavior of the flow is roughly independent of the value of the depth ratio $\delta$. This is easily explained, as the quadratic nonlinearities do not play an important role when the deformation is small (as they are of order $\O(\mu^3)$). Therefore, the variations of $\delta$ are mostly seen through a variation of the size of the dispersion coefficient $\nu$, which do not change the behavior pattern of the solution.

On the contrary, the Froude number parameter has an essential role on the behavior of the solution. One sees that for subcritical flows (Figure~\ref{fig:subcritical}), up-stream propagating solitary waves are continuously generated, as a widening oscillatory tail is generated at the down-stream propagating area. The former will generate most of the wave resistance, as they overtake the body. However, the drag remains small, and possibly tends to zero in the long time limit. The behavior of supercritical flows (Figure~\ref{fig:supercritical}) is quite different. A down-stream propagating solitary wave, with a small oscillatory tail, is generated and travels at constant velocity $c_0\approx-1/2$. It remains a symmetric, depression wave at the location of the body, that does not generate any wave resistance. Unsurprisingly, this corresponds to the observations of the one-layer theory in the supercritical case. 

\subsubsection{Critical Froude numbers}\label{sec:critical}
\begin{figure}[ht]\centering
\subfigure[$\alpha=\epsilon_2=\mu=0.1$, $\delta=5/12$, $\gamma=0.99$.]{\includegraphics [width=0.4\textwidth]{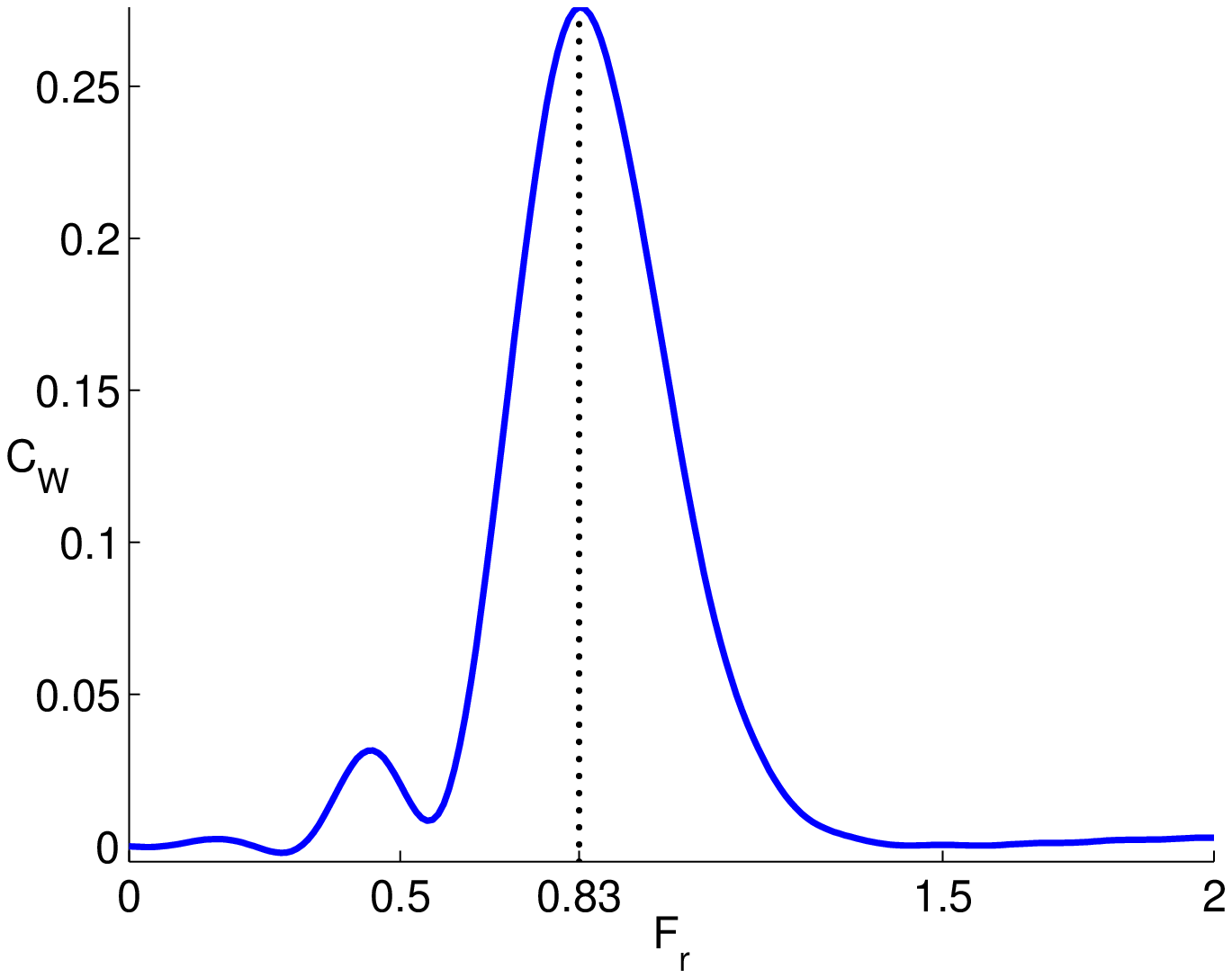}} \hspace{1cm}
\subfigure[$\alpha=\epsilon_2=\mu=0.1$, $\delta=12/5$, $\gamma=0.99$.]{\includegraphics [width=0.4\textwidth]{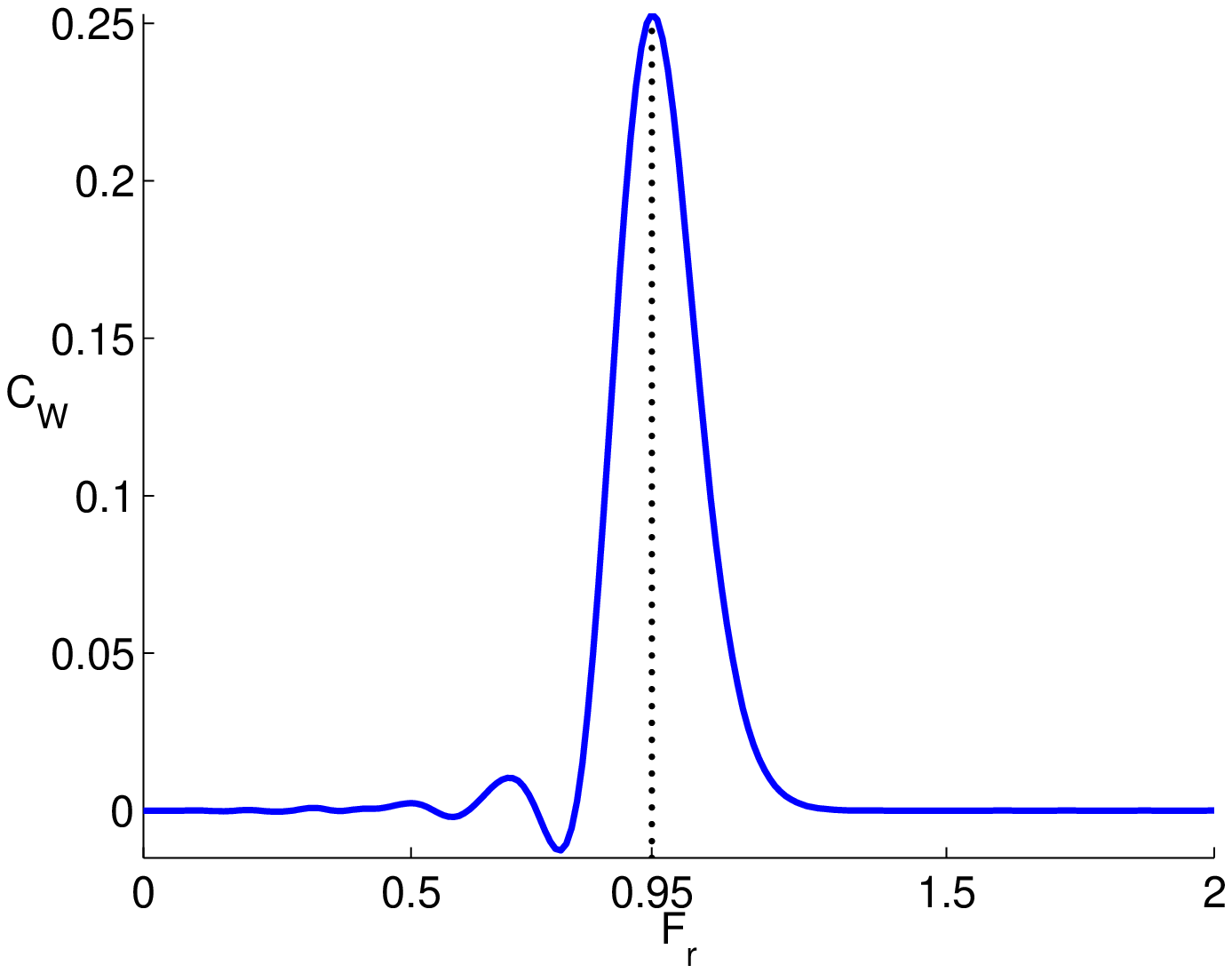}}  \\
\subfigure[$\alpha=\epsilon_2=\mu=0.01$, $\delta=5/12$, $\gamma=0.99$.]{\includegraphics [width=0.4\textwidth]{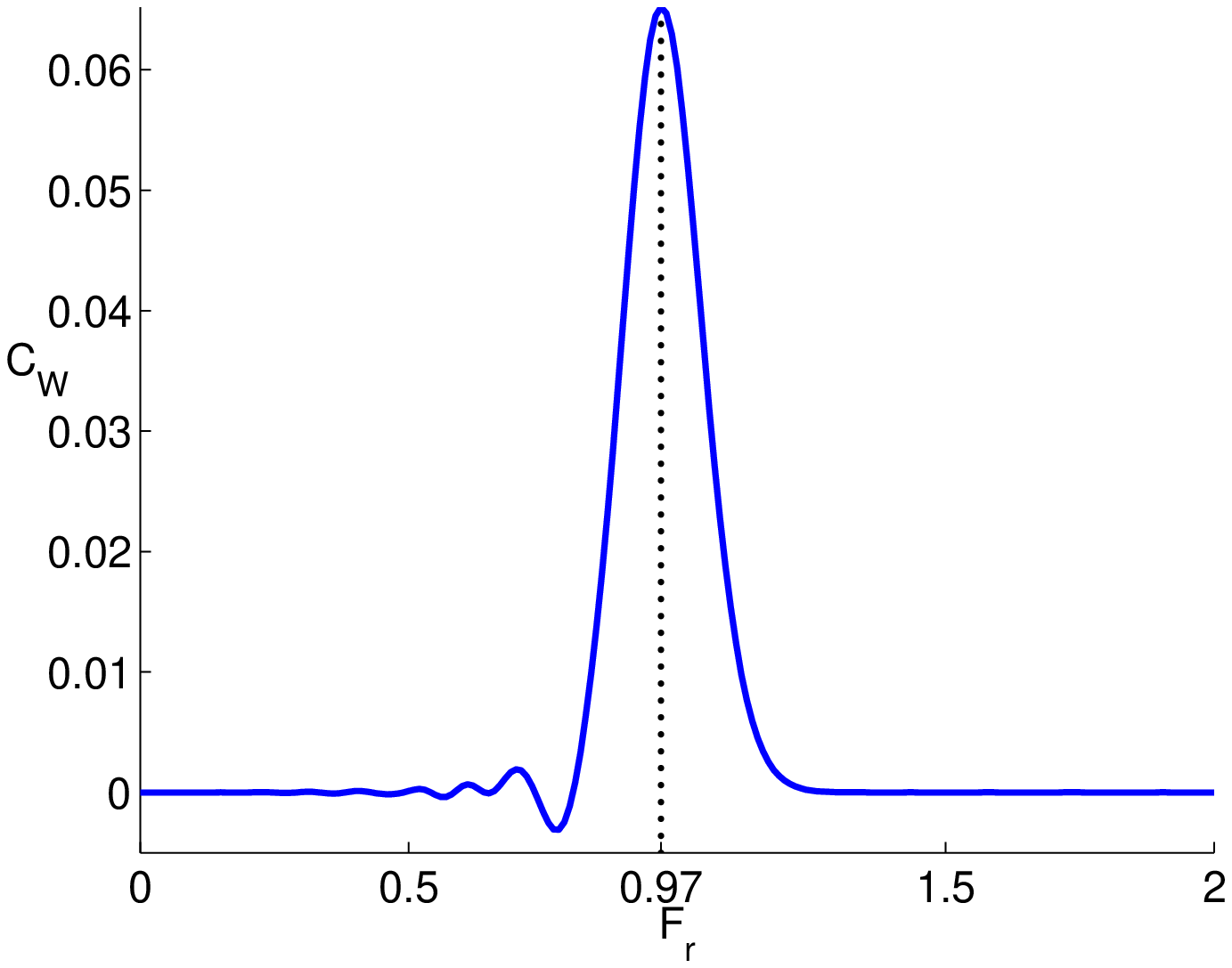}} \hspace{1cm}
\subfigure[$\alpha=\epsilon_2=\mu=0.01$, $\delta=12/5$, $\gamma=0.99$.]{\includegraphics [width=0.4\textwidth]{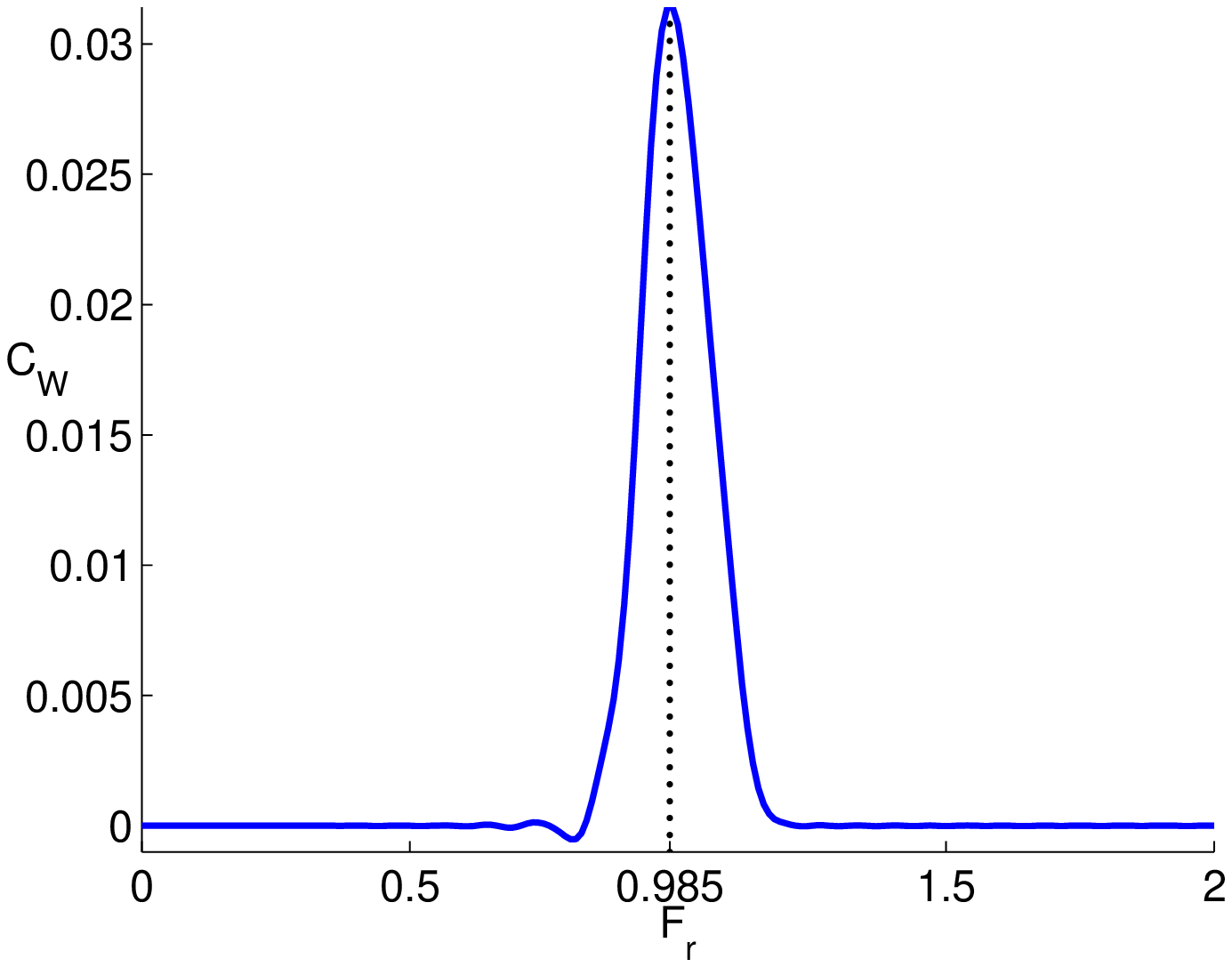}}
 \caption{Dependence of the wave resistance coefficient on the Froude number, at time $T=10$.}
\label{fig:CW(Fr)}
\end{figure}

As we have seen in the previous section, the solution of the fKdV equation, and therefore the wave resistance coefficient~\eqref{eqn:WaveResistanceSA}, are small when the Froude number is away from its critical number $\Fr=1$. We present in Figure~\ref{fig:CW(Fr)} the behavior of the wave resistance at time $t=10$, depending on the Froude number, for different values of the depth ratio $\delta$, and different values of the shallowness parameter $\mu$.

It appears that the maximum of the critical wave is obtained below the critical speed, at approximate value $\Fr \ \approx\ 1-\text{Cstt}\ \lambda_+\ =\ 1-\text{Cstt}\ \mu\frac{1+\gamma\delta}{\gamma+\delta}$. The fact that the maximum peak is slightly subcritical has been obtained using the linear theory in~\cite{MilohTulinZilman93}, and explained as a result of dispersion effects (without the nonlinear and dispersion terms, the peak is infinite and obtained at $\Fr = 1$). This effect is therefore preserved in the nonlinear theory.

In Figure~\ref{fig:critical}, we compute the solution of~\eqref{eqn:KdV+}, with zero initial data, in the critical case ($\Fr=1$). Again, we plot the flow, depending on space and time variables, as well as the wave resistance coefficient, and we set $\delta\in\{5/12,12/5\}$. It is known, since the work of Wu~\cite{Wu87}, that the forced KdV equation can generate up-stream propagating solitary waves. Our simulations exhibit that behavior, and show that the depth ratio $\delta$ now plays an essential role, as it determines the sign of the coefficient of nonlinearity. Indeed, even if up-stream propagating waves are generated for both of the values of $\delta$, these waves are of elevation if $\delta^2>\gamma$, and of depression in the opposite case.

\begin{figure}[hptb]
\subfigure[$\delta=5/12$.]{\includegraphics [width=0.49\textwidth]{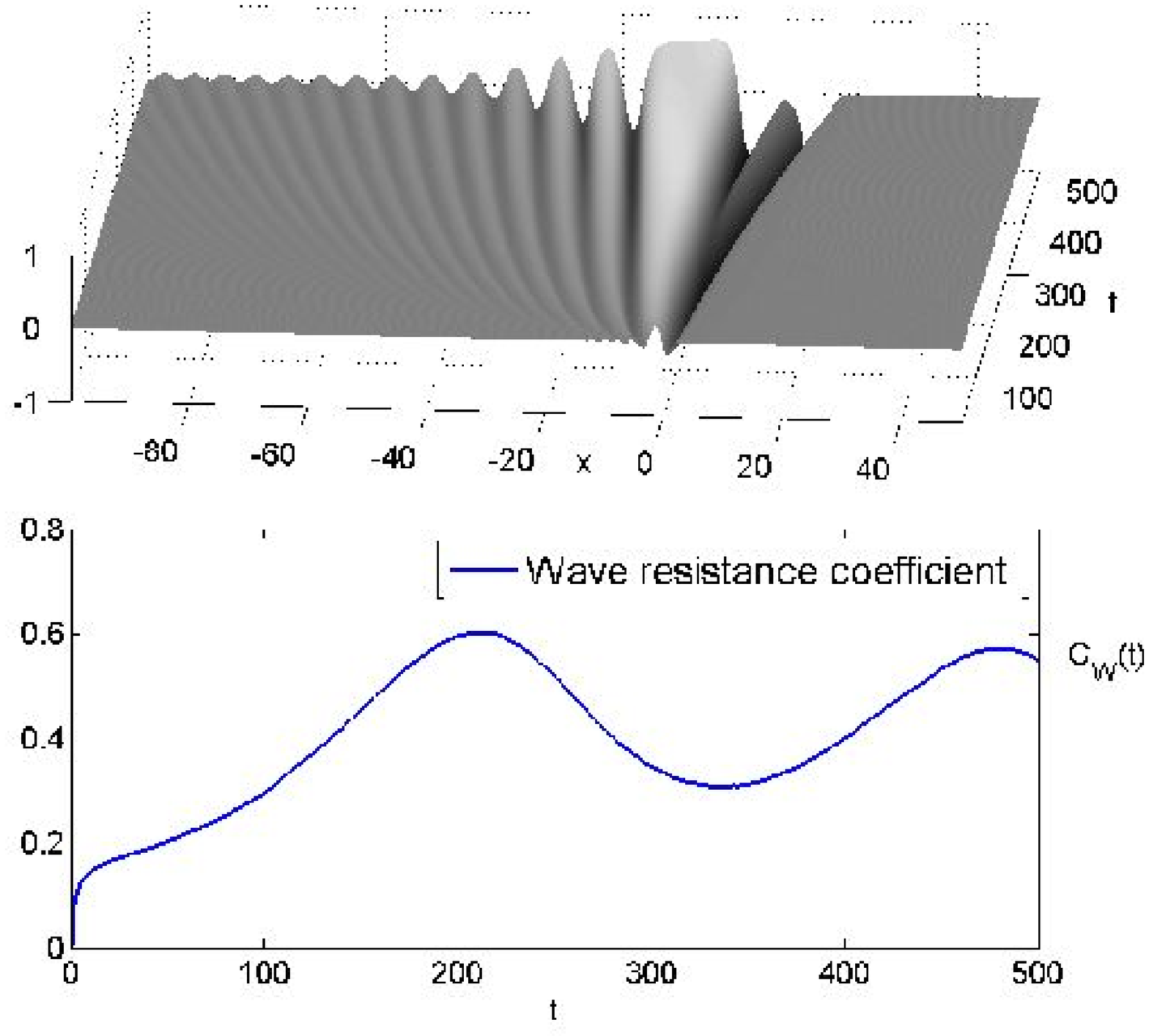}} 
\subfigure[$\delta=12/5$.]{\includegraphics [width=0.49\textwidth]{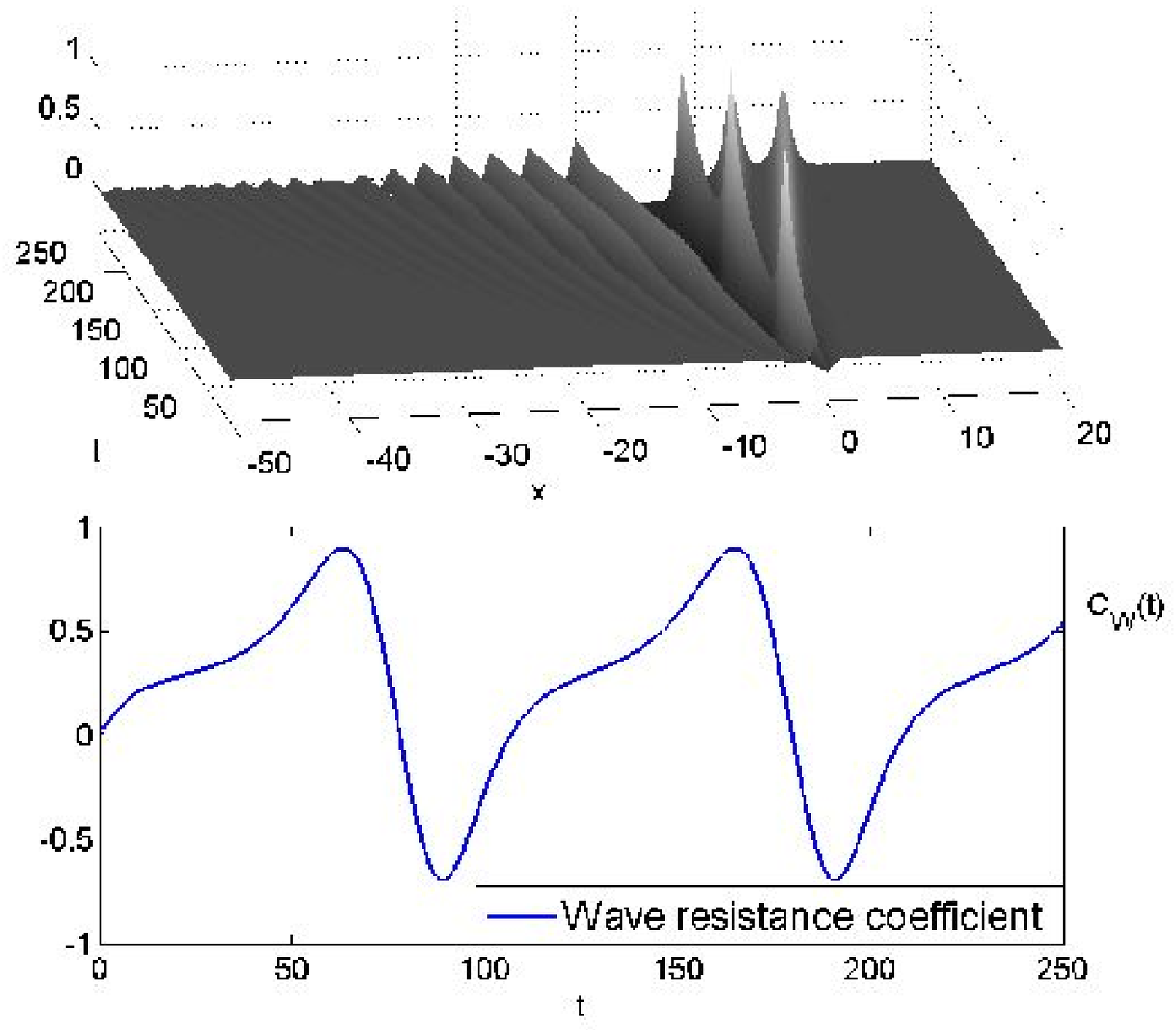}} 
 \caption{Critical flow: $\Fr=1$. $\alpha=\epsilon_2=\mu=0.1$, $\gamma=0.99$.}
\label{fig:critical}
\end{figure}

The continuous generation of up-stream propagating waves induce a periodic aspect to the wave resistance, that recalls the hysteretic behavior of the dead-water phenomenon recorded in~\cite{VasseurMercierDauxois11}. However, we do not believe that these waves are the cause of the periodic behavior recorded in \cite{VasseurMercierDauxois11}, based on the two following facts.
\begin{enumerate}
 \item The time period of the generation of waves $\Delta T$ is very large. A simple scaling argument shows that $\Delta T \varpropto 1/\mu$ if $c=0$ and $\lambda,\alpha,\nu\varpropto\mu$, and the proportionality constant appears to be big (see also scaling arguments in~\cite{Wu87}). By Proposition~\ref{PROP:CONVKDV}, the KdV approximation is a valid model only up to times of order $\O(1/\mu)$, so that nothing shows that the solutions of the full Euler equations~\eqref{eqn:AdimEulerComplet} follow this behavior. It is possible that the time-range of validity of the model is shorter than one time period of generation of up-stream propagating waves.
 \item As a matter of fact, the solutions of the strongly nonlinear model~\eqref{eqn:Dfinal} do not exhibit such a phenomenon, as the generated internal wave stays in the trail of the body, and never overtakes the ship, even for long times. 
\end{enumerate}
In Figure~\ref{fig:hysteresis} page~\pageref{fig:hysteresis}, we present the results of a simulation where the velocity of the ship is affected by the generated wave resistance. A similar periodic behavior is exhibited, with a more reasonable time period ($\Delta T \approx 10$). What is more, in this case, the up-stream propagating solitons vanish as soon as they overtake and slow down the body. This phenomenon corresponds to the observations of~\cite{VasseurMercierDauxois11}, but cannot be seen in Figure~\ref{fig:critical}. This is why we believe that constant-velocity models do not accurately predict the hysteretic aspect of the dead-water phenomenon, and that only a constant-force model would be able to recover this behavior (see also discussion in~\cite{VasseurMercierDauxois11}, section~1.3).

\section{Overview of results and discussion}\label{sec:conclusion}
We have presented several nonlinear models describing the behavior of a system composed of two homogeneous fluids, when a body moves at the surface with constant velocity. Our asymptotic models are obtained through smallness assumptions on parameters of the system (Regime~\ref{regimeRL} and~\ref{regimeSA}, see page~\pageref{regimeRL}). Contrarily to (linear) models existing in the literature, our models are valid for large (Regime~\ref{regimeRL}) and moderate (Regime~\ref{regimeSA}) amplitude waves. They are justified by consistency results (Proposition~\ref{prop:ConsFullyRL}) or convergence results (Propositions~\ref{prop:ConvBouss} and~\ref{PROP:CONVKDV}).

These models duly recover the key features of the dead-water phenomenon, described in details in~\cite{VasseurMercierDauxois11}, namely:
\begin{enumerate}
 \item transverse internal waves are generated at the rear of a body while moving at the surface;
 \item the body suffers from a positive drag when an internal elevation wave is located at its stern;
 \item this effect is strong only near critical Froude numbers;
 \item the maximum peak of the drag is reached at slightly subcritical values.
\end{enumerate}

Moreover, we numerically computed our models using several ratios between the depth of the two layers of fluid. It appears that when the upper layer is thicker than the lower layer, the generated wave is able to reach larger amplitudes, and the wave resistance suffered by the body is stronger than in the opposite case (see figures~\ref{fig:FNLd1},~\ref{fig:FNLd2} and~\ref{fig:critical}). This aspect have never been expressed before, as far as we know.

One important characteristic of the dead-water phenomenon, investigated in~\cite{VasseurMercierDauxois11}, is the hysteretic behavior of the system, and is not predicted by our models. This is due to the fact that our analysis assumes the velocity of the body to be constant, whereas a constant force was imposed to the body in their experimental setting. It would be of great interest to construct and justify models in the constant-force setting, in order to recover the complex behavior of the system in that configuration.

Finally, our work is limited to the two-dimensional configuration, that is horizontal dimension $d=1$. Extending the study to the three-dimensional configuration would allow to observe divergent waves generated by the moving body, as well as the transverse ones~\cite{MilohTulinZilman93,YeungNguyen99}.

\appendix

\section{Derivation of the Green-Naghdi type model}\label{sec:fullyscheme}
In this section, we derive the strongly nonlinear model~\eqref{eqn:Sbegin}. This model only requires the shallow water assumption
\[\mu \ \ll \ 1,\]
and is a convenient intermediate step to construct the models of Regimes~\ref{regimeRL} and~\ref{regimeSA} used throughout the paper. The full Euler system is proved to be consistent with our model at order $\O(\mu^2)$, in Proposition~\ref{prop:ConsFullyNL} below.

\medskip

Let us first plug the expansions of Lemma~\ref{Lem:extendVD1} into~\eqref{eqn:AdimEulerComplet}, and drop $\O(\mu^2)$ terms. One can easily check that we obtain the approximate system 
\begin{equation}\label{eqn:S1} \left\{ \begin{array}{l}
-\alpha \Fr\frac{\dd}{\dd x} \zeta_1 + \partial_x  (h_1\partial_x\psi_1)+\partial_x  (h_2\partial_x\psi_2)\ =\ \mu \partial_x\Big(\R_1(\partial_x\psi_1,\partial_x\psi_2))  \Big),\\ \\
(\partial_{ t}-\Fr\partial_x) \zeta_2 +\partial_x  (h_2\partial_x\psi_2)\ =\ \mu \partial_x\T[h_2,0]\partial_x\psi_2 , \\ \\
(\partial_{ t}-\Fr\partial_x)\left( \partial_x\psi_2-\gamma  \partial_x\psi_1\right) + (\gamma+\delta)\partial_x \zeta_2 +\dfrac{\epsilon_2}{2}\partial_x\left( |\partial_x\psi_2|^2 -\gamma|\partial_x\psi_1|^2\right) \qquad \qquad \\ \hfill= \ \mu\Big(\gamma(\partial_{ t}-\Fr\partial_x)\partial_x \H(\partial_x\psi_1,\partial_x\psi_2) +\epsilon_2\partial_x \R_2[\partial_x\psi_1,\partial_x\psi_2] \Big) \\ \hfill + \ \dfrac{1}{\Bo} \partial_x^2\bigg(\dfrac{\partial_x\zeta_2}{\sqrt{1+\mu\epsilon_2^2|\partial_x\zeta_2|^2}}\bigg),
\end{array} \right. \end{equation}
where we have used the following notations:
\begin{align*}
\T[h,b]V\ &\equiv -\frac{1}{3}\partial_x(h^3\partial_x V)+\frac{1}{2}\big(\partial_x(h^2(\partial_x b) V))-h^2(\partial_x b)(\partial_x V)\big) +h(\partial_x b)^2  V,\\
\N[V_1,V_2] \ &\equiv \ \dfrac{((\partial_x h) V_2-\partial_x(h V_2))^2-\gamma((\partial_x h)V_1-\partial_x(h V_2))^2}{2},\\
\H(V_1,V_2)\ &\equiv\ h_1\Big(\partial_x  (h_1 V_1)+\partial_x  (h_2 V_2)-\frac{1}{2}h_1\partial_x V_1-\partial_x(h_1+h_2)V_1\Big),\\
\R_1(V_1,V_2)\ &\equiv \ \T[h_1,h_2]V_1+\T[h_2,0]V_2-\frac{1}{2}\partial_x(h_1^2\partial_x  (h_2 V_2))-h_1 \partial_x h_2 \partial_x  (h_2 V_2)),\\
\R_2[V_1,V_2]\ &\equiv \ \gamma V_1 \partial_x\big(\H(V_1,V_2)\big)  +\N[V_1,V_2].
\end{align*}
Now, thanks to the fact that $\zeta_1$ is a forced parameter of our problem, the system reduces to two evolution equations for $(\zeta_2,v)$, with $v$ the {\em shear velocity} defined by
\[
v \ \equiv\  \partial_x\left({\phi_2}\id{z=\epsilon_2\zeta_2} - \gamma{\phi_1}\id{z=\epsilon_2\zeta_2} \right)\ =\ \partial_x\psi_2-\gamma H(\psi_1,\psi_2).\]
From the last estimate of Lemma~\ref{Lem:extendVD1}, one has immediately
\begin{equation}\label{eqn:approxv}
 v \ =\ \partial_x\psi_2-\gamma \partial_x\psi_1-\mu\gamma\partial_x \H(\partial_x\psi_1,\partial_x\psi_2)+\O(\mu^2).
\end{equation}
Now, one deduces from the first equation of~\eqref{eqn:S1}, and~\eqref{eqn:approxv}, the following relations
\begin{align}
 \displaystyle  h_1\partial_x\psi_1\ +\ h_2\partial_x\psi_2\ & =\ \alpha \Fr  \zeta_1\ +\ \mu \ \R_1(\partial_x\psi_1,\partial_x\psi_2)\ \ +\O(\mu^2),\nonumber \\
 \displaystyle  v \ & =\ \ \partial_x\psi_2-\gamma \partial_x\psi_1\ -\ \mu\ \gamma\partial_x \H(\partial_x\psi_1,\partial_x\psi_2)\ \ +\O(\mu^2).\label{eqn:vv1v20}
  \end{align}
 It follows that any linear operator defined above can be approximated as in the following example
 \[ \H(\partial_x \psi_1,\partial_x \psi_2 ) \ = \ \H\left(-\dfrac{h_2v-\alpha \Fr  \zeta_1}{h_1+\gamma h_2},\dfrac{h_1v+\gamma\alpha \Fr \zeta_1}{h_1+\gamma h_2}\right) \ + \O(\mu).\]
 For the sake of readability, we do not precise the arguments in the following, and simply write 
 \[\H \ \equiv \  \H\left(-\dfrac{h_2v-\alpha \Fr  \zeta_1}{h_1+\gamma h_2},\dfrac{h_1v+\gamma\alpha \Fr \zeta_1}{h_1+\gamma h_2}\right),\quad \R_1 \ \equiv \  \R_1\left(-\dfrac{h_2v-\alpha \Fr  \zeta_1}{h_1+\gamma h_2},\dfrac{h_1v+\gamma\alpha \Fr \zeta_1}{h_1+\gamma h_2}\right) .\]
 Using~\eqref{eqn:vv1v20}, one can approximate $\partial_x\psi_1$ and $\partial_x\psi_2$ at order $\O(\mu^2)$ with
\begin{align} 
\displaystyle(h_1+\gamma h_2) \partial_x\psi_1 &\ =\ -h_2 v\ +\ \alpha \Fr  \zeta\ +\ \mu \left(\R_1-\gamma h_2\partial_x \H\right) \ + \ \O(\mu^2), \nonumber \\
\displaystyle(h_1+\gamma h_2) \partial_x\psi_2  & \ =\ h_1  v\ +\ \gamma\alpha \Fr  \zeta\ +\ \mu\gamma \left(\R_1+h_1 \partial_x \H\right)+\O(\mu^2).\label{eqn:vv1v2}
 \end{align}

Using these formulae, the system~\eqref{eqn:S1} becomes (dropping $\O(\mu^2)$ terms)
 \begin{equation}\label{eqn:S2} \left\{ \begin{array}{l}
(\partial_{ t}-\Fr\partial_x) \zeta_2 +\partial_x\left(\dfrac{h_2}{h_1+\gamma h_2}(h_1  v+\gamma\alpha \Fr  \zeta)\right)+\mu\partial_x \L  = 0,\\ \\
(\partial_{ t}-\Fr\partial_x) v + (\gamma+\delta)\partial_x \zeta_2 +\dfrac{\epsilon_2}{2}\partial_x \left(\dfrac{|h_1  v+\gamma\alpha \Fr  \zeta|^2 -\gamma |h_2 v -\alpha \Fr  \zeta|^2 }{(h_1+\gamma h_2)^2} \right)+\mu\epsilon_2 \partial_x\Q \\
\hfill = \ \dfrac{1}{\Bo} \partial_x^2\bigg(\dfrac{\partial_x\zeta_2}{\sqrt{1+\mu\epsilon_2^2|\partial_x\zeta_2|^2}}\bigg),
\end{array} \right. \end{equation}
where $\L$ and $\Q$ are defined by
\begin{align*}
\L & = \ \gamma \frac{h_2}{h_1+\gamma h_2} \left(\R_1+h_1 \partial_x \H\right) \ -\ \T[h_2,0]\left(\dfrac{h_1v+\gamma\alpha \Fr \zeta_1}{h_1+\gamma h_2}\right),  \\
\Q & = \ \gamma\dfrac{ (h_1+h_2)  v-(1-\gamma)\alpha \Fr  \zeta_1}{(h_1+\gamma h_2)^2}\R_1 \ -\  \R_2\left[-\dfrac{h_2v-\alpha \Fr  \zeta_1}{h_1+\gamma h_2},\dfrac{h_1v+\gamma\alpha \Fr \zeta_1}{h_1+\gamma h_2}\right]\\
& \qquad  \qquad \qquad  \qquad \qquad  \qquad +\ \gamma\dfrac{((h_1^2-\gamma h_2^2)  v+\gamma(h_1+h_2)\alpha \Fr  \zeta_1}{(h_1+\gamma h_2)^2}\partial_x \H.
\end{align*}
It is now simple to check that the system~\eqref{eqn:S2} is exactly the system~\eqref{eqn:Sbegin}.

\bigskip

This model is justified by the following proposition:
\begin{Proposition}\label{prop:ConsFullyNL}
The full Euler system~\eqref{eqn:AdimEulerComplet} is consistent with~\eqref{eqn:S2}, at precision $\O(\mu^2)$ on $[0,T]$, with $T>0$.
\end{Proposition}
\begin{proof}
 Let $U\equiv(\zeta_1,\zeta_2,\psi_1,\psi_2)$ be a strong solution of~\eqref{eqn:AdimEulerComplet}, bounded in $W^{1,\infty}([0,T];H^{s+t_0})$ with $s>1$ and $t_0\ge 9/2+5$, and such that~\eqref{eqn:h} is satisfied with $\zeta_1(t,x) \ \equiv \ \zeta_1(x)$. Using Lemma~\ref{Lem:extendVD1}, it is clear that $(\zeta_1,\zeta_2,\partial_x\psi_1,\partial_x\psi_2)$ satisfies~\eqref{eqn:S1}, up to $R_1=(r_1,r_2,r_3,r_4)^T\in L^\infty([0,T];H^{s+5})^4$, satisfying (for $i=1\ldots4$)
\[ \big\vert r_i \big\vert_{W^{1,\infty} H^{s+5}} \ \leq \ \mu^2\ C_0\left(\frac{1}{h_{\min}}, \big\vert U\big\vert_{W^{1,\infty}H^{s+t_0}}\right).\]
Now, we set $v \equiv\  \partial_x\psi_2-\gamma H(\psi_1,\psi_2)$, and $v$ satisfies $\eqref{eqn:approxv}$, up to $R_2\in W^{1,\infty}([0,T];H^{s+5})$, with (again, thanks to Lemma~\ref{Lem:extendVD1})
\[ \big\vert R_2\big\vert_{W^{1,\infty} H^{s+5}} \ \leq \ \mu^2\ C_0\left(\frac{1}{h_{\min}}, \big\vert U\big\vert_{W^{1,\infty}H^{s+t_0}}\right).\]
Therefore, since $\R_1$ and $\partial_x \H$ involve two spatial derivatives, one has
\[
\left\{\begin{array}{rl} 
(h_1+\gamma h_2) \partial_x\psi_1 &=\ -h_2 v +\alpha \Fr  \zeta+\mu R_3,\\
(h_1+\gamma h_2) \partial_x\psi_2  &=\ h_1  v+\gamma\alpha \Fr  \zeta+\mu R_4,
  \end{array}\right.
\]
with $R_3, \ R_4\in L^\infty([0,T];H^{s+3})$. Consequently, using the fact that $H^s(\RR)$ is an algebra for $s>1/2$ and that~\eqref{eqn:h} is satisfied, we deduce that~\eqref{eqn:vv1v2} is satisfied, up to $R_5\in L^\infty([0,T];H^{s+1})$, with
\[ \big\vert R_5\big\vert_{L^\infty H^{s+1}} \ \leq \ \mu^2\ \left(\big\vert R_3\big\vert_{L^\infty H^{s+3}}+\big\vert R_4\big\vert_{L^\infty H^{s+3}}\right)\ C_0\left(\frac{1}{h_{\min}}, \big\vert U\big\vert_{W^{1,\infty}H^{s+t_0}}\right).\]
Finally, when plugging $(v,\zeta_2)$ into~\eqref{eqn:S1}, the residuals are clearly bounded by $\mu^2 C_0$, uniformly in $L^\infty([0,T];H^{s})$, and the proposition is proved.
\end{proof}


\section{Proof of Proposition~\ref{PROP:CONVKDV}}\label{sec:proof}
The sketch of the proof is the following. We will first prove the convergence between \[U_{\text{app}}(t,x)=U_0(\mu t,t,x)+\mu U_1(\mu t,t,x),\] defined by~\eqref{eqn:transport}--\eqref{eqn:rest}, and the solutions of the symmetric Boussinesq-type system~\eqref{eqn:SymBouss}. This is obtained thanks to a consistency result, together with energy estimates. The convergence towards solutions of the full Euler system~\eqref{eqn:AdimEulerComplet} follows then immediately from Proposition~\ref{prop:ConvBouss}. Finally, the proposition is completed when remarking that the corrector term $U_1$ obeys to a sublinear growth. This result is a consequence of the following Lemma, that proceeds from Propositions~3.2 and 3.5 of~\cite{Lannes03}: 
\begin{Lemma} \label{lem:Lannes}Let $u$ be the solution of 
\begin{equation}
 \left\{\begin{array}{l}
       (\partial_t+c\partial_x) u=g(v_1,v_2),\\
       u\id{t=0}=0,
        \end{array}\right. \ \text{ with }\ \forall i\in\{1,2\},  \quad
 \left\{\begin{array}{l}
       (\partial_t+c_i\partial_x) v_i=0,\\
       {v_i}\id{t=0}=v^0_i,
        \end{array}\right.
\end{equation}
with $c_1\neq c_2$, $v^0_1$, $v^0_2 \in H^{s}(\RR)$, ${s}>1/2$, and $g$ is a bilinear mapping defined on $\RR^2$ and with values in $\RR$. Then one has the following estimates:
\begin{enumerate}
\item If $c=c_1$ , then $\quad \lim\limits_{t \to \infty }\frac{1}{\sqrt t}\big\vert u(t,\cdot) \big\vert_{H^{s}(\RR)}=0$.
\item If $c\neq c_1\neq c_2$, then $\quad \frac{1}{\sqrt t}\big\vert u(t,\cdot) \big\vert_{H^{s}(\RR)}=\mathcal{O}(1)$.
\end{enumerate}
Moreover, if there exists $a>1/2$ such that $v^0_1 (1+x^2)^a$,  and $v^0_2 (1+x^2)^a\in H^{s}(\RR)$, then one has the better estimate 
\[ \big\vert u \big\vert_{L^\infty H^{s}(\RR)}\ \leq\  C_0 \big\vert v^0_1 (1+x^2)^a \big\vert_{H^{s}(\RR)}\big\vert v^0_2 (1+x^2)^a \big\vert_{H^{s}(\RR)},\]
with $C_0=C(c,c_1,c_2) $.
\end{Lemma}
We can now proceed with the proof.
 
 \medskip

\noindent {\em Step 1: Well-posedness of $U_{\text{app}}(t,x)$.}   
  The global well-posedness of the forced Korteweg-de Vries equation is given by Bona and Zhang in~\cite{BonaZhang96}. 

  The proof relies on an {\em a priori} estimate, that we recall here, as it will be useful for the following arguments. 
  Let $u$ be a solution of 
  \[\partial_t u\ + \ c\partial_x u \ +\ \lambda u\partial_x u\ +\ \nu\partial_x^3 u \ =\  b(x).\]
    As we multiply the equation by $\Lambda^{2{s'}} u$ (with $s'>3/2$), and integrate with respect to the space variable, one obtains 
 \[\frac12\frac{d}{dt}\int_\RR(\Lambda^{s'} u)^2\ \dd x= \left|\lambda \int_\RR \Lambda^{s'} (u \partial_x u)\Lambda^{s'} u\ \dd x\right| + \left|\int_\RR (\Lambda^{s'} b)( \Lambda^{s'} u) \ \dd x \right|. \]
Thanks to the Kato-Ponce Lemma, we estimate the right-hand side as follows: 
 \[\left|\int_\RR \Lambda^{s'} (u\partial_x u)\Lambda^{s'} u\ \dd x \right|\leq \left|\frac12\int_\RR  \partial_x u(\Lambda^{s'} u )^2\right| + \left| \int_\RR  [\Lambda^{s'},u]\partial_x u(\Lambda^{s'} u )\ \dd x\right| \leq C_s \big\vert u\big\vert_{H^{s'}}^2\big\vert \partial_x u\big\vert_{L^\infty},\]
 and the Cauchy-Schwarz inequality leads to
 \[\frac12\frac{d}{dt}\big\vert u\big\vert_{H^{s'}}^2 \ \leq \ C_0\lambda\big\vert u\big\vert_{H^{s'}}^2 \big\vert \partial_x u\big\vert_{L^\infty}\ + \ \big\vert u\big\vert_{H^{s'}}\big\vert b \big\vert_{H^{s'}}.\]
 It follows
 \[\sup_{t\in [0,T]} \big\vert u \big\vert_{H^{s'}} \ \leq \exp\left(\lambda\int_0^T \big\vert \partial_x u(s,\cdot)\big\vert_{L^\infty} \ \dd s\right)\left(\big\vert u\id{t=0}\big\vert_{H^{s'}}+T \big\vert  b\big\vert_{H^{s'}} \right).\]
 
 In the particular case of~\eqref{eqn:KdV1}, since $\alpha=\O(\mu)$, $\epsilon_2=\O(\mu)$ and ${\eta_\pm}\id{t=0}=\frac12 \big(\zeta_2\pm\frac{1}{\gamma+\delta}v\big)\id{t=0}$, it is straightforward that there exists $T=T(\gamma+\delta,\frac1{\gamma+\delta})>0$ such that for all ${s'}>3/2$,
 \begin{equation}\label{eqn:estU0}\sup_{t\in [0,T/\mu]} \big\vert u \big\vert_{H^{s'}} \leq C_0(\big\vert \frac{\dd}{\dd x} \zeta_1 \big\vert_{H^{s'}},\big\vert  V\id{t=0} \big\vert_{H^{s'}}  ,T). \end{equation}

 We can then exhibit $U_1$. Indeed,~\eqref{eqn:rest} can be written in a simplified form as 
 \begin{equation}(\partial_t+c_i\partial_x)\e_i\cdot S_0 U_1=\sum_{(j,k)\neq (i,i)}f_{ijk}(\tau,t,x)+\sum_{j\neq i}\partial_x g_{ij}(\tau,t,x),\end{equation} 
 with 
 \[
    f_{ijk}(\tau,t,x) \ \equiv \ \alpha_{ijk}u_j(\tau,x-c_j t)\partial_x u_k(\tau,x-c_k t) \ \quad \text{ and } \quad  g_{ij}(\tau,t,x) \ \equiv \ \beta_{ij}\partial_x^2 u_j(\tau,x-c_j t).
   \]

  Thanks to estimate~\eqref{eqn:estU0}, one has for $s'>1/2$,
\[\big\vert f_{ijk}\big\vert_{L^\infty([0,T]\times[0,T/\epsilon);H^{s'+1})}+\big\vert g_{ij}\big\vert_{L^\infty([0,T];H^{s'})} \leq C_0(\big\vert \frac{\dd}{\dd x} \zeta_1 \big\vert_{H^{s'+2}},\big\vert  V\id{t=0} \big\vert_{H^{s'+2}}  ,T).\]
with $C_0=C_0(\frac1{\gamma+\delta},\gamma+\delta)$. Hence, for $s'>1/2$, one can set
 \begin{align}
 \e_\pm \cdot S_0 U_1(\tau,t,x)&=\sum_{(j,k)\neq (i,i)}\int_0^t f_{ijk}(\tau,s,x+c_i(s-t))\ \dd s\nonumber \\
&\qquad +\sum_{j\neq i} \frac1{c_i-c_j}\big(g_{ij}(\tau,x-c_j t)-g_{ij}(\tau,x-c_i t)\big)\nonumber \\ 
 &\equiv \sum_{(j,k)\neq (i,i)} U^{ijk}+\sum_{j\neq i} V^{ij}. \label{eqn:simpleformU1}  \end{align}
One checks immediately that $U_1 \in L^\infty([0,T]\times[0,T/\mu);H^{s'})$ and ${{U_1}\id{\tau=t=0}=0}$.
  
   \medskip

\noindent  {\em Step 2: Estimate on $U_1$.} Let us estimate each term of the decomposition~\eqref{eqn:simpleformU1}.
Thanks to the uniform estimates of Step {\em 1.} above, one has $g_{ij}\in L^\infty([0,T];H^{s'})$, so that it follows immediately
\[\forall j\neq i,\qquad  \big\vert  V^{ij}\big\vert_{L^\infty(\RR^+\times[0,T];H^{s'})}\ \leq \ C_0(\big\vert \frac{\dd}{\dd x} \zeta_1 \big\vert_{H^{s'+2}},\big\vert  V\id{t=0} \big\vert_{H^{s'+2}}  ,T).\] Moreover, for $j\neq i$, we remark that $f_{ijj}$ can be written as 
\[f_{ijj}(\tau,t,x) \ \equiv \ \alpha_{ijj}u_j(\tau,x-c_j t)\partial_x u_j(\tau,x-c_j t) \ \equiv \ \partial_x h_{ij}(\tau,x-c_j t), \]
 so that $U^{ijj}$ has actually the same form as $V^{ij}$, and can be treated in the same way. Since $f_{ijj}\in L^\infty([0,T]\times[0,T/\epsilon);H^{s'+1})$, $U^{ijj}$ is bounded in $H^{s'}$ by $C_0(\big\vert \frac{\dd}{\dd x} \zeta_1 \big\vert_{H^{s'}},\big\vert  V\id{t=0} \big\vert_{H^{s'}}  ,T)$, for all $j\neq i$.
 
 Finally, for all $(j,k)\neq (i,i)$ with $j\neq k$, $U^{ijk}$ satisfies the hypothesis of Lemma~\ref{lem:Lannes}, with $f_{ijk}=g(u_j,\partial_x u_k)$. 
Therefore, we deduce
\begin{equation}\label{eqn:estU1}
\big\vert U_1\big\vert_{L^\infty(\RR^+\times[0,T];H^{s'})} \leq\sqrt{t} C_0(\big\vert \frac{\dd}{\dd x} \zeta_1 \big\vert_{H^{s'+2}},\big\vert  V\id{t=0} \big\vert_{H^{s'+2}}  ,T).
 \end{equation}
 
 \medskip
 
 As for the second estimate of the proposition, let us first remark that the estimates of $V^{ij}$ and $U^{ijj}$ are time-independent, and in agreement with the improved estimate. Therefore, the only remaining terms we have to control are $U^{ijk}$ with $j\neq k$. Of course, we will use the second case of Lemma~\ref{lem:Lannes}, but we have to check first that for every $\tau\in\RR^+$, the initial data $u_j(\tau,0,x)$ and $\partial_x u_k(\tau,0,x)$ are localized in space, that is   
 \[\forall \tau\in\RR^+,\qquad \big\vert  (1+x^2) {u_\pm}(\tau,0,x) \big\vert_{H^{s'}} \ + \ \big\vert  (1+x^2) \partial_x{u_\pm}(\tau,0,x)\big\vert_{H^{s'}} \ < \ \infty.\]
 This property is true at $\tau=0$ (by hypothesis of the proposition), and is propagated to $\tau>0$, using the fact that $u_\pm(\tau,x_\pm)$ satisfies the KdV equation~\eqref{eqn:KdV0}. This propagation of the localization in space has been proved for by Schneider and Wayne in~\cite[Lemma 6.4]{SchneiderWayne00}. We do not recall the proof here, and make direct use of the statement: 
 \begin{Lemma}
  If $(1+x^2) V\id{t=0} \in H^{s'+2}$, then there exists $C_1,\t C_1>0$ such that
 \[ \big\vert  (1+x^2) {u_\pm}(\tau,0,x) \big\vert_{H^{s'+1}} \ \leq \ C_1\big\vert  (1+x^2) {u_\pm}\id{\tau=t=0} \big\vert_{H^{s'+2}} \ \leq \  \t C_1\big\vert  (1+x^2) {V}\id{t=0} \big\vert_{H^{s'+2}} .\]
 \end{Lemma}
This Lemma, together with the second estimate of Lemma~\ref{lem:Lannes}, allows to control uniformly $U^{ijk}\in L^\infty(\RR^+\times[0,T];H^{s'})$.
Every term of the decomposition~\eqref{eqn:simpleformU1} has been controlled, and one has the following estimate:
\begin{equation}\label{eqn:estU12}
\big\vert U_1\big\vert_{L^\infty(\RR^+\times[0,T];H^{s'})} \leq\ C_0(\big\vert \frac{\dd}{\dd x} \zeta_1 \big\vert_{H^{s'+2}},\big\vert  V\id{t=0} \big\vert_{H^{s'+2}} ,\big\vert  (1+x^2) V\id{t=0} \big\vert_{H^{s'+2}}  ,T).
 \end{equation}
  
   \medskip

\noindent  {\em Step 3: Consistency results.} We first prove the the consistency of the Boussinesq-type system~\eqref{eqn:SymBouss} with our approximation. The precision is $\O(\mu^{3/2})$ in the general case, and $\O(\mu^2)$ in the second case. Here and in the following, we set $\alpha=\epsilon_2=\mu$ in order to simplify the notations. The general case $\alpha=\O(\mu)$, $\epsilon_2=\O(\mu)$ is obtained with slight obvious modifications.
  
  Plugging $U_{\text{app}}(t,x)$ into~\eqref{eqn:AnsatzInBouss}, we see from~\eqref{eqn:TRANSPORT}-\eqref{eqn:rest} that the only remaining term we have to control is $\mu^2 R(\mu t,t,x)$, with 
\begin{align*}R\ \equiv\ \partial_\tau U_1+\Sigma_1(U_0)\partial_x U_1+\Sigma_1(U_1)\partial_x U_0+S_1(U_0)\partial_t U_1+S_1(U_1)\partial_t U_0\\ -\Sigma_2\partial_x^3 U_1-S_2\partial_x^2\partial_t U_1 +\mu \Sigma_1(U_1)\partial_x U_1+\mu S_1(U_1)\partial_x U_1,\end{align*}
where $U_0(\mu t,t,x)\ =\ u_+(\mu t,x-c_+ t)\e_i\ +\ u_-(\mu t,x-c_- t)\e_-$. 
We bound each term of the right-hand side in the Sobolev $H^{s}$-norm, with $s>1/2$ as in the proposition. 

The estimate~\eqref{eqn:estU1}, with $s'=s+3$, leads to
\[\big\vert \Sigma_2 \partial_x^3 U_1 \big\vert_{H^{s}} \leq \sqrt{t}\ C_0(\big\vert \frac{\dd}{\dd x} \zeta_1 \big\vert_{H^{s+5}},\big\vert  V\id{t=0} \big\vert_{H^{s+5}}  ,T).\]
Then, from~\eqref{eqn:rest}, one has \[\e_i\cdot\partial_t S_0 U_1=-c_i\e_i\cdot\partial_x S_0 U_1+f_i\] with ${f_i\in L^\infty([0,T]\times[0,T/\mu);H^{s+2})}$, and $\big\vert f_i\big\vert_{H^{s+2}}\leq  C_0(\big\vert \frac{\dd}{\dd x} \zeta_1 \big\vert_{H^{s+5}},\big\vert  V\id{t=0} \big\vert_{H^{s+5}}  ,T)$, so that 
\[\big\vert S_2 \partial_x^2\partial_t U_1 \big\vert_{H^{s}} \leq \sqrt{t}\ C_0 \big\vert \partial_t U_1 \big\vert_{H^{s+2}} \leq \sqrt{T/\mu} C_0(\big\vert \frac{\dd}{\dd x} \zeta_1 \big\vert_{H^{s+5}},\big\vert  V\id{t=0} \big\vert_{H^{s+5}}  ,T).\]
One obtains in the same way the desired estimates for $\Sigma_1(U_0)\partial_x U_1$, $\Sigma_1(U_1)\partial_x U_0$, $S_1(U_0)\partial_t U_1$, $S_1(U_1)\partial_t U_0$, $\Sigma_1(U_1)\partial_x U_1$ and $S_1(U_1)\partial_x U_1$. 

Finally, in order to estimate $\partial_\tau U_1$, we differentiate~\eqref{eqn:rest} with respect to $\tau$. Since $u_i$ satisfies~\eqref{eqn:KdV0}, one has $\partial_\tau u_i \in L^\infty([0,T];H^{s+2})$. We are on the framework of Lemma~\ref{lem:Lannes}, so that we can obtain, as for~\eqref{eqn:estU1}, that $\partial_\tau U_1 \in L^\infty([0,T];H^{s})$, with 
\[\big\vert \partial_\tau U_1\big\vert_{H^{s}} \leq \sqrt{t}\ C_0(\big\vert \frac{\dd}{\dd x} \zeta_1 \big\vert_{H^{s+5}},\big\vert  V\id{t=0} \big\vert_{H^{s+5}}  ,T).\]
Hence, the residual $\mu^2 R$ is uniformly bounded in $L^\infty ([0,T/\mu) ;H^{s})$, and
\[\big\vert R \big\vert_{H^{s}}\leq C_0  \mu^{-1/2}C_0(\big\vert \frac{\dd}{\dd x} \zeta_1 \big\vert_{H^{s+5}},\big\vert  V\id{t=0} \big\vert_{H^{s+5}},T) ,\] 
which gives the consistency at order $\O(\mu^{3/2})$.

\medskip

The consistency at order $\O(\mu^2)$, for initial data sufficiently decreasing in space, is obtained in the same way, using the second estimate~\eqref{eqn:estU12}.

  \medskip

\noindent {\em Step 4: Convergence results.} The convergence is deduced from the consistency, as in the proof of Proposition~\ref{prop:ConvBouss}. Indeed, 
 when setting $R^\mu\equiv  U_{\text{app}}-V \equiv U_0(\mu t,t,x)+\mu U_1(\mu t,t,x) - V(t,x)$, one can check that $R^\mu$ satisfies the following equation:
\begin{equation}\label{eqn:RBOUSS}\big(S_0-\mu S_2\partial_x^2+\mu S_1(U_{\text{app}} \big) \partial_t R^\mu +\big(\Sigma_0-\mu \Sigma_2\partial_x^2+\mu \Sigma_1(U_{\text{app}})\big)\partial_x R^\mu=\mu^{3/2} f+\mu\mathcal A+\mu\mathcal B,\end{equation}
with $\mathcal A=\partial_t S_1(U_{\text{app}})R^\mu-S_1(R^\mu)\partial_t V$, $\mathcal B=\partial_x \Sigma_1(U_{\text{app}}) R^\mu-\Sigma_1(R^\mu)\partial_x V$ and the function $f$ proved to be uniformly bounded in $H^s$, by the consistency result.
We define the energy as
\[E_s(R^\mu)\equiv\frac12(S_0 \Lambda^s R^\mu, \Lambda^s R^\mu)+\frac{\mu}{2}(S_2 \partial_x \Lambda^s R^\mu,\partial_x \Lambda^s R^\mu)+\frac{\mu}{2}(S_1(U_{\text{app}}) \Lambda^s R^\mu, \Lambda^s R^\mu),\]
as in the proof of Proposition~\ref{prop:ConvBouss}. The same calculations lead to the following differential inequality: 
\[\frac{d}{dt}E_s(R^\mu)\leq C_0  \mu E_s(R^\mu)+ C_0 \mu^{3/2}(E_s(R^\mu))^2,\]
with $C_0=C_0(\big\vert \frac{\dd}{\dd x} \zeta_1 \big\vert_{H^{s+5}},\big\vert  V\id{t=0} \big\vert_{H^{s+5}},T)$. Then, Gronwall-Bihari's theorem allows to obtain ${(E_s(R^\mu))^{1/2} \leq C_0  \mu^{1/2} (e^{C_0 \mu t}-1)}$, and finally for $\mu t\leq T(\frac1{\gamma+\delta},\gamma+\delta)$,
\[\big\vert U_{\text{app}}-V\big\vert_{H^{s}}\leq C_0 (E_s(R^\mu))^{1/2} \leq  \mu^{3/2}t\ C_0(\big\vert \frac{\dd}{\dd x} \zeta_1 \big\vert_{H^{s+5}},\big\vert  V\id{t=0} \big\vert_{H^{s+5}}  ,T).\]
The first estimate of the proposition is now a direct consequence of~\eqref{eqn:estU1}, and 
\[\big\vert \frac{\dd}{\dd x} \zeta_1(x) \big\vert_{H^{s+5}}+\big\vert  V(t,x) \big\vert_{H^{s+5}} \leq \big\vert U(t,x) \big\vert_{H^{s+t_0}},\]
thanks to an appropriate estimate of the operator $H(\psi_1,\psi_2)$; see~\cite[Proposition 2.7]{Duchene10}.

\medskip

The second estimate of the proposition follows in the same way, using~\eqref{eqn:estU12} and the consistency at order $\O(\mu^2)$.

\section{Wave resistance and the dead-water phenomenon}\label{sec:deadwater}
Following Lamb~\cite{Lamb16} and Kostyukov~\cite{Kostyukov59}, we assume that the drag experienced by ships is mostly due to the wave (making) resistance. This section is devoted to the analysis of this wave resistance. We first work with the variables {\em with dimension}, and deduce an explicit formula for the wave resistance $R_W$, depending on the flow, given as a solution of the full Euler equation~\eqref{eqn:EulerComplet}. Accordingly with the nondimensionalization performed in Section~\ref{sec:nondimensionalization}, we introduce the dimensionless version of the wave resistance, that we call {\em wave resistance coefficient} and denote $C_W$. Finally, we derive simple approximations, for the two regimes considered throughout the paper.

\medskip

The wave resistance acknowledges the energy required from the body to push the water away, and is defined by (see~\cite{MotyginKuznetsov97,MilohTulinZilman93,WuChen03} and references therein)
\[R_W\ \equiv \ \int_{\Gamma_{\text{ship}}} P\ (-\e_x\cdot \n) \ \dd S,\]
where $\Gamma_{\text{ship}}$ is the exterior domain of the ship, $P$ is the pressure, $\e_x$ is the horizontal unit vector and $\n$ the normal unit vector exterior to the ship. The pressure being constant ($\equiv P_\infty$) on the non-submerged part of the ship, one has
\begin{align*}R_W & = \ \int_{\Gamma_{\text{ship}}}(P-P_\infty)(-\e_x\cdot \n) \ \dd S \ = \ \int_\RR \big(P\id{d_1+\zeta_1}-P_\infty\big)(-\partial_x\zeta_1)\ \dd x \\
 &=\ -\int_\RR P\id{d_1+\zeta_1}\partial_x\zeta_1\ \dd x.
\end{align*}
As a solution of the Bernoulli equation in~\eqref{eqn:EulerComplet}, the pressure $P$ is given in the upper domain by
\[\frac{P(x,z)}{\rho_1} \ =\ -\partial_t \phi_1(x,z)\ -\ \frac{1}{2} |\nabla_{x,z} \phi_1(x,z)|^2\ -\ gz,\]
so that the wave resistance satisfies
\begin{align*}R_W& = \ \rho_1 \int_\RR g(d_1+\zeta_1)\partial_x\zeta_1 - \left( \partial_x\partial_t \phi_1 +\frac12\partial_x\left(\partial_x\phi_1\right)^2+\frac12\partial_x\left(\partial_z\phi_1\right)^2\right)\id{z=d_1+\zeta_1} \zeta_1 \ \dd x.\\
 &\equiv \ \int_\RR F_W[\zeta_2,\phi_1] \left(d_1+\zeta_1\right)  \zeta_1 \ \dd x.
\end{align*}

\bigskip

Let us now construct the dimensionless version of this formula. Using the same change of variables as in Section~\ref{sec:nondimensionalization}, it is straightforward to obtain the dimensionless wave resistance (that we call {\em wave resistance coefficient}, and denote $C_W$):
\begin{align}\label{eqn:CW}
&C_W \ \equiv \ \frac{\rho_1\lambda}{ a_2 c_0^2}\ R_W \ = \ \int_\RR \t F_W[\t\phi_1] (\t x) \t \zeta_1(\t x) \ \dd\t x, \quad \text{with}\\
 &-\t F_W[\t\phi_1] \ \equiv \ (\partial_{\t t} -\Fr\partial_{\t x}) \partial_{\t x}\t \phi_1 +\frac{\epsilon_2}2 \partial_{\t x}\left(\left(\partial_{\t x} \t \phi_1\right)^2 + \frac{\epsilon_2}{2\mu}\left(\partial_{\t z} \t \phi_1 \right)^2\right).\nonumber
\end{align}
Again, we omit the tildes in the following, for readability. Let us remark that by definition of the Dirichlet-Neumann operator $G_1$, and using~\eqref{eqn:AdimEulerComplet}, one has 
\begin{equation} \label{eqn:FW}
-F_W[\phi_1] \ = \ (\partial_{ t} -\Fr\partial_{ x}) \partial_{ x} \phi_1 + \frac{\epsilon_2}2 \partial_{ x} \left( \left( \partial_{ x}  \phi_1\right)^2 + \left(-\alpha \Fr \partial_x\zeta_1 + \epsilon_1\frac{\dd}{\dd x} \zeta_1 \partial_x\phi_1\right)^2\right).
 \end{equation}

\bigskip

For practical purposes, we use approximations to compute the wave resistance coefficient, corresponding to the leading order term in the asymptotic expansion of $C_W$, for the different regimes at stake. These formulae use the variables of Sections \ref{sec:stronglynonlinear} and \ref{sec:weaklynonlinear}, {\em i.e.} the surface and interface deviations $\zeta_1$ and $\zeta_2$ and the {\em shear velocity}, defined by
\[ v \ \equiv\  \partial_x\left({\phi_2}\id{z=\epsilon_2\zeta_2} - \gamma{\phi_1}\id{z=\epsilon_2\zeta_2} \right).\]
First of all, we use the following estimate, justified in~\cite{Duchene10}:
\[ \phi_1(x,z) \  = \ \psi_1(x) \ +\ \O(\mu).\]
When combined with the first equation of~\eqref{eqn:vv1v2}:
\[ \partial_x\psi_1 \  =\  \frac{-h_2 v + \alpha \Fr \zeta_1}{h_1+\gamma h_2}\  + \O(\mu), \quad \text{ with } \quad h_1\equiv 1+\epsilon_1\zeta_1, \ \ \text{ and } \ \ h_2\equiv\frac1\delta+\epsilon_2\zeta_2,\]
one has immediately that~\eqref{eqn:FW} simplifies into
\begin{align}\label{eqn:FWSW}
-F_W[\phi_1] & = \  (\partial_{ t} -\Fr\partial_{ x})\left(\frac{-h_2 v + \alpha \Fr \zeta_1}{h_1+\gamma h_2}\right)\hfill \nonumber \\
&+\frac{\epsilon_2}2 \partial_{ x} \left(\frac{\left(-h_2 v + \alpha \Fr \zeta_1 \right)^2+(\alpha\frac{\dd}{\dd x} \zeta_1)^2\left((h_1+\gamma h_2)\Fr +\epsilon_2( h_2 v - \alpha \Fr \zeta_1)\right)^2}{(h_1+\gamma h_2)^2}\right)+ \O(\mu).\end{align}
Now, $\partial_t h_2$, and $\partial_t v$ can be written using only $\zeta_1$, $\zeta_2$, $v$, and their spatial derivatives, since they satisfy system~\eqref{eqn:S2} up to $\O(\mu^2)$. Therefore, the leading order term of the wave resistance coefficient $C_W$ can be deduced from the knowledge of the solution $(\zeta_2,v)$. We do not write explicitly this expression, as the models used in our simulations benefit from extra smallness assumptions, and simpler formulae are deduced in these cases.

\medskip

\para{Case of Regime~\ref{regimeRL}}
\[\mu \ \ll \ 1 \ ; \qquad \alpha\equiv\frac{\epsilon_1}{\epsilon_2} \ = \ \O(\mu)\ , \ 1-\gamma \ = \ \O(\mu).\]
The first immediate simplification in this regime is
\[ h_1+\gamma h_2 \ = \ h_1 \ + \ h_2 \ + \ \O(\mu) \ = \ 1+\frac1\delta \ + \ \O(\mu).\]
Then, we use that $(\zeta_2,v)$ satisfies system~\eqref{eqn:D1} up to $\O(\mu^2)$ to deduce\footnote{as we use $\Bo^{-1}=\mu^2$ in our simulations, we do not take here into account the surface tension term.} (with $\b h_1\equiv 1+\frac1\delta - h_2$)
\begin{align*}
 \partial_t(h_2v) &= \ \Fr \partial_x(h_2 v) + (1+\delta)h_2\partial_x \zeta_2 + \epsilon_2\frac{\delta}{1+\delta}\left(v\partial_x\left(\b h_1h_2v\right)+\frac{h_2}2\partial_x\left(\b h_1-h_2)v^2\right)\right)+\O(\mu)\\
 &= \ \Fr \partial_x(h_2 v)+ (1+\delta)h_2\partial_x h_2  + \epsilon_2\partial_x(h_2v)+\frac{3\epsilon_2}2\frac{\delta}{1+\delta}\partial_x(h_2^2v^2).
\end{align*}
Finally, using these approximations into~\eqref{eqn:CW} and~\eqref{eqn:FWSW}, one has
\begin{align}C_W &=\ \frac{\delta}{1+\delta}\int_\RR  \left( (1+\delta)h_2\partial_x \zeta_2 +\epsilon_2\partial_x(h_2v)+\epsilon_2\frac{\delta}{1+\delta}\partial_x(h_2^2v^2)\right) \zeta_1(x) \ \dd x + \O(\mu^2) \nonumber \\
&= \  -\int_\RR  \left( \left(\zeta_2+\frac{\epsilon_2}2\zeta_2^2\right) +\epsilon_2\frac{\delta\ h_2v}{1+\delta}+\epsilon_2\left(\frac{\delta\ h_2v}{1+\delta}\right)^2\right) \frac{\dd}{\dd x} \zeta_1(x) \ \dd x + \O(\mu). \label{eqn:WaveResistanceRL}
\end{align}
This is the formula used in Figures~\ref{fig:FNLd1}--\ref{fig:hysteresis}, in Section~\ref{sec:numericsFNL}.

\medskip

\para{Case of Regime~\ref{regimeSA}}
\[\mu \ll 1\ ; \qquad \epsilon_2\ = \ \O(\mu)\ , \ \alpha\ = \ \O(\mu) .\]
Most of the terms of~\eqref{eqn:FWSW} are now of size $\O(\mu)$. The first order of system~\eqref{eqn:Dfinal} leads immediately to  
\[
 \partial_t(h_2v) - \Fr \partial_x(h_2v) \ = \  -\frac{\gamma+\delta}{\delta}\partial_x \zeta_2 +\O(\mu),
\]
so that we obtain simply 
\begin{equation}\label{eqn:WaveResistanceSA}
C_W = - \int_\RR \zeta_2  \frac{\dd}{\dd x} \zeta_1 \ \dd x + \O(\mu).
\end{equation}
This formula is used in Figures~\ref{fig:subcritical}--\ref{fig:critical}, in Section~\ref{sec:KdVana}.

\bigskip

\begin{Remark}\label{rem:waveresistance}
 As we can see in~\eqref{eqn:WaveResistanceRL} and~\eqref{eqn:WaveResistanceSA}, the wave resistance coefficient will be small in the two following cases
 \begin{enumerate}
 \item The internal wave is small, or is not located below the ship,
 \item The internal wave is symmetric and centered at the location of the ship.
 \end{enumerate}
 The first case is obvious, and the last case reflects the fact that the integrands of~\eqref{eqn:WaveResistanceRL} and~\eqref{eqn:WaveResistanceSA} are odd if $\zeta_2$, $v$ and $\zeta_1$ are even.
 
 On the contrary, the ship will encounter a strong positive wave resistance if the functions $\zeta_2$ and $h_2 v$ are {\em decreasing at the location of the body}. This is the case when an internal wave of elevation is located just behind the ship. The dead-water effect is then explained in that way: the ship, traveling in a density-stratified water, generates internal waves of elevation in its trail, and therefore suffers from a severe wave resistance.
\end{Remark}

\section{The Numerical schemes}\label{sec:numericalschemes}
We present in this section the numerical method used to obtains the figures of this study. For all of the simulations, we use a scheme based one the Crank-Nicholson method, and take care of the nonlinearities using a predictive step. This method has been introduced in~\cite{BesseBruneau98,Besse98}, and used in the water wave framework in~\cite{Chazel09,DurufleIsrawi,Duchene}. We give the general directions of such a method in the following, and then present the exact schemes we used for each of the models. 

\medskip

\noindent {\em Time discretization.}
Denoting $\Delta t$ the time step of the scheme, we approximate a function $u(x,t)$ at time $t=n\Delta t$ by $ u(x,n\Delta t)\equiv u^n(x)$. Then, we approximate the time derivative by 
\[\partial_t u \approx \frac{u^{n+1}-u^n}{\Delta t}\]
and any linear function of $u$ by
\[ F(u)\ \approx\ \frac{F(u^{n+1})+F(u^n)}{2} \ = \ F\Big(\frac{u^{n+1}+u^n}2\Big).\]
Now we deal with the nonlinearities using a predictive step, $u^{n+\frac12}$, defined by
\[\frac{u^{n+\frac12}+u^{n-\frac12}}{2} = u^n.\]
The first half-step, $u^{\frac12}$, is computed through a simple explicit scheme. 

With the help of the predictive step, the discretization of a quadratic nonlinearity $F(u,v)$ ($F$ being linear with respect to both variables)
is then a linear combination of the two possible discretizations, namely
\[F(u,v) \approx \alpha F\left(u^{n+\frac12},\frac{v^{n+1}+v^n}2 \right)+ (1-\alpha) F\left(\frac{u^{n+1}+u^n}2,v^{n+\frac12}\right).\]
The parameter $\alpha$ can be chosen so that natural quantities are conserved.

\medskip

\noindent {\em Space discretization.}
With $\Delta x$ the space step of the scheme, the functions are discretized spatially with a central difference. It is also useful to introduce a Lax-type averaging, by
\[ u(x,t)\approx M(\beta)u \mbox{ with } (M(\beta)u^n)_i=(1-\beta) u_i^n+ \frac{\beta}2(u_{i+1}^n+u_{i-1}^n). \]
The spatial derivatives are given by
\[
\partial_x u \approx \frac{u_{i+1}^n-u_{i-1}^n}{2\Delta x} \equiv (D_1 u^n)_i, \ \
 \partial_x^2 u \approx \frac{u_{i+1}^n-2u_i^2+u_{i-1}^n}{\Delta x^2} \equiv (D_2 u^n)_i, \ \
 \dots
\]
We use periodic boundary conditions.

\medskip

As we see in the following section, the parameters $\alpha$ and $\beta$ can be set so that natural quantities are conserved. In the case of a Korteweg-de Vries equation, when choosing carefully the parameters, one can obtain a numerical scheme that preserves the discrete $l^2$-norm, in agreement with the solutions of the KdV equation having constant $L^2$-norms. As for the more complicated case of system~\eqref{eqn:Dfinal}, we set the parameters by analogy with the simple Korteweg-de-Vries equation.

\subsection{The forced Korteweg-de-Vries equation}
We want to solve numerically the following generic forced KdV equation
 \begin{equation}\label{fKdVForScheme}
 \partial_t u + c \partial_x u + \lambda u \partial_x u +\nu \partial_x^3 u = f(x).
 \end{equation}
When there is no forcing term ($f\equiv 0$), it is known that the KdV equation preserves the energy $E\equiv \big\vert u \big\vert_{L^2}^2$. The following scheme has been presented and studied (without the forcing term) in~\cite{DurufleIsrawi}, and has the property to preserve the discrete energy when $f\equiv 0$.
\begin{align}\label{fKdVscheme}
\frac{u^{n+1}_i-u^n_i}{\Delta t}+ c\left(D_1\frac{u^{n+1}+u^n}{2}\right)_i +\nu \left(D_3 \frac{u^{n+1}+u^n}{2}\right)_i + \frac{\lambda}{3} \left( \left(M(1)\frac{u^{n+1}_i+u^n_i}{2}\right)_i \left(D_1 u^{n+\frac12}\right)_i \right. & \nonumber \\ \left. + 2 \left(M(1/2)u^{n+\frac12}\right)_i \left(D_1 \frac{u^{n+1}+u^n}{2}\right)_i \right)=f(i\Delta x).& 
 \end{align}
\begin{Proposition}
 If $f \equiv 0$, then the scheme~\eqref{fKdVscheme} preserves the discrete $l^2$-norm:
 \[\forall\ n\in\NN,\quad \big\vert u^n \big\vert_{l^2}^2 \ \equiv\ \sum_i |u_i^n|^2 \ = \  \sum_i |u_i^0|^2. \]
 \end{Proposition}
The proof is given in~\cite[Theorem 2]{DurufleIsrawi}.

\subsection{The fully nonlinear model of Regime~\ref{regimeRL}}
The system we deal with now is~\eqref{eqn:Dfinal}, that we can write under the compact form 
 \begin{equation}\label{FNLforScheme}
 (\partial_{ t}-\Fr\partial_x) \left(U  -  \mu \R_1[U]\right)\ +\ \partial_x \left((\A[U] + \epsilon_1 \B(x))U\right) \ +\ \mu\epsilon_2 \partial_x (\R_2[U]\cdot U) \ = \ b(x)  +  \frac{1}{\Bo} \partial_x^2\left(\T[U]\right),
 \end{equation}
with $U\ \equiv \ (\zeta_2,w)^T$, $b(x) \ \equiv \ -\alpha\Fr/(1+\delta)\ ( \zeta_1(x),0)^T$, and 
\begin{align*}
  \A[U]\ &\equiv\ \begin{pmatrix}
             0 & \frac{h_1h_2}{h_1+\gamma h_2} \\
             \gamma+\delta & \epsilon_2\frac{h_1^2-\gamma h_2^2}{(h_1+\gamma h_2)^2}w
              \end{pmatrix},\qquad 
   & \B(x)\ &\equiv\ \frac{\Fr}h \begin{pmatrix}
                \zeta_1(x) & 0 \\
               0 & \zeta_1(x)
              \end{pmatrix}, \\
    \R_i[U]\ &\equiv\ \begin{pmatrix}
               0 \\
               \S_i[h_2]w
              \end{pmatrix}, \qquad 
    &\T[U]\ &\equiv\ \begin{pmatrix}
               0 \\
           \frac{\partial_x\zeta_2}{\sqrt{1+\mu\epsilon_2^2|\partial_x\zeta_2|^2}}
              \end{pmatrix}.
\end{align*}
It is convenient to denote, with $h_1\equiv 1+\epsilon_1\zeta_1-\epsilon_2\zeta_2$ and $h_2\equiv \frac1\delta+\epsilon_2\zeta_2$,
\[f[\zeta_2] \ \equiv \ \frac{h_1h_2}{h_1+\gamma h_2}, \quad  \quad g[\zeta_2] \ \equiv \ \epsilon_2\frac{h_1^2-\gamma h_2^2}{(h_1+\gamma h_2)^2}, \quad \text{and} \quad \ t[\zeta_2]  \ \equiv \ \frac{\partial_x\zeta_2}{\sqrt{1+\mu\epsilon_2^2|\partial_x\zeta_2|^2}}.\]

Advised by the above work on the KdV equation, we use the following time discretization:
\begin{align*}
 f[\zeta]w & \approx  f[\zeta^{n+\frac12}]\frac{w^n+w^{n+1}}{2}, \qquad \qquad 
 g[\zeta]w^2 \approx   g[\zeta^{n+\frac12}]w^{n+\frac12}\frac{w^n+w^{n+1}}{2},\\ \\
 t[\zeta] &\approx  \dfrac{1}{\sqrt{1+\mu\epsilon_2^2|\partial_x\zeta^{n+\frac12}|^2}} \partial_x\left(\frac{\zeta^{n+1}+\zeta^n}2\right),\\ \\
 \partial_t (\S_1[h]w)  &\approx \S_1[h^{n+\frac12}]\frac{w^{n+1}-w^n}{\Delta t} +\t\S_1[h^{n+\frac12},w^{n+\frac12}]\frac{\zeta^{n+1}-\zeta^n}{\Delta t}, \\ 
  &\qquad \qquad \text{ with } \quad \t\S_1[h,w]\zeta \ \equiv \ \frac{\epsilon_2}3\partial_x^2(\zeta(1+1/\delta-2h)w)-2\epsilon_2(\partial_x\zeta)(\partial_x h)w,\\ \\
w\ \S_2[H] w &\approx  \frac13 w^{n+\frac12}\S_2[H^{n+\frac12}]\left(\frac{w^n+w^{n+1}}2\right) +\frac23\frac{w^n+w^{n+1}}2\S_2[H^{n+\frac12}]w^{n+\frac12}.
\end{align*}
%
Finally, after the space discretization, this leads to the following scheme (with $h^n\equiv\frac1\delta + \epsilon_2\zeta^n$)
\begin{align}\label{SchemeFNL}
 &\frac{\zeta_i^{n+1}-\zeta_i^n}{\Delta t}-\Fr \left(D_1\frac{\zeta^{n+1}+\zeta^n}{2}\right)_i +\alpha \frac{\Fr\delta}{1+\delta}\left(D_1\left(\left(\frac1\delta+\epsilon_2\frac{\zeta^{n+1}+\zeta^n}{2}\right)\zeta_1(i\Delta x)\right)\right)_i\\
\nonumber &\qquad\qquad +\ \frac13D_1\left(M(1)f[h^{n+\frac12}] \frac{w^n+w^{n+1}}2\ + \ f[h^{n+\frac12}]M(1/2)(w^n+w^{n+1})\right)_i=0,  \\ 
 &\frac{(I-\mu \S_1[h^{n+\frac12}])(w^{n+1}-w^n)-\mu\t\S_1[h^{n+\frac12},w^{n+\frac12}](\zeta^{n+1}-\zeta^n)}{\Delta t}\ + \ (\gamma+\delta) \left( D_1\frac{\zeta^n+\zeta^{n+1}}2 \right)_i\\
\nonumber &\qquad\qquad-\ \Fr \left(D_1(I-\mu \S_1[h^{n+\frac12}])\frac{w^{n+1}+w^n}{2}\right)_i +\ \alpha \frac{\Fr\delta}{1+\delta}\left(D_1\left(\frac{w^{n+1}+w^n}{2}\zeta_1(i\Delta x)\right)\right)_i\\
\nonumber&\qquad\qquad + \frac{\epsilon_2}3D_1 \left(M(1)g[h^{n+\frac12}] w^{n+\frac12} \frac{w^n+w^{n+1}}2 +g[h^{n+\frac12}]w^{n+\frac12}M(1/2)(w^n+w^{n+1})\right)_i\\
\nonumber& \qquad\qquad +\ \frac{\mu\epsilon_2}3 D_1\left( w^{n+\frac12}\S_2[H^{n+\frac12}]\left(\frac{w^n+w^{n+1}}2\right) +(w^n+w^{n+1})\S_2[H^{n+\frac12}]w^{n+\frac12}\right)_i \\
\nonumber& \qquad\qquad = \  \frac1{\Bo}\left(D_2\dfrac{1}{\sqrt{1+\mu\epsilon_2^2|\partial_x\zeta^{n+\frac12}|^2}} D_1\left(\frac{\zeta^{n+1}+\zeta^n}2\right)\right)_i.
 \end{align}

\subsection{Validation of the method}
In order to validate the proposed schemes, we use known explicit solutions of the forced Korteweg-de Vries equation~\eqref{eqn:fKdV}. The first test function we use is the travelling wave
\begin{equation}\label{init1}
 u_0 \equiv -\sech^2(k(x-c_{sw} t)),
\end{equation}
with $c_{sw} = c+4\nu k^2,  \ k = \sqrt{-\lambda/(12 \nu)}$, corresponding to the classical case where there is no forcing.

The second test function is the steady solution
\begin{equation}  \label{init2} 
u_1 \ \equiv \ \pm \sech^2(kx) ,
\end{equation}
with $k^2 \ = \ \frac\lambda{12\nu}$ and the corresponding forcing
\[f(x)=\Big(c+\frac{\lambda}3\Big)\sech^2(kx).\]

These two functions are exact solutions of the forced KdV equations
\[
 \partial_t u + c \partial_x u + \lambda u \partial_x u +\nu \partial_x^3 u =\frac{\dd}{\dd x} f(x).
\]
In the absence of non-trivial solutions of system~\eqref{eqn:Dfinal}, we use the same functions as reference. After an appropriate change of variables, $u_0$ and $u_1$ will satisfy system~\eqref{eqn:Dfinal} up to small terms. Therefore, adding the corresponding forcing term to the equations, $u_0$ and $u_1$ are exact solutions of the modified system, and we are able to compare the results of our numerical schemes with the theoretical solution. 

These comparisons are presented in Tables~\ref{tab:ErSchemas1} and~\ref{tab:ErSchemas2}. The error are givens in term of the normalized $l^2$ norm. The results exhibit a convergence behavior of order $\O((\Delta x)^2+(\Delta t)^2)$.

\begin{table}[!ht]
\centering {\begin{tabular}{@{}ccccccc@{}} \Hline
$\Delta x=\Delta t$&$L$& $T$  & KdV scheme & fully nonlinear scheme \\ \hline
 0.1\hphantom{0} & 20 & 10  & $8.1530\ 10^{-4}$  & $4.9498\ 10^{-4}$\\  
 0.05 & 20 & 10             & $2.0393\ 10^{-4}$  & $1.5154\ 10^{-4}$ \\  
 0.01 & 20 & 10             & $8.1604\ 10^{-6}$  & $1.8363\ 10^{-5}$ \\ \Hline 
 \end{tabular}}
 \caption{Numerical errors of the KdV and Green-Naghdi schemes for the initial data~\eqref{init1}. We choose the parameters $\mu=\epsilon_2=\alpha=0.1$, $\gamma=0.9$, $\delta=5/12$, $\Fr=1.1$, $\Bo=100$.}
\label{tab:ErSchemas1}
\end{table}
\begin{table}[!ht]
\centering {\begin{tabular}{@{}ccccccc@{}} \Hline
$\Delta x=\Delta t$&$L$& $T$  & KdV scheme  & fully nonlinear scheme \\ \hline
 0.1\hphantom{0} & 20 & 10  & $4.7397\ 10^{-4}$  & $1.7844\ 10^{-4}$ \\  
 0.05 & 20 & 10             & $1.1850\ 10^{-4}$  & $4.4557\ 10^{-5}$ \\  
 0.01 & 20 & 10             & $4.7409\ 10^{-6}$  & $8.1604\ 10^{-6}$\\ \Hline 
 \end{tabular}}
 \caption{Numerical errors of the KdV and Green-Naghdi schemes for the initial data~\eqref{init2}. We choose the parameters $\mu=\epsilon_2=\alpha=0.1$, $\gamma=0.9$, $\delta=5/12$, $\Fr=1.1$, $\Bo=100$.}
\label{tab:ErSchemas2}
\end{table}

\medskip

The test functions~\eqref{init1} and~\eqref{init2} are very smooth, and our problem can lead to rapid variation in the solution, as we can see especially in figures~\ref{fig:subcritical} and~\ref{fig:supercritical}. In order to check the accuracy of our method in such cases, we compute the KdV scheme with different space and time steps, and compare the outcome with the result of a more refined grid. More precisely, we compute the scheme for $\Delta x=\Delta t=0.1,\ 0.05,\ 0.01$, and measure the difference with a reference solution obtained with $\Delta x = \Delta t = 0.001$. The error is then discrete $l^2$ norm of the difference, normalized with the $l^2$ norm of the reference solution.

These comparisons, using the settings of figures~\ref{fig:subcritical} and~\ref{fig:supercritical}, are presented in Table~\ref{tab:ErSchemas3}. The normalized error can be relatively large; this is due to the fact that the reference solution is very small (of order $10^{-2}$). However, one clearly sees that the schemes are stable, and converge fast when space and time steps become small. Moreover, the precision for $\Delta x=\Delta t=0.01$ is sufficient to validate our results.

\begin{table}[!ht]
\centering {\begin{tabular}{@{}cccccc@{}} \Hline
$\Delta x=\Delta t$& \quad & Figure~\ref{fig:subcritical} (a) & Figure~\ref{fig:subcritical} (b) & Figure~\ref{fig:supercritical} (a) & Figure~\ref{fig:supercritical} (b) \\ \hline
 0.1\hphantom{0}& & $7.2320\ 10^{-1}$  & $1.2158$ & $2.4538\ 10^{-1}$ & $2.0070\ 10^{-1}$ \\  
 0.05           & & $1.9339\ 10^{-1}$  & $5.1843\ 10^{-1}$ & $9.6817\ 10^{-2}$& $1.0586\ 10^{-1}$\\  
 0.01           & & $9.3586\ 10^{-3}$  & $5.2038\ 10^{-2}$ & $1.9121\ 10^{-2}$ & $6.0949\ 10^{-2}$\\ \Hline 
 \end{tabular}}
 \caption{Convergence of the KdV scheme, with the settings of figures~\ref{fig:subcritical} and~\ref{fig:supercritical}.}
\label{tab:ErSchemas3}
\end{table}

\paragraph{Acknowledgments.}
This work was supported by the Agence Nationale de la Recherche (project ANR-08-BLAN-0301-01).

\bibliographystyle{abbrv}
\bibliography{DeadWaterbib}
\end{document}